\documentclass{amsart}
\usepackage{moreverb,url}
\usepackage{amssymb,amsmath,amsfonts}
\usepackage{bm}
\usepackage[font=scriptsize, labelfont=bf]{caption}
\usepackage{float}
\usepackage{varwidth}
\usepackage{booktabs}
\usepackage{afterpage}
\usepackage{mathrsfs}
\usepackage{subcaption}
\usepackage[squaren]{SIunits}
\usepackage{verbatim}
\usepackage[square,numbers]{natbib}
\usepackage{stmaryrd}
\usepackage{fancyhdr}
\usepackage{syntax,etoolbox}
\usepackage{graphicx}
\usepackage{grffile}
\usepackage{hyperref}
\hypersetup{colorlinks=true,allcolors=blue}
\usepackage[square,numbers]{natbib}
\usepackage[left=1in, right=1in, bottom=1in,top=1in]{geometry}
\usepackage{lineno}
\usepackage{syntax,etoolbox}
\usepackage{algorithmicx}
\usepackage{algorithm}
\usepackage{algpseudocode}
\usepackage{tikz}

\newtheorem{theorem}{Theorem}[section]
\newtheorem{lemma}{Lemma}[section]

\newtheorem{innercustomgeneric}{\customgenericname}
\providecommand{\customgenericname}{}
\newcommand{\newcustomproblem}[2]{%
	\newenvironment{#1}[1]
	{%
		\renewcommand\customgenericname{#2}%
		\renewcommand\theinnercustomgeneric{##1}%
		\innercustomgeneric
	}
	{\endinnercustomgeneric}
}

\newcustomproblem{customprob}{Problem}

\newcommand*{\bqed}{\hfill\ensuremath{\blacksquare}}%

\def\dd{\, \mathrm{d}}

\algblock{Input}{EndInput}

\algnotext{EndInput}

\algblock{Output}{EndOutput}

\algnotext{EndOutput}

\algblock{Iterate}{EndIterate}

\algnotext{EndIterate}

\allowdisplaybreaks

\begin{document}
	
	
	\title[Numerical methods for elliptic membrane shells]{Numerical approximation of the solution of an obstacle problem modelling the displacement of elliptic membrane shells via the penalty method}
	
	
	\author[Aaron Meixner]{Aaron Meixner}
	\address{Department of Mathematics The Ohio State University, 100 Math Tower, 231 West 18th Avenue, Columbus, Ohio, USA}
	\email{meixner.8@buckeyemail.osu.edu}
	
	\author[Paolo Piersanti]{Paolo Piersanti}
	\address{Department of Mathematics and Institute for Scientific Computing and Applied Mathematics, Indiana University Bloomington, 729 East Third Street, Bloomington, Indiana, USA}
	\email[Corresponding author]{ppiersan@iu.edu}
	
\today

\begin{abstract}
In this paper we establish the convergence of a numerical scheme based, on the Finite Element Method, for a time-independent problem modelling the deformation of a linearly elastic elliptic membrane shell subjected to remaining confined in a half space. Instead of approximating the original variational inequalities governing this obstacle problem, we approximate the penalized version of the problem under consideration.
A suitable coupling between the penalty parameter and the mesh size will then lead us to establish the convergence of the solution of the discrete penalized problem to the solution of the original variational inequalities.

We also establish the convergence of the Brezis-Sibony scheme for the problem under consideration. Thanks to this iterative method, we can approximate the solution of the discrete penalized problem without having to resort to nonlinear optimization tools.

Finally, we present numerical simulations validating our new theoretical results.

\smallskip

\noindent \textbf{Keywords.} Obstacle problems $\cdot$ Variational Inequalities $\cdot$ Elasticity theory $\cdot$ Finite Difference Quotients $\cdot$ Penalty Method $\cdot$ Finite Element Method
\end{abstract}

\maketitle

\tableofcontents

\section{Introduction}
\label{sec0}
In this paper we establish the convergence of a numerical scheme, based on the Finite Element Method, for approximating the solution of a set of variational inequalities modelling the displacement of a linearly elastic elliptic membrane shell subject to remaining confined in a prescribed half space.

Differently from the numerical scheme presented in~\cite{PS}, where the authors studied the convergence of a numerical scheme based on the Finite Element Method for approximating the solution of a fourth order set of variational inequalities modelling the displacement of a shallow shell which, we recall, takes the form of a Kirchhoff-Love vector field, the solution of the problem we are studying in this paper is a vector field and the variational inequalities we shall be considering involve all the three components of one such displacement vector field.

Critical to establishing the convergence of the finite element approximation of the solution of the problem under consideration is the augmentation of regularity of the solution of the governing variational inequalities.
This preparatory result improves the standard penalization argument extensively discussed in~\cite{Lions1969} and lets us infer \emph{how fast} the penalized solution converges to the solution of the original variational inequalities.

A similar numerical analysis has been treated by Scholz in the paper~\cite{Scholz1984} where, however, the author resorted to the very peculiar assumption $(\ast)$ on the elliptic operator under consideration.
We will replace this assumption by a more reasonable geometrical assumption, which is exactly the assumption needed to ensure the ``density property'' devised by Ciarlet, Mardare \& Piersanti in~\cite{CiaMarPie2018b,CiaMarPie2018}. In addition to this, the augmentation of regularity argument carried out in~\cite{Scholz1984} is only valid for scalar functions. The fact that the solution of the variational problem under consideration is a vector field renders this analysis substantially more complicated than in the scalar case.

Other references about numerical approximations of the solutions of obstacle problems via the Finite Element Method are, for instance, the seminal paper by Falk~\cite{Falk1974}, where the author exploited the augmentation of regularity result established by Brezis and Stampacchia~\cite{BrezisStampacchia1968}. The scheme there proposed, however, seems not to be reproducible in the case where the unknown of the variational problem under consideration is a vector field.

The study of the augmentation of regularity of solutions for boundary value problems modelled via elliptic equations began between the end of the Fifties and the early Sixties, when Agmon, Douglis \& Nirenberg published the two pioneering papers~\cite{AgmDouNir1959} and~\cite{AgmDouNir1964} about the regularity properties of solutions of elliptic systems up to the boundary of the integration domain.

The augmentation of regularity for solutions of variational inequalities for scalar functions was first addressed by Frehse in the early Seventies~\cite{Frehse1971,Frehse1973}. In the late Seventies and early Eighties, Caffarelli and his collaborators published the two papers~\cite{Caffarelli1979,CafFriTor1982}, where they proved that the solution of an obstacle problem for the biharmonic operator (cf., e.g., Section~6.7 of~\cite{PGCLNFAA}) could not be \emph{too regular}. It was recently established in~\cite{Pie2020-1} that the solution of an obstacle problem for linearly elastic shallow shells enjoys higher regularity properties in the interior of the domain where it is defined. To our best knowledge, the results contained in~\cite{Pie2020-1} constitute the first attempt where the augmented regularity of a vector field solving a set of variational inequalities is studied.

Augmentation of regularity for linear problems in elasticity theory was treated, for instance, by Geyomonat in the seminal paper~\cite{Geymonat1965}, by Alexandrescu-Iosifescu~\cite{Iosifescu1994}, where the augmentation of regularity for Koiter's model is considered, and by Genevey in~\cite{Genevey1996}, where the higher regularity of the solution for a variational problem modelling the displacement of a linearly elastic elliptic membrane shell is established.

To our best knowledge, the only record in the literature treating the augmentation of regularity of the solution of second order variational inequalities in the case where one such solution is a vector field and the constraint defining the non-empty, closed, and convex subset of the Sobolev space where the solution is sought is expressed in terms of all of the three components of the displacement vector field is the recent paper~\cite{Pie-2022-interior}.

This paper is divided into ten sections (including this one). In section~\ref{sec1} we present some background and notation.

In section~\ref{sec2} we recall the formulation and the properties of a three-dimensional obstacle problem for a ``general'' linearly elastic shell. It is worth mentioning that this three-dimensional problem is the starting point for deriving the variational formulation of the two-dimensional problem, whose solution regularity is the object of interest of this paper.

In section~\ref{sec3} we scale the original three-dimensional problem over a domain of fixed thickness and we state the corresponding scaled problem, modelled by a set of variational inequalities. We then recall the result of the asymptotic analysis conducted in~\cite{CiaMarPie2018b,CiaMarPie2018}, we state the two-dimensional limit problem obtained as a result of an application of the ``density property'' and, finally, we de-scale the limit problem by re-introducing the thickness parameter. 

In section~\ref{sec:penalty}, we establish the existence and uniqueness of the solution for the de-scaled penalized limit problem, after recalling the regularity properties of the penalty operator entering the model under consideration.

In section~\ref{sec:aug-interior}, we establish the augmentation of regularity up to the boundary of the de-scaled penalized problem. As a consequence of this, we are able to prove that the solution of the de-scaled variational inequalities is actually the weak limit of the sequence of solutions of the de-scaled penalized problems as the penalty parameter tends to zero with respect to a vector space which is is characterized by a higher regularity than the one where the search for minimizers of the energy functional was originally performed.

In section~\ref{approx:original} we show that the sequence of solutions of the de-scaled penalized problems converges to the solution of the de-scaled variational inequalities at a polynomial rate. To obtain this result, the augmentation of regularity devised in section~\ref{sec:aug-interior} will be playing a crucial role.

In section~\ref{approx:penalty} we approximate the solution of the de-scaled penalized problem by a Finite Element Method, the convergence of which shall strongly be hinging on a suitable coupling between the penalty parameter and the mesh size.

In section~\ref{approx:BrezisSibony}, we prove that the iterative scheme originally proposed by Brezis and Sibony in the seminal paper~\cite{BrezisSibony1968} makes possible to approximate the solution of the discrete penalized problem introduced in section~\ref{approx:penalty} without having to resort to nonlinear optimization tools like, for instance, the Primal-Dual Active Set Method and the Gradient Descent Method.

Finally, in section~\ref{numerics} we present numerical experiments meant to validate our theoretical results.

\section{Background and notation}
\label{sec1}

For a complete overview about the classical notions of differential geometry used in this paper, see, e.g.~\cite{Ciarlet2000} or \cite{Ciarlet2005}.

Greek indices, except $\varepsilon$, take their values in the set $\{1,2\}$, while Latin indices, except when they are used for ordering sequences, take their values in the set $\{1,2,3\}$, and, unless differently specified, the summation convention with respect to repeated indices is used jointly with these two rules. 
As a model of the three-dimensional ``physical'' space $\mathbb{R}^3$, we take a \emph{real three-dimensional affine Euclidean space}, i.e., a set in which a point $O \in\mathbb{R}^3$ has been chosen as the \emph{origin} and with which a \emph{real three-dimensional Euclidean space}, denoted $\mathbb{E}^3$, is associated. We equip $\mathbb{E}^3$ with an \emph{orthonormal basis} consisting of three vectors $\bm{e}^i$, with components $e^i_j=\delta^i_j$. 

The definition of $\mathbb{R}^3$ as an affine Euclidean space means that with any point $x \in \mathbb R^3$ is associated an uniquely determined vector $\boldsymbol{Ox} \in \mathbb{E}^3$. The origin $O \in \mathbb{R}^3$ and the orthonormal vectors $\bm{e}^i \in \mathbb{E}^3$ together constitute a \emph{Cartesian frame} in $\mathbb{R}^3$ and the three components $x_i$ of the vector $\boldsymbol{Ox}$ over the basis formed by $\bm{e}^i$ are called the \emph{Cartesian coordinates} of $x \in \mathbb{R}^3$, or the \emph{Cartesian components} of $\boldsymbol{Ox} \in \mathbb{E}^3$. Once a Cartesian frame has been chosen, any point $x \in \mathbb{R}^3$ may be thus \emph{identified} with the vector $\boldsymbol{Ox}=x_i \bm{e}^i \in \mathbb{E}^3$. As a result, a set in $\mathbb{R}^3$ can be identified with a ``physical'' body in the Euclidean space $\mathbb{E}^3$.
The Euclidean inner product and the vector product of $\bm{u}, \bm{v} \in \mathbb{E}^3$ are respectively denoted by $\bm{u} \cdot \bm{v}$ and $\bm{u} \wedge \bm{v}$; the Euclidean norm of $\bm{u} \in \mathbb{E}^3$ is denoted by $\left|\bm{u} \right|$. The notation $\delta^j_i$ designates the Kronecker symbol.

Given an open subset $\Omega$ of $\mathbb{R}^n$, where $n \ge 1$, we denote the usual Lebesgue and Sobolev spaces by $L^2(\Omega)$, $L^1_{\textup{loc}}(\Omega)$, $H^1(\Omega)$, $H^1_0 (\Omega)$, $H^1_{\textup{loc}}(\Omega)$, and the notation $\mathcal{D} (\Omega)$ designates the space of all functions that are infinitely differentiable over $\Omega$ and have compact supports in $\Omega$. We denote $\left\| \cdot \right\|_X$ the norm in a normed vector space $X$. Spaces of vector-valued functions are written in boldface.
The Euclidean norm of any point $x \in \Omega$ is denoted by $|x|$. 

The boundary $\Gamma$ of an open subset $\Omega$ in $\mathbb{R}^n$ is said to be Lipschitz-continuous if the following conditions are satisfied (cf., e.g., Section~1.18 of~\cite{PGCLNFAA}): Given an integer $s\ge 1$, there exist constants $\alpha_1>0$ and $L>0$, a finite number of local coordinate systems, with coordinates 
$$
\bm{\phi}'_r=(\phi_1^r, \dots, \phi_{n-1}^r) \in \mathbb{R}^{n-1} \textup{ and } \phi_r=\phi_n^r, 1 \le r \le s,
$$ 
sets
$$
\tilde{\omega}_r:=\{\bm{\phi}_r \in\mathbb{R}^{n-1}; |\bm{\phi}_r|<\alpha_1\},\quad 1 \le r \le s,
$$
and corresponding functions
$$
\tilde{\theta}_r:\tilde{\omega}_r \to\mathbb{R},\quad 1 \le r \le s,
$$
such that
$$
\Gamma=\bigcup_{r=1}^s \{(\bm{\phi}'_r,\phi_r); \bm{\phi}'_r \in \tilde{\omega}_r \textup{ and }\phi_r=\tilde{\theta}_r(\bm{\phi}'_r)\},
$$
and 
$$
|\tilde{\theta}_r(\bm{\phi}'_r)-\tilde{\theta}_r(\bm{\upsilon}'_r)|\le L |\bm{\phi}'_r-\bm{\upsilon}'_r|, \textup{ for all }\bm{\phi}'_r, \bm{\upsilon}'_r \in \tilde{\omega}_r, \textup{ and all }1\le r\le s.
$$

We observe that the second last formula takes into account overlapping local charts, while the last set of inequalities expresses the Lipschitz continuity of the mappings $\tilde{\theta}_r$.

An open set $\Omega$ is said to be locally on the same side of its boundary $\Gamma$ if, in addition, there exists a constant $\alpha_2>0$ such that
\begin{align*}
	\{(\bm{\phi}'_r,\phi_r);\bm{\phi}'_r \in\tilde{\omega}_r \textup{ and }\tilde{\theta}_r(\bm{\phi}'_r) < \phi_r < \tilde{\theta}_r(\bm{\phi}'_r)+\alpha_2\}&\subset \Omega,\quad\textup{ for all } 1\le r\le s,\\
	\{(\bm{\phi}'_r,\phi_r);\bm{\phi}'_r \in\tilde{\omega}_r \textup{ and }\tilde{\theta}_r(\bm{\phi}'_r)-\alpha_2 < \phi_r < \tilde{\theta}_r(\bm{\phi}'_r)\}&\subset \mathbb{R}^n\setminus\overline{\Omega},\quad\textup{ for all } 1\le r\le s.
\end{align*}

A \emph{domain in} $\mathbb{R}^n$ is a bounded and connected open subset $\Omega$ of $\mathbb{R}^n$, whose boundary $\partial \Omega$ is Lipschitz-continuous, the set $\Omega$ being locally on a single side of $\partial \Omega$.

Let $\omega$ be a domain in $\mathbb{R}^2$ with boundary $\gamma:=\partial\omega$,  and let $\omega_1 \subset \omega$. The special notation $\omega_1 \subset \subset \omega$ means that $\overline{\omega_1} \subset \omega$ and $\textup{dist}(\gamma,\partial\omega_1):=\min\{|x-y|;x \in \gamma \textup{ and } y \in \partial\omega_1\}>0$.
Let $y = (y_\alpha)$ denote a generic point in $\omega$, and let $\partial_\alpha := \partial / \partial y_\alpha$. A mapping $\bm{\theta} \in \mathcal{C}^1(\overline{\omega}; \mathbb{E}^3)$ is said to be an \emph{immersion} if the two vectors
$$
\bm{a}_\alpha (y) := \partial_\alpha \bm{\theta} (y)
$$
are linearly independent at each point $y \in \overline{\omega}$. Then the set $\bm{\theta} (\overline{\omega})$ is a \emph{surface in} $\mathbb{E}^3$, equipped with $y_1, y_2$ as its \emph{curvilinear coordinates}. Given any point $y\in \overline{\omega}$, the linear combinations of the vectors $\bm{a}_\alpha (y)$ span the \emph{tangent plane} to the surface $\bm{\theta} (\overline{\omega})$ at the point $\bm{\theta} (y)$, the unit vector
$$
\bm{a}_3 (y) := \frac{\bm{a}_1(y) \wedge \bm{a}_2 (y)}{|\bm{a}_1(y) \wedge \bm{a}_2 (y)|}
$$
is orthogonal to $\bm{\theta} (\overline{\omega})$ at the point $\bm{\theta} (y)$, the three vectors $\bm{a}_i(y)$ form the \emph{covariant} basis at the point $\bm{\theta} (y)$, and the three vectors $\bm{a}^j(y)$ defined by the relations
$$
\bm{a}^j(y) \cdot \bm{a}_i (y) = \delta^j_i
$$
form the \emph{contravariant} basis at $\bm{\theta} (y)$; note that the vectors $\bm{a}^\beta (y)$ also span the tangent plane to $\bm{\theta} (\overline{\omega})$ at $\bm{\theta} (y)$ and that $\bm{a}^3(y) = \bm{a}_3 (y)$.

The \emph{first fundamental form} of the surface $\bm{\theta} (\overline{\omega})$ is then defined by means of its \emph{covariant components}
$$
a_{\alpha \beta} := \bm{a}_\alpha \cdot \bm{a}_\beta = a_{\beta \alpha} \in \mathcal{C}^0 (\overline{\omega}),
$$
or by means of its \emph{contravariant components}
$$
a^{\alpha \beta}:= \bm{a}^\alpha \cdot \bm{a}^\beta = a^{\beta \alpha}\in \mathcal{C}^0(\overline{\omega}).
$$

Note that the symmetric matrix field $(a^{\alpha \beta})$ is then the inverse of the positive-definite matrix field $(a_{\alpha \beta})$, that $\bm{a}^\beta = a^{\alpha \beta}\bm{a}_\alpha$ and $\bm{a}_\alpha = a_{\alpha \beta} \bm{a}^\beta$, and that the \emph{area element} along $\bm{\theta} (\overline{\omega})$ is given at each point $\bm{\theta}(y), y \in \overline{\omega}$, by $\sqrt{a(y)} \dd y$, where
$$
a := \det(a_{\alpha \beta}) \in \mathcal{C}^0(\overline{\omega}), 
$$
and satisfies $a_0 \le a(y) \le a_1$, for all $y\in\overline{\omega}$ for some $a_0, a_1>0$.

Given an immersion $\bm{\theta} \in \mathcal{C}^2(\overline{\omega};\mathbb{E}^3)$, the \emph{second fundamental form} of the surface $\bm{\theta}(\overline{\omega})$ is defined by means of its \emph{covariant components}
$$
b_{\alpha \beta}:= \partial_\alpha \bm{a}_\beta \cdot \bm{a}_3 = -\bm{a}_\beta \cdot \partial_\alpha \bm{a}_3 = b_{\beta \alpha} \in \mathcal{C}^0(\overline{\omega}),
$$
or by means of its \emph{mixed components}
$$
b^\beta_\alpha := a^{\beta \sigma} b_{\alpha \sigma} \in \mathcal{C}^0(\overline{\omega}),
$$
and the \emph{Christoffel symbols} associated with the immersion $\bm{\theta}$ are defined by
$$
\Gamma^\sigma_{\alpha \beta}:= \partial_\alpha \bm{a}_\beta \cdot \bm{a}^\sigma = \Gamma^\sigma_{\beta \alpha} \in \mathcal{C}^0 (\overline{\omega}).
$$

The \emph{Gaussian curvature} at each point $\bm{\theta} (y) , \, y \in \overline{\omega}$, of the surface $\bm{\theta} (\overline{\omega})$ is defined by
$$
\kappa (y) := \frac{\det (b_{\alpha \beta} (y))}{\det (a_{\alpha \beta} (y))} = \det \left( b^\beta_\alpha (y)\right).
$$

Observe that the denominator in the above relation does not vanish since $\bm{\theta}$ is assumed to be an immersion. Note that the Gaussian curvature $\kappa (y)$ at the point $\bm{\theta} (y)$ is also equal to the product of the two principal curvatures at this point.

Given an immersion
$\bm{\theta} \in \mathcal{C}^2 (\overline{\omega}; \mathbb{E}^3)$ and a
vector field $\bm{\eta} = (\eta_i) \in \mathcal{C}^1(\overline{\omega};
\mathbb{R}^3)$, the vector field
$$
\tilde{\bm{\eta}} := \eta_i \bm{a}^i
$$
may be viewed as the \emph{displacement field of the surface} $\bm{\theta} (\overline{\omega})$, thus defined by means of its \emph{covariant components} $\eta_i$ over the vectors $\bm{a}^i$ of the contravariant bases along the surface. If the norms $\left\| \eta_i \right\|_{\mathcal{C}^1(\overline{\omega})}$ are small enough, the mapping $(\bm{\theta} + \eta_i \bm{a}^i) \in \mathcal{C}^1(\overline{\omega}; \mathbb{E}^3)$ is also an immersion, so that the set $(\bm{\theta} + \eta_i \bm{a}^i) (\overline{\omega})$ is again a surface in $\mathbb{E}^3$, equipped with the same curvilinear coordinates as those of the surface $\bm{\theta} (\overline{\omega})$ and is called the \emph{deformed surface} corresponding to the displacement field $\tilde{\bm{\eta}} = \eta_i \bm{a}^i$. 

It is thus possible to define the first fundamental form of the deformed surface in terms of its covariant components by
\begin{align*}
	a_{\alpha \beta} (\bm{\eta}) :=& (\bm{a}_\alpha + \partial_\alpha \tilde{\bm{\eta}}) \cdot (\bm{a}_\beta + \partial_\beta \tilde{\bm{\eta}}) \\
	=& a_{\alpha \beta} + \bm{a}_\alpha \cdot \partial_\beta \tilde{\bm{\eta}} + \partial_\alpha \tilde{\bm{\eta}} \cdot \bm{a}_\beta + \partial_\alpha \tilde{\bm{\eta}} \cdot \partial_\beta \tilde{\bm{\eta}}.
\end{align*}

The \emph{linear part with  respect to} $\tilde{\bm{\eta}}$ in the difference $(a_{\alpha \beta}(\bm{\eta}) - a_{\alpha \beta})/2$ is called the \emph{linearized change of metric}, or \emph{strain}, \emph{tensor} associated with the displacement field $\eta_i \bm{a}^i$, the covariant components of which are thus defined by
$$
\gamma_{\alpha \beta}(\bm{\eta}) := \dfrac{1}{2}(\bm{a}_\alpha \cdot \partial_\beta \tilde{\bm{\eta}} + \partial_\alpha \tilde{\bm{\eta}} \cdot \bm{a}_\beta) = \frac12 (\partial_\beta \eta_\alpha + \partial_\alpha \eta_\beta ) - \Gamma^\sigma_{\alpha \beta} \eta_\sigma - b_{\alpha \beta} \eta_3 = \gamma_{\beta \alpha} (\bm{\eta}).
$$

In this paper, we shall consider a specific class of surfaces, according to the following definition: Let $\omega$ be a domain in $\mathbb{R}^2$. Then a surface $\bm{\theta} (\overline{\omega})$ defined by means of an immersion $\bm{\theta} \in \mathcal{C}^2(\overline{\omega};\mathbb{E}^3)$ is said to be \emph{elliptic} if its Gaussian curvature $K$ is everywhere strictly positive in $\overline{\omega}$, or equivalently, if there exists a constant $K_0$ such that:
$$
0 < K_0 \le K (y), \text{ for all } y \in \overline{\omega}.
$$

It turns out that, when an \emph{elliptic surface} is subjected to a displacement field $\eta_i \bm{a}^i$ whose \emph{tangential covariant components $\eta_\alpha$ vanish on the entire boundary of the domain $\omega$}, the following inequality holds. Note that the components of the displacement fields and linearized change of metric tensors appearing in the next theorem are no longer assumed to be continuously differentiable functions; they are instead to be understood in a generalised sense, since they now belong to \emph{ad hoc} Lebesgue or Sobolev spaces.

\begin{theorem} 
	\label{korn}
	Let $\omega$ be a domain in $\mathbb{R}^2$ and let an immersion $\bm{\theta} \in \mathcal{C}^3 (\overline{\omega}; \mathbb{E}^3)$ be given such that the surface $\bm{\theta}(\overline{\omega})$ is elliptic. Define the space
	$$
	\bm{V}_M (\omega) := H^1_0 (\omega) \times H^1_0 (\omega) \times L^2(\omega).
	$$
	
	Then, there exists a constant $c_0=c_0(\omega,\bm{\theta})>0$ such that
	$$
	\left\{ \sum_\alpha \left\| \eta_\alpha \right\|^2_{H^1(\omega)} + \left\| \eta_3 \right\|^2_{L^2(\omega)}\right\}^{1/2} \le c_0 \left\{ \sum_{\alpha, \beta} \left\| \gamma_{\alpha \beta}(\bm{\eta}) \right\|_{L^2(\omega)}^2\right\}^{1/2}
	$$
	for all $\bm{\eta}= (\eta_i) \in \bm{V}_M (\omega)$.
	\qed
\end{theorem}

The above inequality, which is due to \cite{CiaLods1996a} and \cite{CiaSanPan1996} (see also Theorem~2.7-3 of~\cite{Ciarlet2000}),  constitutes an example of a \emph{Korn's inequality on a surface}, in the sense that it provides an estimate of an appropriate norm of a displacement field defined on a surface in terms of an appropriate norm of a specific ``measure of strain'' (here, the linearized change of metric tensor) corresponding to the displacement field under consideration.

\section{The three-dimensional obstacle problem for a ``general'' linearly elastic shell} \label{Sec:2}
\label{sec2}

Let $\omega$ be a domain in $\mathbb{R}^2$, let $\gamma:= \partial \omega$, and let $\gamma_0$ be a non-empty relatively open subset of $\gamma$. For each $\varepsilon > 0$, we define the sets
$$
\Omega^\varepsilon = \omega \times \left] - \varepsilon , \varepsilon \right[ \text{ and } \Gamma^\varepsilon_0 := \gamma_0 \times \left] - \varepsilon , \varepsilon \right[,
$$
we let $x^\varepsilon = (x^\varepsilon_i)$ designate a generic point in the set $\overline{\Omega^\varepsilon}$, and we let $\partial^\varepsilon_i := \partial / \partial x^\varepsilon_i$. Hence we also have $x^\varepsilon_\alpha = y_\alpha$ and $\partial^\varepsilon_\alpha = \partial_\alpha$.

Given an immersion $\bm{\theta} \in \mathcal{C}^3(\overline{\omega}; \mathbb{E}^3)$ and $\varepsilon > 0$, consider a \emph{shell} with \emph{middle surface} $\bm{\theta} (\overline{\omega})$ and with \emph{constant thickness} $2 \varepsilon$. This means that the \emph{reference configuration} of the shell is the set $\bm{\Theta} (\overline{\Omega^\varepsilon})$, where the mapping $\bm{\Theta} : \overline{\Omega^\varepsilon} \to \mathbb{E}^3$ is defined by
$$
\bm{\Theta} (x^\varepsilon) := \bm{\theta} (y) + x^\varepsilon_3 \bm{a}^3(y) \text{ at each point } x^\varepsilon = (y, x^\varepsilon_3) \in \overline{\Omega^\varepsilon}.
$$

One can then show (cf., e.g., Theorem~3.1-1 of~\cite{Ciarlet2000}) that, if $\varepsilon > 0$ is small enough, such a mapping $\bm{\Theta} \in \mathcal{C}^2(\overline{\Omega^\varepsilon}; \mathbb{E}^3)$ is an \emph{immersion}, in the sense that the three vectors
$$
\bm{g}^\varepsilon_i (x^\varepsilon) := \partial^\varepsilon_i \bm{\Theta} (x^\varepsilon)
$$
are linearly independent at each point $x^\varepsilon \in \overline{\Omega^\varepsilon}$; these vectors  then constitute the \emph{covariant basis} at the point $\bm{\Theta} (x^\varepsilon)$, while the three vectors $\bm{g}^{j, \varepsilon} (x^\varepsilon)$ defined by the relations
$$
\bm{g}^{j, \varepsilon} (x^\varepsilon) \cdot \bm{g}^\varepsilon_i (x^\varepsilon) = \delta^j_i
$$
constitute the \emph{contravariant basis} at the same point. It will be implicitly assumed in the sequel that $\varepsilon > 0$ \emph{is small enough so that $\bm{\Theta} : \overline{\Omega^\varepsilon} \to \mathbb{E}^3$} is an \emph{immersion}.

One then defines the \emph{metric tensor associated with the immersion} $\bm{\Theta}$ by means of its \emph{covariant components}
$$
g^\varepsilon_{ij}:= \bm{g}^\varepsilon_i \cdot \bm{g}^\varepsilon_j \in \mathcal{C}^1(\overline{\Omega^\varepsilon}),
$$
or by means of its \emph{contravariant components}
$$
g^{ij, \varepsilon} := \bm{g}^{i, \varepsilon} \cdot \bm{g}^{j,\varepsilon} \in \mathcal{C}^1(\overline{\Omega^\varepsilon}).
$$

Note that the symmetric matrix field $(g^{ij, \varepsilon})$ is then the inverse of the positive-definite matrix field $(g^\varepsilon_{ij})$, that $\bm{g}^{j, \varepsilon} = g^{ij, \varepsilon} \bm{g}^\varepsilon_i$ and $\bm{g}^\varepsilon_i = g^\varepsilon_{ij} \bm{g}^{j, \varepsilon}$, and that the \emph{volume element} in $\bm{\Theta}(\overline{\Omega^\varepsilon})$ is given at each point $\bm{\Theta} (x^\varepsilon)$, $x^\varepsilon \in \overline{\Omega^\varepsilon}$, by $\sqrt{g^\varepsilon (x^\varepsilon)} \dd x^\varepsilon$, where
$$
g^\varepsilon := \det (g^\varepsilon_{ij}) \in \mathcal{C}^1(\overline{\Omega^\varepsilon}).
$$

One also defines the \emph{Christoffel symbols} associated with the immersion $\bm{\Theta}$ by
$$
\Gamma^{p, \varepsilon}_{ij}:= \partial_i \bm{g}^\varepsilon_j \cdot \bm{g}^{p, \varepsilon} = \Gamma^{p, \varepsilon}_{ji} \in \mathcal{C}^0(\overline{\Omega^\varepsilon}).
$$
Note that $\Gamma^{3,\varepsilon}_{\alpha 3} = \Gamma^{p, \varepsilon}_{33} = 0$.

Given a vector field $\bm{v}^\varepsilon = (v^\varepsilon_i) \in \mathcal{C}^1 (\overline{\Omega^\varepsilon}; \mathbb{R}^3)$, the associated vector field
$$
\tilde{\bm{v}}^\varepsilon := v^\varepsilon_i \bm{g}^{i, \varepsilon}
$$
can be viewed as a \emph{displacement field} of the reference configuration $\bm{\Theta} (\overline{\Omega^\varepsilon})$ of the shell, thus defined by means of its \emph{covariant components} $v^ \varepsilon_i$ over the vectors $\bm{g}^{i, \varepsilon}$ of the contravariant bases in the reference configuration.

If the norms $\left\| v^\varepsilon_i \right\|_{\mathcal{C}^1 (\overline{\Omega^\varepsilon})}$ are small enough, the mapping $(\bm{\Theta} + v^\varepsilon_i \bm{g}^{i, \varepsilon})$ is also an immersion, so that one can also define the metric tensor of the \emph{deformed configuration} $(\bm{\Theta} + v^\varepsilon_i \bm{g}^{i, \varepsilon})(\overline{\Omega^\varepsilon})$ by means of its covariant components
\begin{align*}
	g^\varepsilon_{ij} (v^\varepsilon) :=& (\bm{g}^\varepsilon_i + \partial^\varepsilon_i \tilde{\bm{v}}^\varepsilon ) \cdot (\bm{g}^\varepsilon_j + \partial^\varepsilon_j \tilde{\bm{v}}^\varepsilon) \\
	=& g^\varepsilon_{ij} + \bm{g}^\varepsilon_i \cdot \partial_j \tilde{\bm{v}}^\varepsilon + \partial^\varepsilon_i \tilde{\bm{v}}^\varepsilon \cdot \bm{g}^\varepsilon_j + \partial_i \tilde{\bm{v}}^\varepsilon \cdot \partial_j \tilde{\bm{v}}^\varepsilon.
\end{align*}

The linear part with respect to $\tilde{\bm{v}}^\varepsilon$ in the difference $(g^\varepsilon_{ij} (\bm{v}^\varepsilon) - g^\varepsilon_{ij})/2$ is then called the \emph{linearized strain tensor} associated with the displacement field $v^\varepsilon_i \bm{g}^{i, \varepsilon}$, the covariant components of which are thus defined by
$$
e^\varepsilon_{i\|j} (\bm{v}^\varepsilon) := \frac12 \left( \bm{g}^\varepsilon_i \cdot \partial_j \tilde{\bm{v}}^\varepsilon + \partial^\varepsilon_i \tilde{\bm{v}}^\varepsilon \cdot \bm{g}^\varepsilon_j \right) = \frac12 (\partial^\varepsilon_j v^\varepsilon_i + \partial^\varepsilon_i v^\varepsilon_j) - \Gamma^{p, \varepsilon}_{ij} v^\varepsilon_p = e_{j\|i}^\varepsilon (\bm{v}^\varepsilon).
$$

The functions $e^\varepsilon_{i\|j} (\bm{v}^\varepsilon)$ are called the \emph{linearized strains in curvilinear coordinates} associated with the displacement field $v^\varepsilon_i \bm{g}^{i, \varepsilon}$.

We assume throughout this paper that, for each $\varepsilon > 0$, the reference configuration $\bm{\Theta} (\overline{\Omega^\varepsilon})$ of the shell is a \emph{natural state} (i.e., stress-free) and that the material constituting the shell is \emph{homogeneous}, \emph{isotropic}, and \emph{linearly elastic}. The behavior of such an elastic material is thus entirely governed by its two \emph{Lam\'{e} constants} $\lambda \geq 0$ and $\mu > 0$ (for details, see, e.g., Section~3.8 of~\cite{Ciarlet1988}).

We will also assume that the shell is subjected to \emph{applied body forces} whose density per unit volume is defined by means of its covariant components $f^{i, \varepsilon} \in L^2(\Omega^\varepsilon)$, and to a \emph{homogeneous boundary condition of place} along the portion $\Gamma^\varepsilon_0$ of its lateral face (i.e., the displacement vanishes on $\Gamma^\varepsilon_0$).

In this paper, we consider a specific \emph{obstacle problem} for such a shell, in the sense that the shell is also subjected to a \emph{confinement condition}, expressing that any \emph{admissible} deformed configuration remains in a \emph{half-space} of the form
$$
\mathbb{H} := \{\bm{Ox} \in \mathbb{E}^3; \, \bm{Ox} \cdot \bm{q} \geq 0\},
$$
where $\bm{q} \in\mathbb{E}^3$ is a \emph{non-zero vector} given once and for all. In other words, any admissible displacement field must satisfy
$$
\left(\bm{\Theta} (x^\varepsilon) + v^\varepsilon_i (x^\varepsilon) \bm{g}^{i,\varepsilon} (x^\varepsilon)\right) \cdot \bm{q} \ge 0
$$
for all $x^\varepsilon \in \overline{\Omega^\varepsilon}$, or possibly only for almost all (a.a. in what follows) $x^\varepsilon \in \Omega^\varepsilon$ when the covariant components $v^\varepsilon_i$ are required to belong to the Sobolev space $H^1(\Omega^\varepsilon)$ as in Theorem \ref{t:2} below.

We will of course assume that the reference configuration satisfies the confinement condition, i.e., that
$$
\bm{\Theta} (\overline{\Omega^\varepsilon}) \subset \mathbb{H}.
$$

It is to be emphasized that the above confinement condition \emph{considerably departs} from the usual \emph{Signorini condition} favoured by most authors, who usually require that only the points of the undeformed and deformed ``lower face'' $\omega \times \{-\varepsilon\}$ of the reference configuration satisfy the confinement condition (see, e.g., \cite{Leger2008}, \cite{LegMia2018}, \cite{MezChaBen2020}, \cite{Rodri2018}). Clearly, the confinement condition considered in the present paper is more physically realistic, since a Signorini condition imposed only on the lower face of the reference configuration does not prevent -- at least ``mathematically'' -- other points of the deformed reference configuration to ``cross'' the plane  $\{\bm{Ox}\in \mathbb{E}^3; \; \bm{Ox} \cdot \bm{q} = 0\}$ and then to end up on the ``other side'' of this plane. The mathematical models characterized by the confinement condition introduced beforehand, confinement condition which is also considered in the seminal paper~\cite{Leger2008} in a different geometrical framework, do not take any traction forces into account. Indeed, by Classical Mechanics, there could be no traction forces applied to the portion of the three-dimensional shell boundary that engages contact with the obstacle. In the same spirit as~\cite{Pie2023}, friction is not considered in the context of this analysis.

Unlike the classical \emph{Signorni condition}, the confinement condition here considered is more suitable in the context of multi-scales multi-bodies problems like, for instance, the study of the motion of the human heart valves, conducted by Quarteroni and his associates in~\cite{Quarteroni2021-3,Quarteroni2021-2,Quarteroni2021-1} and the references therein.

Such a confinement condition renders the study of this problem considerably more difficult, however, as the constraint now bears on a vector field, the displacement vector field of the reference configuration, instead of on only a single component of this field.

The mathematical modelling of such an \emph{obstacle problem for a linearly elastic shell} is then clear; since, \emph{apart from} the confinement condition, the rest, i.e., the \emph{function space} and the expression of the quadratic \emph{energy} $J^\varepsilon$, is classical (viz.~\cite{Ciarlet2000}). More specifically, let
$$
A^{ijk\ell, \varepsilon} := \lambda g^{ij, \varepsilon} g^{k\ell, \varepsilon} + \mu \left( g^{ik, \varepsilon} g^{j\ell, \varepsilon} + g^{i\ell, \varepsilon} g^{jk, \varepsilon} \right) =
A^{jik\ell, \varepsilon} =  A^{k\ell ij, \varepsilon} 
$$
denote the contravariant components of the \emph{elasticity tensor} of the elastic material constituting the shell. Then the unknown of the problem, which is the vector field $\bm{u}^\varepsilon = (u^\varepsilon_i)$ where the functions $u^\varepsilon_i : \overline{\Omega^\varepsilon} \to \mathbb{R}$ are the three covariant components of the unknown ``three-dimensional'' displacement vector field $u^\varepsilon_i \bm{g}^{i, \varepsilon}$ of the reference configuration of the shell, should minimize the energy $J^\varepsilon : \bm{H}^1(\Omega^\varepsilon) \to \mathbb{R}$ defined by
$$
J^\varepsilon (\bm{v}^\varepsilon) := \frac12 \int_{\Omega^\varepsilon} A^{ijk\ell, \varepsilon} e^\varepsilon_{k\| \ell}  (\bm{v}^\varepsilon)e^\varepsilon_{i\|j} (\bm{v}^\varepsilon) \sqrt{g^\varepsilon} \dd x^\varepsilon - \int_{\Omega^\varepsilon} f^{i, \varepsilon} v^\varepsilon_i \sqrt{g^\varepsilon} \dd x^\varepsilon
$$
for each $\bm{v}^\varepsilon = (v^\varepsilon_i) \in \bm{H}^1(\Omega^\varepsilon)$
over the \emph{set of admissible displacements} defined by:
\begin{equation*}
\bm{U}(\Omega^\varepsilon):=\{ \bm{v}^\varepsilon=(v^\varepsilon_i) \in \bm{H}^1(\Omega^\varepsilon); \bm{v}^\varepsilon = \bm{0} \text{ on } \Gamma^\varepsilon_0 \textup{ and } (\bm{\Theta}(x^\varepsilon)+v^\varepsilon_i(x^\varepsilon) \bm{g}^{i,\varepsilon}(x^\varepsilon)) \cdot \bm{q} \ge 0 \textup{ for a.a. } x^\varepsilon \in \Omega^\varepsilon\}. 
\end{equation*}

The solution to this \emph{minimization problem} exists and is unique, and it can be also characterized as the solution of a set of appropriate variational inequalities (cf., Theorem~2.1 of~\cite{CiaMarPie2018}).

\begin{theorem} \label{t:2}
	The quadratic minimization problem: Find a vector field $\bm{u}^\varepsilon \in \bm{U} (\Omega^\varepsilon)$ such that
	$$
	J^\varepsilon (\bm{u}^\varepsilon) = \inf_{\bm{v}^\varepsilon \in \bm{U} (\Omega^\varepsilon)} J^\varepsilon (\bm{v}^\varepsilon)
	$$
	has one and only one solution. Besides, the vector field $\bm{u}^\varepsilon$ is also the unique solution of the variational problem $\mathcal{P} (\Omega^\varepsilon)$: Find $\bm{u}^\varepsilon \in \bm{U} (\Omega^\varepsilon)$ that satisfies the following variational inequalities:
	$$
	\int_{\Omega^\varepsilon} 
	A^{ijk\ell, \varepsilon} e^\varepsilon_{k\| \ell}  (\bm{u}^\varepsilon)
	\left( e^\varepsilon_{i\| j}  (\bm{v}^\varepsilon) -  e^\varepsilon_{i\| j}  (\bm{u}^\varepsilon)  \right) \sqrt{g^\varepsilon} \dd x^\varepsilon \geq \int_{\Omega^\varepsilon} f^{i , \varepsilon} (v^\varepsilon_i - u^\varepsilon_i)\sqrt{g^\varepsilon} \dd x^\varepsilon
	$$
	for all $\bm{v}^\varepsilon = (v^\varepsilon_i) \in \bm{U}(\Omega^\varepsilon)$.
	\qed
\end{theorem}

Since $\bm{\theta} (\overline{\omega}) \subset \bm{\Theta} (\overline{\Omega^\varepsilon})$, it evidently follows that $\bm{\theta} (y) \cdot \bm{q} \geq 0$ for all $y \in \overline{\omega}$. But in fact, a stronger property holds (cf., Lemma~2.1 of~\cite{CiaMarPie2018}, and see also~\cite{Pie-2022-jde} for a different approach to the asymptotic analysis):

\begin{lemma}
	\label{lem0}
	Let $\omega$ be a domain in $\mathbb{R}^2$, let $\bm{\theta} \in \mathcal{C}^1(\overline{\omega}; \mathbb{E}^3)$ be an immersion, let $\bm{q} \in \mathbb{E}^3$ be a non-zero vector, and let $\varepsilon > 0$. Then the inclusion
	$$
	\bm{\Theta} (\overline{\Omega^\varepsilon} ) \subset \mathbb{H} = \{ x \in \mathbb{E}^3; \; \boldsymbol{Ox} \cdot \bm{q} \geq 0\}
	$$
	implies that
	$$
	\min_{y \in \overline{\omega}} (\bm{\theta} (y) \cdot \bm{q}) > 0.
	$$
	\qed
\end{lemma}

\section{The scaled three-dimensional problem for a family of linearly elastic elliptic membrane shells} \label{sec3}

In section~\ref{Sec:2}, we considered an obstacle problem for ``general'' linearly elastic shells. From now on, we will restrict ourselves  to a specific class of shells, according to the following definition that was originally proposed in~\cite{Ciarlet1996} (see also \cite{Ciarlet2000}).

Consider a linearly elastic shell, subjected to the various assumptions set forth in section~\ref{Sec:2}. Such a shell is said to be a \emph{linearly elastic elliptic membrane shell} (from now on simply \emph{membrane shell}) if the following two additional assumptions are satisfied: \emph{first}, $\gamma_0 = \gamma$, i.e., the homogeneous boundary condition of place is imposed over the \emph{entire lateral face} $\gamma \times \left] - \varepsilon , \varepsilon \right[$ of the shell, and \emph{second}, its middle surface $\bm{\theta}(\overline{\omega})$ is \emph{elliptic}, according to the definition given in section \ref{sec1}.

In this paper, we consider the \emph{obstacle problem} (as defined in section~\ref{sec2}) \emph{for a family of membrane shells}, all sharing the \emph{same middle surface} and whose thickness $2 \varepsilon > 0$ is considered as a ``small'' parameter approaching zero. In order to conduct an asymptotic analysis on the three-dimensional model as the thickness $\varepsilon \to 0$, we resorted to a (by now standard) methodology first proposed in~\cite{CiaDes1979}: To begin with, we ``scale'' each problem $\mathcal{P} (\Omega^\varepsilon), \, \varepsilon > 0$, over a \emph{fixed domain} $\Omega$, using appropriate \emph{scalings on the unknowns} and \emph{assumptions on the data}. 

More specifically, let
$$
\Omega := \omega \times \left] - 1, 1 \right[ ,
$$
let $x = (x_i)$ denote a generic point in the set $\overline{\Omega}$, and let $\partial_i := \partial/ \partial x_i$. With each point $x = (x_i) \in \overline{\Omega}$, we associate  the point $x^\varepsilon = (x^\varepsilon_i)$ defined by
$$
x^\varepsilon_\alpha := x_\alpha = y_\alpha \text{ and } x^\varepsilon_3 := \varepsilon x_3,
$$
so that $\partial^\varepsilon_\alpha = \partial_\alpha$ and $\partial^\varepsilon_3 = \varepsilon^{-1} \partial_3$. To the unknown $\bm{u}^\varepsilon = (u^\varepsilon_i)$ and to the vector fields $\bm{v}^\varepsilon = (v^\varepsilon_i)$ appearing in the formulation of the problem $\mathcal{P} (\Omega^\varepsilon)$ corresponding to a membrane shell, we then associate the \emph{scaled unknown} $\bm{u} (\varepsilon) = (u_i(\varepsilon))$ and the \emph{scaled vector fields} $\bm{v} = (v_i)$ by letting
$$
u_i (\varepsilon) (x) := u^\varepsilon_i (x^\varepsilon) \quad\text{ and }\quad v_i(x) := v^\varepsilon_i (x^\varepsilon)
$$
at each $x\in \overline{\Omega}$. Finally, we \emph{assume} that there exist functions $f^i \in L^2(\Omega)$ \emph{independent of} $\varepsilon$ such that the following \emph{assumptions on the data} hold
$$
f^{i, \varepsilon} (x^\varepsilon) = f^i(x) \text{ at each } x \in \Omega.
$$

Note that the independence on $\varepsilon$ of the Lam\'{e} constants assumed in Section \ref{Sec:2} in the formulation of problem $\mathcal{P} (\Omega^\varepsilon)$ implicitly constituted another \emph{assumption on the data}.

The variational problem $\mathcal{P} (\varepsilon; \Omega)$ defined in the next theorem will constitute the point of departure of the asymptotic analysis performed in~\cite{CiaMarPie2018}.

\begin{theorem} \label{t:3}
	For each $\varepsilon > 0$, define the set
	\begin{align*}
		\bm{U}(\varepsilon;\Omega) &:= \{\bm{v} = (v_i) \in \bm{H}^1(\Omega); \bm{v} = \bm{0} \textup{ on } \gamma \times \left] -1, 1 \right[ , \\
		& \big(\bm{\theta} (y) + \varepsilon x_3 \bm{a}_3 (y) + v_i (x) \bm{g}^i(\varepsilon) (x)\big) \cdot \bm{q} \ge 0 \textup{ for a.a. } x = (y,x_3) \in \Omega  \},   
	\end{align*}
	where
	$$
	\bm{g}^i(\varepsilon ) (x) := \bm{g}^{i, \varepsilon} (x^\varepsilon) \textup{ at each } x\in \overline{\Omega}.
	$$
	
	Then the scaled unknown of the variational problem $\mathcal{P}(\Omega^\varepsilon)$ is the unique solution of the variational problem $\mathcal{P} (\varepsilon; \Omega)$: Find $\bm{u}(\varepsilon) \in \bm{U}(\varepsilon; \Omega)$ that satisfies the following variational inequalities:
	\begin{equation*}
		\int_\Omega A^{ijk\ell}(\varepsilon) e_{k\| \ell}(\varepsilon; \bm{u}(\varepsilon)) \left(e_{i\| j}(\varepsilon; \bm{v}) - e_{i\|j}(\varepsilon; \bm{u}(\varepsilon))\right) \sqrt{g(\varepsilon)} \dd x 
		\ge \int_\Omega f^i (v_i - u_i(\varepsilon)) \sqrt{g(\varepsilon)} \dd x,
	\end{equation*}
	for all $\bm{v} \in \bm{U}(\varepsilon;\Omega)$, where
	\begin{align*}
		g(\varepsilon)(x) &:= g^\varepsilon(x^\varepsilon)
		\textup{ and } A^{ijk\ell}(\varepsilon) (x) :=
		A^{ijk\ell, \varepsilon}(x^\varepsilon)
		\textup{ at each } x\in \overline{\Omega}, \\
		e_{\alpha \| \beta}(\varepsilon; \bm{v}) &:= \frac{1}{2}(\partial_\beta v_\alpha + \partial_\alpha v_\beta) - \Gamma^k_{\alpha \beta}(\varepsilon) v_k = e_{\beta \| \alpha}(\varepsilon;\bm{v}) , \\
		e_{3\| \alpha}(\varepsilon; \bm{v}) &:= \frac12 \left(\frac{1}{\varepsilon} \partial_3 v_\alpha + \partial_\alpha v_3\right) - \Gamma^\sigma_{\alpha3}(\varepsilon) v_\sigma=e_{\alpha \| 3}(\varepsilon; \bm{v}),\\
		e_{3\| 3}(\varepsilon;\bm{v}) &:= \frac{1}{\varepsilon} \partial_3 v_3,
	\end{align*}
	where
	$$
	\Gamma^p_{ij}(\varepsilon)(x) := \Gamma^{p,\varepsilon}_{ij}(x^\varepsilon) \textup{ at each } x\in \overline{\Omega}.
	$$
	\qed
\end{theorem}

The problem we are interested in is derived as a result of the rigorous asymptotic analysis conducted in Theorem~4.1 of~\cite{CiaMarPie2018}.

\begin{theorem}\label{t:4}
	Let $\omega$ be a domain in $\mathbb{R}^2$, let $\bm{\theta} \in \mathcal{C}^3(\overline{\omega}; \mathbb{E}^3)$ be an immersion such that the surface $\bm{\theta} (\overline{\omega})$ is elliptic \textup{(cf.\, section \ref{sec1})}. Define the space and sets
	\begin{align*}
		\bm{V}_M (\omega) &:= H^1_0 (\omega) \times H^1_0 (\omega) \times L^2(\omega), \\
		\bm{U}_M (\omega) &:= \{\bm{\eta} = (\eta_i) \in H^1_0 (\omega) \times H^1_0 (\omega) \times L^2(\omega); \big(\bm{\theta} (y) + \eta_i (y) \bm{a}^i(y)\big) \cdot \bm{q} \geq 0 \textup{ for a.a. } y \in \omega \}, \\
		\tilde{\bm{U}}_M (\omega) &:= \{\bm{\eta} = (\eta_i) \in H^1_0 (\omega) \times H^1_0 (\omega) \times H^1_0(\omega); \big(\bm{\theta} (y) + \eta_i (y) \bm{a}^i(y)\big) \cdot \bm{q} \geq 0 \textup{ for a.a. } y \in \omega \}, 
	\end{align*}
	and assume that the immersion $\bm{\theta}$ is such that
	$$
	\tilde{d}:=\min_{y \in \overline{\omega}}(\bm{\theta}(y)\cdot\bm{q})>0,
	$$
	is independent of $\varepsilon$, and assume that  the following ``density property'' holds:
	$$
	\tilde{\bm{U}}_M (\omega) \textup{ is dense in } \bm{U}_M(\omega) \textup{ with respect to the norm of } \left\| \cdot \right\|_{H^1 (\omega) \times H^1(\omega) \times L^2(\omega)}.
	$$
	
	Let there be given a family of membrane shells with the same middle surface $\bm{\theta} (\overline{\omega})$ and thickness $2 \varepsilon > 0$, and let
	\begin{align*}
		\bm{u}(\varepsilon) &= (u_i(\varepsilon)) \in \bm{U} ( \varepsilon; \Omega) := \{ \bm{v} = (v_i) \in \bm{H}^1(\Omega) ; \; \bm{v} = \bm{0} \textup{ on } \gamma \times \left] -1, 1 \right[ , \\
		& \big(\bm{\theta} (y) + \varepsilon x_3 \bm{a}_3 (y) + v_i (x) \bm{g}^i(\varepsilon) (x)\big) \cdot \bm{q} \geq 0 \textup{ for a.a. } x = (y, x_3 ) \in \Omega  \}
	\end{align*}
	denote for each $\varepsilon > 0$ the unique solution of the corresponding problem $\mathcal{P} (\varepsilon; \Omega)$ introduced in Theorem~\ref{t:3}.
	
	Then there exist functions $u_\alpha \in H^1(\Omega)$ independent of the variable $x_3$ and satisfying
	$$
	u_\alpha = 0 \textup{ on } \gamma \times \left]-1, 1\right[,
	$$
	and there exists a function $u_3 \in L^2(\Omega)$ independent of the variable $x_3$, such that
	$$
	u_\alpha (\varepsilon ) \to u_\alpha \textup{ in } H^1(\Omega) \textup{ and } u_3 (\varepsilon) \to u_3 \textup{ in } L^2(\Omega).
	$$
	
	Define the average
	$$
	\overline{\bm{u}} = (\overline{u}_i) := \frac12 \int^1_{-1} \bm{u} \dd x_3 \in \bm{V}_M(\omega).
	$$
	Then
	$$
	\overline{\bm{u}} = \bm{\zeta},
	$$
	where $\bm{\zeta}$ is the unique solution of the two-dimensional variational problem $\mathcal{P}_M (\omega)$: Find $\bm{\zeta} \in \bm{U}_M(\omega)$ that satisfies the following variational inequalities
	$$
	\int_\omega a^{\alpha \beta \sigma \tau} \gamma_{\sigma \tau}(\bm{\zeta}) \gamma_{\alpha \beta} (\bm{\eta} - \bm{\zeta}) \sqrt{a} \dd y \geq \int_\omega p^i (\eta_i - \zeta_i) \sqrt{a} \dd y \quad\textup{ for all } \bm{\eta} = (\eta_i) \in \bm{U}_M(\omega),
	$$
	where
	$$
	a^{\alpha \beta \sigma \tau} := \frac{4\lambda \mu}{\lambda + 2 \mu} a^{\alpha \beta} a^{\sigma \tau} + 2\mu \left(a^{\alpha \sigma} a^{\beta \tau} + a^{\alpha \tau} a^{\beta \sigma}\right) \textup{ and } p^i := \int^1_{-1} f^i \dd x_3.
	$$
	\qed
\end{theorem}

Note that it does not make sense to talk about the trace of $\zeta_3$ along $\gamma$, since $\zeta_3$ is \emph{a priori} only of class $L^2(\omega)$.
The loss of the homogeneous boundary condition for the transverse component of the limit model, which is \emph{a priori} only square integrable, is \emph{compensated} by the appearance of a boundary layer for the transverse component.
By proving that the solution enjoys a higher regularity, we will establish that it is possible to \emph{restore} the boundary condition for the transverse component of the solution too, and that the trace of the transverse component of the solution along the boundary is almost everywhere (in the sense of the measure of the contour) equal to zero.

Critical to establish the convergence of the family $\{\bm{u}(\varepsilon)\}_{\varepsilon > 0}$ is the ``density property" assumed there, which asserts that \emph{the set $\tilde{\bm{U}}_M (\omega)$ is dense in the set $\bm{U}_M(\omega)$ with respect to the norm $\left\| \cdot \right\|_{H^1(\omega) \times H^1(\omega) \times L^2(\omega)}$}. The same ``density property" is used to provide a justification, via a rigorous asymptotic analysis, of Koiter's model for membrane shells subject to an obstacle (cf. \cite{CiaPie2018bCR}, \cite{CiaPie2018b}).
We hereby recall a sufficient \emph{geometric} condition ensuring the  assumed ``density property'' (cf. Theorem~5.1 of~\cite{CiaMarPie2018}).

\begin{theorem}
	\label{density}
	Let $\bm{\theta} \in \mathcal{C}^2(\overline{\omega}; \mathbb{E}^3)$ be an immersion with the following property: There exists a non-zero vector $\bm{q} \in \mathbb{E}^3$ such that
	\begin{equation*}
	\min_{y \in \overline{\omega}} (\bm{\theta} (y) \cdot \bm{q}) > 0
	\textup{ and }
	\min_{y \in \overline{\omega}} (\bm{a}_3 (y) \cdot \bm{q}) > 0.
	\end{equation*}
	
	Define the sets
	\begin{align*}
		\bm{U}_M (\omega) := \{\bm{\eta} = &(\eta_i) \in H^1_0 (\omega) \times H^1_0 (\omega) \times L^2(\omega); \big(\bm{\theta} (y) + \eta_i (y) \bm{a}^i(y)\big) \cdot \bm{q} \geq 0 \textup{ for a.a. } y \in \omega \}, \\
		\bm{U}_M(\omega) \cap \boldsymbol{\mathcal{D}} (\omega) := \{ \bm{\eta} =& (\eta_i ) \in \mathcal{D} (\omega) \times \mathcal{D} (\omega) \times \mathcal{D} (\omega); \big(\bm{\theta} (y) + \eta_i (y) \bm{a}^i(y)\big) \cdot \bm{q} \geq 0   \textup{ for a.a. } y \in \omega \}.
	\end{align*}
	Then the set $\bm{U}_M (\omega) \cap \boldsymbol{\mathcal{D}} (\omega)$ is dense in the set $\bm{U}_M(\omega)$ with respect to the norm $\left\| \cdot \right\|_{H^1(\omega) \times H^1(\omega) \times L^2(\omega)}$.
	\qed
\end{theorem}

Examples of membrane shells satisfying the ``density property'' thus include those whose middle surface is a portion of an ellipsoid that is strictly contained in one of the open half-spaces that contain two of its main axes, the boundary of the half-space coinciding with the obstacle in this case.

As a final step, we \emph{de-scale} Problem $\mathcal{P}_M(\omega)$ and we obtain the following variational formulation (cf. Theorem~4.2 of~\cite{CiaMarPie2018}).

\begin{customprob}{$\mathcal{P}_M^\varepsilon(\omega)$}
	\label{problem1}
	Find $\bm{\zeta}^\varepsilon=(\zeta_i^\varepsilon) \in \bm{U}_M(\omega)$ satisfying the following variational inequalities:
	\begin{equation*}
		\varepsilon \int_\omega a^{\alpha \beta \sigma \tau} \gamma_{\sigma \tau}(\bm{\zeta}^\varepsilon) \gamma_{\alpha \beta} (\bm{\eta} - \bm{\zeta}^\varepsilon) \sqrt{a} \dd y \geq \int_\omega p^{i,\varepsilon} (\eta_i - \zeta_i^\varepsilon) \sqrt{a} \dd y,
	\end{equation*}
	for all $\bm{\eta} = (\eta_i) \in \bm{U}_M(\omega)$, where $p^{i,\varepsilon}:=\varepsilon\int_{-1}^{1} f^i \dd x_3$.
	\bqed
\end{customprob}

By virtue of the Korn inequality recalled in Theorem~\ref{korn}, it results that Problem~\ref{problem1} admits a unique solution. Solving Problem~\ref{problem1} amounts to minimizing the energy functional $J^\varepsilon: H^1(\omega) \times H^1(\omega) \times L^2(\omega) \to \mathbb{R}$, which is defined by
\begin{equation*}
\label{Jeps}
J^\varepsilon(\bm{\eta}):=\dfrac{\varepsilon}{2} \int_{\omega} a^{\alpha\beta\sigma\tau} \gamma_{\sigma\tau}(\bm{\eta}) \gamma_{\alpha\beta}(\bm{\eta})\sqrt{a} \dd y-\int_{\omega} p^{i,\varepsilon} \eta_i \sqrt{a} \dd y,
\end{equation*}
along all the test functions $\bm{\eta}=(\eta_i) \in \bm{U}_M(\omega)$.

\section{Approximation of the solution of Problem~$\mathcal{P}_M^\varepsilon(\omega)$ by penalization}
\label{sec:penalty}

Following~\cite{Scholz1984}, we first approximate the solution of Problem~\ref{problem1} by penalty method. By so doing, the geometrical constraint appearing in the definition of the set $\bm{U}_M(\omega)$ the deformation must obey now appears in the governing model in the form of a monotone term. As a consequence of this, the test vector fields are no longer sought in a non-empty, closed and convex subset of $\bm{V}_M(\omega)$, but in the whole $\bm{V}_M(\omega)$, and the variational inequalities are replaced by a set of nonlinear equations, where the nonlinearity is monotone.

More precisely, define the operator $\bm{\beta}:\bm{L^2}(\omega) \to\bm{L}^2(\omega)$ in the following fashion
\begin{equation*}
\bm{\beta}(\bm{\xi}):=\left(-\{(\bm{\theta}+\xi_j \bm{a}^j)\cdot\bm{q}\}^{-}\left(\dfrac{\bm{a}^i\cdot\bm{q}}{\sqrt{\sum_{\ell=1}^{3}|\bm{a}^\ell \cdot\bm{q}|^2}}\right)\right)_{i=1}^3,\quad\textup{ for all }\bm{\xi}=(\xi_i) \in \bm{L}^2(\omega),
\end{equation*}
and we notice that this operator is associated with a penalization proportional to the extent the constraint is broken. Note that the denominator never vanishes and that this fact is independent of the assumption $\min_{y \in \overline{\omega}}(\bm{a}^3\cdot\bm{q})>0$.
Following the ideas of~\cite{PWDT3D} (see also~\cite{Pie2023,PT2023}), we show that the operator $\bm{\beta}$ is monotone, bounded and non-expansive.

\begin{lemma}
\label{lem:beta}
Let $\bm{q} \in \mathbb{E}^3$ be a given unit-norm vector. Assume that $\min_{y \in \overline{\omega}}(\bm{a}^3(y)\cdot\bm{q})>0$.
Then, the operator $\bm{\beta}:\bm{L^2}(\omega) \to\bm{L}^2(\omega)$ defined by 
\begin{equation*}
\bm{\beta}(\bm{\xi}):=\left(-\{(\bm{\theta}+\xi_j \bm{a}^j)\cdot\bm{q}\}^{-}\left(\dfrac{\bm{a}^i\cdot\bm{q}}{\sqrt{\sum_{\ell=1}^{3}|\bm{a}^\ell \cdot\bm{q}|^2}}\right)\right)_{i=1}^3,\quad\textup{ for all }\bm{\xi}=(\xi_i) \in \bm{L}^2(\omega),
\end{equation*}
is bounded, monotone and Lipschitz continuous with Lipschitz constant $L=1$.
\end{lemma}
\begin{proof}
Let $\bm{\xi}$ and $\bm{\eta}$ be arbitrarily given in $\bm{L}^2(\omega)$. Evaluating
\begin{align*}
&\int_{\omega} (\bm{\beta}(\bm{\xi})-\bm{\beta}(\bm{\eta}))\cdot(\bm{\xi}-\bm{\eta})\dd y
=\int_{\omega} \left(\left[-\{(\bm{\theta}+\xi_j\bm{a}^j)\cdot\bm{q}\}^{-}\right] - \left[-\{(\bm{\theta}+\eta_j\bm{a}^j)\cdot\bm{q}\}^{-}\right]\right) \left(\dfrac{(\xi_i-\eta_i)\bm{a}^i\cdot\bm{q}}{\sqrt{\sum_{\ell=1}^{3}|\bm{a}^\ell \cdot\bm{q}|^2}}\right) \dd y\\
&=\int_{\omega}\dfrac{\left|-\{(\bm{\theta}+\xi_j\bm{a}^j)\cdot\bm{q}\}^{-}\right|^2}{\sqrt{\sum_{\ell=1}^{3}|\bm{a}^\ell \cdot\bm{q}|^2}} \dd y +\int_{\omega}\dfrac{\left|-\{(\bm{\theta}+\eta_j\bm{a}^j)\cdot\bm{q}\}^{-}\right|^2}{\sqrt{\sum_{\ell=1}^{3}|\bm{a}^\ell \cdot\bm{q}|^2}} \dd y\\
&\quad+\int_{\omega}\dfrac{\left(-\{(\bm{\theta}+\xi_j\bm{a}^j)\cdot\bm{q}\}^{-}\right)}{\sqrt{\sum_{\ell=1}^{3}|\bm{a}^\ell \cdot\bm{q}|^2}}  \left(-\{(\bm{\theta}+\eta_i\bm{a}^i)\cdot\bm{q}\}^{+}+\{(\bm{\theta}+\eta_i\bm{a}^i)\cdot\bm{q}\}^{-}\right)\dd y\\
&\quad+\int_{\omega}\dfrac{\left(-\{(\bm{\theta}+\eta_j\bm{a}^j)\cdot\bm{q}\}^{-}\right)}{\sqrt{\sum_{\ell=1}^{3}|\bm{a}^\ell \cdot\bm{q}|^2}}  \left(-\{(\bm{\theta}+\xi_i\bm{a}^i)\cdot\bm{q}\}^{+}+\{(\bm{\theta}+\xi_i\bm{a}^i)\cdot\bm{q}\}^{-}\right)\dd y\\
&\ge \int_{\omega}\dfrac{\left|-\{(\bm{\theta}+\xi_j\bm{a}^j)\cdot\bm{q}\}^{-}\right|^2}{\sqrt{\sum_{\ell=1}^{3}|\bm{a}^\ell \cdot\bm{q}|^2}} \dd y 
+\int_{\omega}\dfrac{\left|-\{(\bm{\theta}+\eta_j\bm{a}^j)\cdot\bm{q}\}^{-}\right|^2}{\sqrt{\sum_{\ell=1}^{3}|\bm{a}^\ell \cdot\bm{q}|^2}} \dd y
+\int_{\omega}\dfrac{\left(-\{(\bm{\theta}+\xi_j\bm{a}^j)\cdot\bm{q}\}^{-}\right)}{\sqrt{\sum_{\ell=1}^{3}|\bm{a}^\ell \cdot\bm{q}|^2}}  \left(\{(\bm{\theta}+\eta_i\bm{a}^i)\cdot\bm{q}\}^{-}\right)\dd y\\
&\quad+\int_{\omega}\dfrac{\left(-\{(\bm{\theta}+\eta_j\bm{a}^j)\cdot\bm{q}\}^{-}\right)}{\sqrt{\sum_{\ell=1}^{3}|\bm{a}^\ell \cdot\bm{q}|^2}}  \left(\{(\bm{\theta}+\xi_i\bm{a}^i)\cdot\bm{q}\}^{-}\right)\dd y\\
&=\int_{\omega}\dfrac{\left|\left(-\{(\bm{\theta}+\eta_j\bm{a}^j)\cdot\bm{q}\}^{-}\right) - \left(-\{(\bm{\theta}+\xi_j\bm{a}^j)\cdot\bm{q}\}^{-}\right)\right|^2}{\sqrt{\sum_{\ell=1}^{3}|\bm{a}^\ell \cdot\bm{q}|^2}}\dd y\ge 0,
\end{align*}
proves the monotonicity of the operator $\bm{\beta}$.

For showing the boundedness of the operator $\bm{\beta}$, we show that it maps bounded sets of $\bm{L}^2(\omega)$ into bounded sets of $\bm{L}^2(\omega)$.
Let the set $\mathscr{F} \subset \bm{L}^2(\omega)$ be bounded. For each $\bm{\xi} \in \mathscr{F}$, we have that
\begin{align*}
&\|\bm{\beta}(\bm{\xi})\|_{\bm{L}^2(\omega)}=\left(\int_{\omega}\dfrac{|-\{(\bm{\theta}+\xi_j\bm{a}^j)\cdot\bm{q}\}^{-}|^2}{\sum_{\ell=1}^3|\bm{a}^\ell \cdot\bm{q}|^2}\sum_{i=1}^3|\bm{a}^i\cdot\bm{q}|^2 \dd y\right)^{1/2}\\
&= \|-\{(\bm{\theta}+\xi_j\bm{a}^j)\cdot\bm{q}\}^{-}\|_{L^2(\omega)} \le \|\bm{\theta}\cdot\bm{q}\|_{L^2(\omega)}+\|\bm{\xi}\|_{\bm{L}^2(\omega)},
\end{align*}
and the sought boundedness is thus asserted, being $\bm{\theta} \in \mathcal{C}^3(\overline{\omega};\mathbb{E}^3)$ and $\mathscr{F}$ bounded in $\bm{L}^2(\omega)$.

Finally, to establish the Lipschitz continuity, for all $\bm{\xi}$ and $\bm{\eta} \in \bm{L}^2(\omega)$, we evaluate $\|\bm{\beta}(\bm{\xi})-\bm{\beta}(\bm{\eta})\|_{\bm{L}^2(\omega)}$. We have that
\begin{equation*}
\begin{aligned}
&\|\bm{\beta}(\bm{\xi})-\bm{\beta}(\bm{\eta})\|_{\bm{L}^2(\omega)}=\left(\int_{\omega}\dfrac{1}{\sum_{\ell=1}^{3}|\bm{a}^\ell \cdot\bm{q}|^2}\left\{\left|[-\{(\bm{\theta}+\xi_i\bm{a}^i)\cdot\bm{q}\}^{-}] - [-\{(\bm{\theta}+\eta_j\bm{a}^j)\cdot\bm{q}\}^{-}] \right|^2 \left(\sum_{\ell=1}^{3}|\bm{a}^\ell \cdot\bm{q}|^2\right) \right\}\dd y\right)^{1/2}\\
&=\left(\int_{\omega}\left|\left[-\{(\bm{\theta}+\xi_j\bm{a}^j)\cdot\bm{q}\}^{-}\right] - \left[-\{(\bm{\theta}+\eta_j\bm{a}^j)\cdot\bm{q}\}^{-}\right]\right|^2 \dd y\right)^{1/2}\\
&=\left(\int_{\omega}\left|\dfrac{|(\bm{\theta}+\xi_j\bm{a}^j)\cdot\bm{q}|-(\bm{\theta}+\xi_j\bm{a}^j)\cdot\bm{q}}{2} - \dfrac{|(\bm{\theta}+\eta_j\bm{a}^j)\cdot\bm{q}|-(\bm{\theta}+\eta_j\bm{a}^j)\cdot\bm{q}}{2}\right|^2 \dd y\right)^{1/2}\\
&=\left(\int_{\omega}\left|\dfrac{|(\bm{\theta}+\xi_j\bm{a}^j)\cdot\bm{q}|-|(\bm{\theta}+\eta_j\bm{a}^j)\cdot\bm{q}|}{2} - \dfrac{(\bm{\theta}+\xi_j\bm{a}^j)\cdot\bm{q}-(\bm{\theta}+\eta_j\bm{a}^j)\cdot\bm{q}}{2}\right|^2 \dd y\right)^{1/2}\\
&\le\left(\int_{\omega}\left|\left|\dfrac{(\bm{\theta}+\xi_j\bm{a}^j)\cdot\bm{q}-(\bm{\theta}+\eta_j\bm{a}^j)\cdot\bm{q}}{2}\right| - \dfrac{(\bm{\theta}+\xi_j\bm{a}^j)\cdot\bm{q}-(\bm{\theta}+\eta_j\bm{a}^j)\cdot\bm{q}}{2}\right|^2 \dd y\right)^{1/2}\\
&\le\left(\int_{\omega}\left(\left|\dfrac{(\bm{\theta}+\xi_j\bm{a}^j)\cdot\bm{q}-(\bm{\theta}+\eta_j\bm{a}^j)\cdot\bm{q}}{2}\right| + \left|\dfrac{(\bm{\theta}+\xi_j\bm{a}^j)\cdot\bm{q}-(\bm{\theta}+\eta_j\bm{a}^j)\cdot\bm{q}}{2}\right|\right)^2 \dd y\right)^{1/2}\\
&\le \left\|(\bm{\theta}+\xi_j\bm{a}^j)\cdot\bm{q}-(\bm{\theta}+\eta_j\bm{a}^j)\cdot\bm{q}\right\|_{L^2(\omega)}\le \|\bm{\xi}-\bm{\eta}\|_{\bm{L}^2(\omega)},
\end{aligned}
\end{equation*}
and the sought Lipschitz continuity is thus established. Note in passing that the Lipschitz constant is $L=1$. This completes the proof.
\end{proof}

Let $\kappa>0$ denote a penalty parameter which is meant to approach zero. The penalized version of Problem~\ref{problem1} is formulated as follows.

\begin{customprob}{$\mathcal{P}_{M,\kappa}^\varepsilon(\omega)$}
	\label{problem2}
	Find $\bm{\zeta}^\varepsilon_\kappa=(\zeta^\varepsilon_{\kappa,i}) \in \bm{V}_M(\omega)$ satisfying the following variational equations:
	\begin{equation*}
	\varepsilon \int_\omega a^{\alpha \beta \sigma \tau} \gamma_{\sigma \tau}(\bm{\zeta}^\varepsilon_\kappa) \gamma_{\alpha \beta} (\bm{\eta}) \sqrt{a} \dd y 
	+\dfrac{\varepsilon}{\kappa}\int_{\omega} \bm{\beta}(\bm{\zeta}^\varepsilon_\kappa) \cdot \bm{\eta} \dd y
	= \int_\omega p^{i,\varepsilon} \eta_i \sqrt{a} \dd y,
	\end{equation*}
	for all $\bm{\eta} = (\eta_i) \in \bm{V}_M(\omega)$.
	\bqed
\end{customprob}

The existence and uniqueness of solutions of Problem~\ref{problem2} can be established by resorting to the Minty-Browder theorem (cf., e.g., Theorem~9.14-1 of~\cite{PGCLNFAA}). For the sake of completeness, we present the proof of this existence and uniqueness result.

\begin{theorem}
\label{ex-un-kappa}
Let $\bm{q} \in\mathbb{E}^3$ be a given unit-norm vector. Assume that $\bm{\theta} \in \mathcal{C}^3(\overline{\omega};\mathbb{E}^3)$ is such that $\min_{y \in \overline{\omega}}(\bm{\theta}(y)\cdot\bm{q})>~0$.

Then, for each $\kappa>0$ and $\varepsilon>0$, Problem~\ref{problem2} admits a unique solution. Moreover, the family of solutions $\{\bm{\zeta}^\varepsilon_\kappa\}_{\kappa>0}$ is bounded in $\bm{V}_M(\omega)$ independently of $\kappa$ and $\varepsilon$, and 
$$
\bm{\zeta}^\varepsilon_\kappa \to \bm{\zeta}^\varepsilon,\quad\textup{ in }\bm{V}_M(\omega) \textup{ as }\kappa \to 0^+,
$$
where $\bm{\zeta}^\varepsilon$ is the solution of Problem~\ref{problem1}.
\end{theorem}
\begin{proof}
Let us define the operator $\bm{A}^\varepsilon:\bm{V}_M(\omega) \to \bm{V}'_M(\omega)$ by
\begin{equation*}
\langle \bm{A}^\varepsilon \bm{\xi},\bm{\eta}\rangle_{\bm{V}'_M(\omega), \bm{V}_M(\omega)}:=\varepsilon\int_{\omega} a^{\alpha\beta\sigma\tau} \gamma_{\sigma\tau}(\bm{\xi}) \gamma_{\alpha\beta}(\bm{\eta}) \sqrt{a} \dd y.
\end{equation*}

We observe that the operator $\bm{A}^\varepsilon$ is linear, continuous and, thanks to Korn's inequality (Theorem~\ref{korn}), such that
\begin{equation}
\label{Aeps}
\langle \bm{A}^\varepsilon \bm{\xi} -\bm{A}^\varepsilon\bm{\eta},\bm{\xi}-\bm{\eta}\rangle_{\bm{V}'_M(\omega), \bm{V}_M(\omega)} \ge \varepsilon c \|\bm{\xi}-\bm{\eta}\|_{\bm{V}_M(\omega)}^2, \quad\textup{ for all }\bm{\xi}, \bm{\eta} \in \bm{V}_M(\omega),
\end{equation}
for some $c=c(\omega,\bm{\theta})>0$. Define the operator $\hat{\bm{\beta}}:\bm{V}_M(\omega) \to \bm{V}'_M(\omega)$ as the following composition
\begin{equation*}
\bm{V}_M(\omega) \hookrightarrow \bm{L}^2(\omega) \xrightarrow{\bm{\beta}} \bm{L}^2(\omega) \hookrightarrow \bm{V}'_M(\omega).
\end{equation*}

Thanks to the monotonicity of $\bm{\beta}$ established in Lemma~\ref{lem:beta}, we easily infer that $\hat{\bm{\beta}}$ is monotone.
Therefore, as a direct consequence of~\eqref{Aeps} and Lemma~\ref{lem:beta}, we can infer that the operator $(\bm{A}^\varepsilon+\hat{\bm{\beta}}):\bm{V}_M(\omega) \to \bm{V}'_M(\omega)$ is strictly monotone. To see this, observe that for all $\bm{\eta}$, $\bm{\xi} \in \bm{V}_M(\omega)$ with $\bm{\xi}\neq\bm{\eta}$, we have that
\begin{equation*}
\begin{aligned}
&\langle (\bm{A}^\varepsilon+\hat{\bm{\beta}}) \bm{\xi} -(\bm{A}^\varepsilon+\hat{\bm{\beta}})\bm{\eta},\bm{\xi}-\bm{\eta}\rangle_{\bm{V}'_M(\omega), \bm{V}_M(\omega)}\\
&= \langle \bm{A}^\varepsilon \bm{\xi} -\bm{A}^\varepsilon\bm{\eta},\bm{\xi}-\bm{\eta}\rangle_{\bm{V}'_M(\omega), \bm{V}_M(\omega)}\\
&\quad+\langle\hat{\bm{\beta}}(\bm{\xi})-\hat{\bm{\beta}}(\bm{\eta}),\bm{\xi}-\bm{\eta}\rangle_{\bm{V}'_M(\omega), \bm{V}_M(\omega)}\ge \varepsilon c \|\bm{\xi}-\bm{\eta}\|_{\bm{V}_M(\omega)}^2>0.
\end{aligned}
\end{equation*}

Similarly, we can establish the coerciveness of the operator $(\bm{A}^\varepsilon+\hat{\bm{\beta}})$. Indeed, we have that
\begin{equation*}
\dfrac{\langle (\bm{A}^\varepsilon+\hat{\bm{\beta}}) \bm{\eta},\bm{\eta}\rangle_{\bm{V}'_M(\omega), \bm{V}_M(\omega)}}{\|\bm{\eta}\|_{\bm{V}_M(\omega)}} =\dfrac{\langle\bm{A}^\varepsilon\bm{\eta},\bm{\eta}\rangle_{\bm{V}'_M(\omega), \bm{V}_M(\omega)}}{\|\bm{\eta}\|_{\bm{V}_M(\omega)}} +\dfrac{\langle \hat{\bm{\beta}}(\bm{\eta}),\bm{\eta}\rangle_{\bm{V}'_M(\omega), \bm{V}_M(\omega)}}{\|\bm{\eta}\|_{\bm{V}_M(\omega)}} \ge c\varepsilon \|\bm{\eta}\|_{\bm{V}_M(\omega)},
\end{equation*}
where the last inequality is obtained by combining~\eqref{Aeps}, Lemma~\ref{lem:beta} with the fact that $\bm{0} \in \bm{U}_M(\omega)$ or, equivalently, that $\bm{\beta}(\bm{0})=\bm{0}$ in $\bm{L}^2(\omega)$.

The continuity of the operator $\bm{A}^\varepsilon$ and the Lipschitz continuity of the operator $\bm{\beta}$ established in Lemma~\ref{lem:beta} in turn give that the operator $(\bm{A}^\varepsilon+\hat{\bm{\beta}})$ is hemicontinuous, and we are in position to apply the Minty-Browder theorem (cf., e.g., Theorem~9.14-1 of~\cite{PGCLNFAA}) to establish that there exists a unique solution $\bm{\zeta}^\varepsilon_\kappa \in \bm{V}_M(\omega)$ for Problem~\ref{problem2}.

Observe that the fact that $\min_{y\in\overline{\omega}} (\bm{\theta}(y)\cdot\bm{q})>0$ implies:
\begin{equation}
	\label{beta-2}
	\begin{aligned}
		&\int_{\omega} \bm{\beta}(\bm{\zeta}^\varepsilon_\kappa) \cdot\bm{\zeta}^\varepsilon_\kappa \dd y
		=\int_{\omega}\dfrac{1}{\sqrt{\sum_{\ell=1}^{3}|\bm{a}^\ell \cdot\bm{q}|^2}} \left(-\{(\bm{\theta}+\zeta^\varepsilon_{\kappa,j}\bm{a}^j)\cdot\bm{q}\}^{-}\right) (\zeta^\varepsilon_{\kappa,i}\bm{a}^i\cdot\bm{q}) \dd y\\
		&=\int_{\omega}\dfrac{1}{\sqrt{\sum_{\ell=1}^{3}|\bm{a}^\ell \cdot\bm{q}|^2}} \left(-\{(\bm{\theta}+\zeta^\varepsilon_{\kappa,j}\bm{a}^j)\cdot\bm{q}\}^{-}\right) ((\bm{\theta}+\zeta^\varepsilon_{\kappa,i}\bm{a}^i)\cdot\bm{q}) \dd y\\
		&\quad-\int_{\omega}\dfrac{1}{\sqrt{\sum_{\ell=1}^{3}|\bm{a}^\ell \cdot\bm{q}|^2}} \left(-\{(\bm{\theta}+\zeta^\varepsilon_{\kappa,j}\bm{a}^j)\cdot\bm{q}\}^{-}\right) (\bm{\theta}\cdot\bm{q}) \dd y\\
		&\ge \int_{\omega}\dfrac{1}{\sqrt{\sum_{\ell=1}^{3}|\bm{a}^\ell \cdot\bm{q}|^2}} |-\{(\bm{\theta}+\zeta^\varepsilon_{\kappa,i}\bm{a}^i)\cdot \bm{q}\}^{-}|^2  \dd y.
	\end{aligned}
\end{equation}

Furthermore, if we specialize $\bm{\eta}=\bm{\zeta}^\varepsilon_\kappa$ in the variational equations of Problem~\ref{problem2}, we have that an application of Korn's inequality (Theorem~\ref{korn}), the monotonicity of $\bm{\beta}$ (Lemma~\ref{lem:beta}), the strict positiveness and boundedness of $a$ (Theorems~ 3.1-1 of~\cite{Ciarlet2000}), the uniform positive definiteness of the fourth order two-dimensional elasticity tensor $(a^{\alpha\beta\sigma\tau})$ (Theorem~3.3-2 of~\cite{Ciarlet2000}), and the fact that $\bm{0} \in \bm{U}_M(\omega)$ or, equivalently, that $\bm{\beta}(\bm{0})=\bm{0}$ in $\bm{L}^2(\omega)$ give:
\begin{equation*}
\begin{aligned}
\dfrac{\varepsilon\sqrt{a_0}}{c_0 c_e}\|\bm{\zeta}^\varepsilon_\kappa\|_{\bm{V}_M(\omega)}^2 
&\le \varepsilon\int_{\omega} a^{\alpha\beta\sigma\tau} \gamma_{\sigma\tau}(\bm{\zeta}^\varepsilon_\kappa) \gamma_{\alpha \beta}(\bm{\zeta}^\varepsilon_\kappa) \sqrt{a} \dd y +\dfrac{\varepsilon}{\kappa} \int_{\omega} \bm{\beta}(\bm{\zeta}^\varepsilon_\kappa) \cdot\bm{\zeta}^\varepsilon_\kappa \dd y\\
&\le \|\bm{p}^\varepsilon\|_{\bm{L}^2(\omega)} \|\bm{\zeta}^\varepsilon_\kappa\|_{\bm{V}_M(\omega)} \sqrt{a_1}
=\varepsilon \sqrt{a_1}\|\bm{p}\|_{\bm{L}^2(\omega)} \|\bm{\zeta}^\varepsilon_\kappa\|_{\bm{V}_M(\omega)}.
\end{aligned}
\end{equation*}
Note that the last equality holds thanks to the definition of $\bm{p}=(p^i)$ and $\bm{p}^\varepsilon=(p^{i,\varepsilon})$ introduced, respectively, in Theorem~\ref{t:4} and Problem~\ref{problem1}.

By virtue of the definition of $p^{i,\varepsilon}$ and the assumptions on the data stated at the beginning of section~\ref{sec3}, we get that $\|\bm{\zeta}^\varepsilon_\kappa\|_{\bm{V}_M(\omega)}$ is bounded independently of $\kappa$ and $\varepsilon$. Therefore, by the Banach-Eberlein-Smulian theorem (cf., e.g., Theorem~5.14-4 of~\cite{PGCLNFAA}), we can extract a subsequence, still denoted $\{\bm{\zeta}^\varepsilon_\kappa\}_{\kappa>0}$ such that
\begin{equation}
\label{beta-1}
\bm{\zeta}^\varepsilon_\kappa \rightharpoonup \bm{\zeta}^\varepsilon, \quad\textup{ in } \bm{V}_M(\omega) \textup{ as } \kappa\to0^+. 
\end{equation}

Specializing $\bm{\eta}=\bm{\zeta}^\varepsilon_\kappa$ in the variational equations of Problem~\ref{problem2} and applying~\eqref{beta-1} and~\eqref{beta-2} give that
\begin{equation}
\label{beta-2.5}
\dfrac{\left(3\max\{\|\bm{a}^j \cdot \bm{q}\|_{\mathcal{C}^0(\overline{\omega})}^2;1\le j\le 3\}\right)^{-1/2}}{\kappa}\|-\{(\bm{\theta}+\zeta^\varepsilon_{\kappa,j}\bm{a}^j)\bm
q\}^{-}\|_{\bm{L}^2(\omega)}^2
\le\dfrac{1}{\kappa}\int_{\omega} \bm{\beta}(\bm{\zeta}^\varepsilon_\kappa) \cdot\bm{\zeta}^\varepsilon_\kappa \dd y\le C,
\end{equation}
for some $C>0$ independent of $\varepsilon$ and $\kappa$. Therefore, we have that an application of the Banach-Eberlein-Smulian theorem and~\eqref{beta-2.5} give that
\begin{equation}
\label{beta-3}
\bm{\beta}(\bm{\zeta}^\varepsilon_\kappa) \to \bm{0},\quad\textup{ in }\bm{L}^2(\omega) \textup{ as }\kappa\to0^+,
\end{equation}
and that
\begin{equation}
\label{beta-4}
\langle \hat{\bm{\beta}}(\bm{\zeta}^\varepsilon_\kappa),\bm{\zeta}^\varepsilon_\kappa\rangle_{\bm{V}'_M(\omega), \bm{V}_M(\omega)} \to 0,\quad\textup{ as }\kappa\to0^+.
\end{equation}

Therefore, the monotonicity of $\hat{\bm{\beta}}$ (which is a direct consequence of Lemma~\ref{lem:beta}), and the the properties established in~\eqref{beta-1}, \eqref{beta-3} and~\eqref{beta-4} give that $\hat{\bm{\beta}}(\bm{\zeta}^\varepsilon)=\bm{0}$, so that $\bm{\zeta}^\varepsilon \in \bm{U}_M(\omega)$.

Observe that the monotonicity of $\bm{\beta}$ (viz. Lemma~\ref{lem:beta}), the properties of $\bm{\zeta}^\varepsilon_\kappa$, the continuity of the components $\gamma_{\alpha \beta}$ of the linearized change of metric tensor, the definition of $\bm{p}^\varepsilon$ (Theorem~\ref{t:4}), the boundedness $\bm{\zeta}^\varepsilon$ independently of $\varepsilon$ (Theorem~\ref{t:4}), and the weak convergence~\eqref{beta-1} give
\begin{align*}
\|\bm{\zeta}^\varepsilon_\kappa - \bm{\zeta}^\varepsilon\|_{\bm{V}_M(\omega)}^2 &\le
\dfrac{c_0 c_e}{\sqrt{a_0}}\int_{\omega} a^{\alpha\beta\sigma\tau}\gamma_{\sigma\tau}(\bm{\zeta}^\varepsilon_\kappa - \bm{\zeta}^\varepsilon)\gamma_{\alpha \beta}(\bm{\zeta}^\varepsilon_\kappa - \bm{\zeta}^\varepsilon)\sqrt{a} \dd y\\
&=-\dfrac{c_0 c_e}{\kappa\sqrt{a_0}}\int_{\omega} \bm{\beta}(\bm{\zeta}^\varepsilon_\kappa) \cdot (\bm{\zeta}^\varepsilon_\kappa - \bm{\zeta}^\varepsilon) \dd y\\
&\quad+\dfrac{c_0 c_e}{\varepsilon\sqrt{a_0}} \int_{\omega} p^{i, \varepsilon} (\zeta^\varepsilon_{\kappa,i}-\zeta^\varepsilon_i) \sqrt{a} \dd y\\
&\quad-\dfrac{c_0 c_e}{\sqrt{a_0}}\int_{\omega} a^{\alpha\beta\sigma\tau}\gamma_{\sigma\tau}(\bm{\zeta}^\varepsilon)\gamma_{\alpha \beta}(\bm{\zeta}^\varepsilon_\kappa - \bm{\zeta}^\varepsilon)\sqrt{a} \dd y\\
&\le \dfrac{c_0 c_e}{\varepsilon\sqrt{a_0}} \int_{\omega} p^{i, \varepsilon} (\zeta^\varepsilon_{\kappa,i}-\zeta^\varepsilon_i) \sqrt{a} \dd y\\
&\quad-\dfrac{c_0 c_e}{\sqrt{a_0}}\int_{\omega} a^{\alpha\beta\sigma\tau}\gamma_{\sigma\tau}(\bm{\zeta}^\varepsilon)\gamma_{\alpha \beta}(\bm{\zeta}^\varepsilon_\kappa - \bm{\zeta}^\varepsilon)\sqrt{a} \dd y\\
&=\dfrac{c_0 c_e}{\sqrt{a_0}} \int_{\omega} p^i (\zeta^\varepsilon_{\kappa,i}-\zeta^\varepsilon_i) \sqrt{a} \dd y\\
&\quad-\dfrac{c_0 c_e}{\sqrt{a_0}}\int_{\omega} a^{\alpha\beta\sigma\tau}\gamma_{\sigma\tau}(\bm{\zeta}^\varepsilon)\gamma_{\alpha \beta}(\bm{\zeta}^\varepsilon_\kappa - \bm{\zeta}^\varepsilon)\sqrt{a} \dd y\to 0,
\end{align*}
as $\kappa \to 0^+$. Observe that the latter term is bounded independently of $\varepsilon$ and $\kappa$. In conclusion, we have been able to establish the strong convergence:
\begin{equation}
\label{beta-5}
\bm{\zeta}^\varepsilon_\kappa \to \bm{\zeta}^\varepsilon,\quad\textup{ in } \bm{V}_M(\omega) \textup{ as } \kappa \to0^+.
\end{equation}

Specializing $(\bm{\eta}-\bm{\zeta}^\varepsilon_\kappa)\in\bm{V}_M(\omega)$ in the variational equations of Problem~\ref{problem2}, with $\bm{\eta}\in\bm{U}_M(\omega)$, the monotonicity of $\bm{\beta}$, the convergence~\eqref{beta-3} and the convergence~\eqref{beta-5} immediately give that the limit $\bm{\zeta}^\varepsilon$ satisfies the variational inequalities in Problem~\ref{problem1}. This completes the proof.
\end{proof}

We observe that in the proof of Theorem~\ref{ex-un-kappa}, we established that $\|\bm{\zeta}^\varepsilon_\kappa-\bm{\zeta}^\varepsilon\|_{\bm{V}_M(\omega)}$ converges to zero as $\kappa\to0^+$. For the purpose of constructing a convergent numerical scheme for approximating the solution of the variational inequalities in Problem~\eqref{problem1}, we need to establish \emph{how fast} the latter norm converges to zero as $\kappa\to0^+$.
In order to establish this property, we need to prove a preparatory result concerning the augmentation of regularity of th solution of Problem~\ref{problem2} by resorting to the finite difference quotients approach originally proposed by Agmon, Douglis \& Nirenberg~\cite{AgmDouNir1959,AgmDouNir1964}, as well as the approach proposed by Frehse~\cite{Frehse1971} for variational inequalities, that was later on generalized in~\cite{Pie-2022-interior,Pie2020-1}.

Recalling that $\bm{\zeta}^\varepsilon_\kappa$ denotes the solution of Problem~\ref{problem2}, in the same spirit as Theorem~4.5-1(b) of~\cite{Ciarlet2000} we define
$$
n^{\alpha\beta,\varepsilon}_\kappa:=\varepsilon a^{\alpha\beta\sigma\tau}\gamma_{\sigma\tau}(\bm{\zeta}^\varepsilon_\kappa),
$$
and we also define
$$
n^{\alpha\beta,\varepsilon}_\kappa|_{\sigma}:=\partial_\sigma n^{\alpha\beta,\varepsilon}_\kappa+\Gamma^\alpha_{\sigma\tau}n^{\beta\tau,\varepsilon}_\kappa+\Gamma^\beta_{\sigma\tau}n^{\alpha\tau,\varepsilon}_\kappa.
$$

If the solution $\bm{\zeta}^\varepsilon_\kappa$ of Problem~\ref{problem2} is smooth enough, then it is immediate to see that it satisfies the following boundary value problem:
\begin{equation}
\label{BVP}
\begin{cases}
	-n^{\alpha\beta,\varepsilon}_\kappa|_{\beta}+\dfrac{\varepsilon}{\kappa\sqrt{a}}\beta_\alpha(\bm{\zeta}^\varepsilon_\kappa)&=p^{\alpha,\varepsilon},\textup{ in }\omega,\\
	\\
	-b_{\alpha\beta}n^{\alpha\beta,\varepsilon}_\kappa+\dfrac{\varepsilon}{\kappa\sqrt{a}}\beta_3(\bm{\zeta}^\varepsilon_\kappa)&=p^{3,\varepsilon},\textup{ in }\omega,\\
	\\
	\zeta^\varepsilon_{\kappa,\alpha}=0,\textup{ on }\gamma.
\end{cases}
\end{equation}

\section{Augmentation of the regularity of the solution of Problem~\ref{problem2}}
\label{sec:aug-interior}

Let $\omega_0\subset \omega$ and $\omega_1 \subset \omega$ be such that
\begin{equation}
\label{sets}
\omega_1 \subset\subset \omega_0 \subset \subset \omega.
\end{equation}

Let $\varphi_1 \in \mathcal{D}(\omega)$ be such that 
$$
\text{supp }\varphi_1 \subset\subset \omega_1 \textup{ and } 0\le \varphi_1 \le 1.
$$

By the definition of the symbol $\subset\subset$ in~\eqref{sets}, we obtain that the quantity
\begin{equation}
\label{d}
d=d(\varphi_1):=\dfrac{1}{2}\min\{\textup{dist}(\partial\omega_1,\partial\omega_0),\textup{dist}(\partial\omega_0,\gamma),\textup{dist}(\textup{supp }\varphi_1,\partial\omega_1)\}
\end{equation}
is strictly greater than zero. 

Denote by $D_{\rho h}$ the first order (forward) finite difference quotient of either a function or a vector field in the canonical direction $\bm{e}_\rho$ of $\mathbb{R}^2$ and with increment size $0<h<d$ sufficiently small. We can regard the first order (forward) finite difference quotient of a function as a linear operator defined as follows:
$$
D_{\rho h}: L^2(\omega) \to L^2(\omega_0).
$$

The first order finite difference quotient of a function $\xi$ in the canonical direction $\bm{e}_\rho$ of $\mathbb{R}^2$ and with increment size $0<h<d$ is defined by:
$$
D_{\rho h}\xi(y):=\dfrac{\xi(y+h\bm{e}_\rho)-\xi(y)}{h},
$$
for all (or, possibly, a.a.) $y\in\omega$ such that $(y+h\bm{e}_\rho)\in\omega$.

The first order finite difference quotient of a vector field $\bm{\xi}=(\xi_i)$ in the canonical direction $\bm{e}_\rho$ of $\mathbb{R}^2$ and with increment size $0<h<d$ is defined by
$$
D_{\rho h}\bm{\xi}(y):=\dfrac{\bm{\xi}(y+h\bm{e}_\rho)-\bm{\xi}(y)}{h},
$$
or, equivalently,
$$
D_{\rho h}\bm{\xi}(y)=(D_{\rho h}\xi_i(y)).
$$

Similarly, we can show that the first order (forward) finite difference quotient of a vector field is a linear operator from $\bm{L}^2(\omega)$ to $\bm{L}^2(\omega_0)$. 

We define the second order finite difference quotient of a function $\xi$ in the canonical direction $\bm{e}_\rho$ of $\mathbb{R}^2$ and with increment size $0<h<d$ by
$$
\delta_{\rho h}\xi(y):=\dfrac{\xi(y+h \bm{e}_\rho)-2 \xi(y)+\xi(y-h \bm{e}_\rho)}{h^2},
$$
for all (or, possibly, a.a.) $y \in \omega$ such that $(y\pm h\bm{e}_\rho) \in \omega$.

The second order finite difference quotient of a vector field $\bm{\xi}=(\xi_i)$ in the canonical direction $\bm{e}_\rho$ of $\mathbb{R}^2$ and with increment size $0<h<d$ is defined by
$$
\delta_{\rho h}\bm{\xi}(y):=\left(\dfrac{\xi_i(y+h \bm{e}_\rho)-2 \xi_i(y)+\xi_i(y-h \bm{e}_\rho)}{h^2}\right),
$$
for all (or, possibly, a.a.) $y \in \omega$ such that $(y\pm h\bm{e}_\rho) \in \omega$.

Define, following page~293 of~\cite{Evans2010}, the mapping $D_{-\rho h}:L^2(\omega) \to L^2(\omega_0)$ by
$$
D_{-\rho h}\xi(y):=\dfrac{\xi(y)-\xi(y-h\bm{e}_\rho)}{h},
$$
as well as the mapping $D_{-\rho h}:\bm{L}^2(\omega) \to \bm{L}^2(\omega_0)$ by
$$
D_{-\rho h}\bm{\xi}(y):=\dfrac{\bm{\xi}(y)-\bm{\xi}(y-h\bm{e}_\rho)}{h}.
$$

Note in passing that the second order finite difference quotient of a function $\xi$ can be expressed in terms of the first order finite difference quotient via the following identity:
\begin{equation*}
\label{ide}
\delta_{\rho h} \xi=D_{-\rho h} D_{\rho h} \xi.
\end{equation*}

Similarly, the second order finite difference quotient of a vector field $\bm{\xi}=(\xi_i)$ can be expressed in terms of the first order finite difference quotient via the following identity:
\begin{equation*}
\label{ide2}
\delta_{\rho h} \bm{\xi}=D_{-\rho h} D_{\rho h} \bm{\xi}.
\end{equation*}

	Let us define the translation operator $E$ in the canonical direction $\bm{e}_\rho$ of $\mathbb{R}^2$ and with increment size $0<h<d$ for a smooth enough function $v:\omega_0 \to \mathbb{R}$ by
\begin{align*}
E_{\rho h} v(y)&:=v(y+h \bm{e}_\rho),\\
E_{-\rho h} v(y)&:=v(y-h \bm{e}_\rho).
\end{align*}

Moreover, the following identities can be easily checked out (cf.\,\cite{Frehse1971} and~\cite{Pie2020-1}):
\begin{align}
D_{\rho h}(v w)&=(E_{\rho h} w) (D_{\rho h} v)+v D_{\rho h} w, \label{D+}\\
D_{-\rho h}(v w)&=(E_{-\rho h} w )(D_{-\rho h} v)+v D_{-\rho h} w, \label{D-}\\
\delta_{\rho h}(vw)&=w \delta_{\rho h} v+(D_{\rho h}w )(D_{\rho h} v) +(D_{-\rho h}w)( D_{-\rho h} v)+v\delta_{\rho h}w.\label{delta+}
\end{align}



We observe that the following properties hold for finite difference quotients.

The proof of the first lemma can be found in Lemma~4 of~\cite{Pie-2022-interior} and for this reason it is omitted.
\begin{lemma}
	\label{lem:fdq-1}
	Let $\{v_k\}_{k\ge1}$ be a sequence in $\mathcal{C}^1(\overline{\omega})$ that converges to a certain element $v \in H^1(\omega)$ with respect to the norm $\|\cdot\|_{H^1(\omega)}$.
	Then, we have that for all $0<h<d$ and all $\rho\in\{1,2\}$, 
	\begin{equation*}
	\label{conv1}
	D_{\rho h} v\in H^1(\omega_0) \textup{ with }\partial_\alpha(D_{\rho h} v)=D_{\rho h} (\partial_\alpha v) \quad\textup{ and }\quad D_{\rho h} v_k\to D_{\rho h} v \textup{ in }H^1(\omega_0) \textup{ as } k\to\infty.
	\end{equation*}
	\qed
\end{lemma}

As a direct consequence of Lemma~\ref{lem:fdq-1}, if $\{v_k\}_{k\ge1}$ is a sequence in $\mathcal{C}^1(\overline{\omega})$ that converges to a certain element $v \in H^1(\omega)$ with respect to the norm $\|\cdot\|_{H^1(\omega)}$, then, we have that for all $0<h<d$ and all $\rho\in\{1,2\}$, 
\begin{equation*}
\label{conv1-delta}
\delta_{\rho h} v\in H^1(\omega_1) \textup{ with }\partial_\alpha(\delta_{\rho h} v)=\delta_{\rho h} (\partial_\alpha v) \quad\textup{ and }\quad \delta_{\rho h} v_k\to \delta_{\rho h} v \textup{ in }H^1(\omega_1) \textup{ as } k\to\infty.
\end{equation*}

We also state the following elementary lemma, which exploits the compactness of the support of the test function $\varphi_1$ defined beforehand.
\begin{lemma}
\label{fdq-neg-part}
Let $f \in\mathcal{D}(\omega)$ with $\textup{supp }f \subset\subset \omega_1$.
Let $0<h<d$, where $d>0$ has been defined in~\eqref{d} and let $\rho\in\{1,2\}$ be given. Then,
\begin{equation*}
\int_{\omega} D_{\rho h}(-f^{-}) D_{\rho h}(f^{+}) \dd y \ge 0.
\end{equation*}
\end{lemma}
\begin{proof}
By the definition of $D_{\rho h}$ and the definition of the positive and negative part of a function, we have that
\begin{equation*}
\int_{\omega} D_{\rho h}(-f^{-}) D_{\rho h}(f^{+}) \dd y
=-\int_{\omega} \left(\dfrac{f^{-}(y+h\bm{e}_\rho)-f^{-}(y)}{h}\right) \left(\dfrac{f^{+}(y+h\bm{e}_\rho)-f^{+}(y)}{h}\right) \dd y.
\end{equation*}

If $y\in \omega$ is such that $f(y+h\bm{e}_\rho)>0$ and $f(y)>0$ then the integrand (i.e., the argument of the integral under consideration) of interest is equal to zero.

If $y\in \omega$ is such that $f(y+h\bm{e}_\rho)<0$ and $f(y)<0$ then the integrand (i.e., the argument of the integral under consideration) of interest is equal to zero.

If $y\in \omega$ is such that $f(y+h\bm{e}_\rho)>0$ and $f(y)<0$ then the integrand (i.e., the argument of the integral under consideration) of interest becomes equal to
\begin{equation*}
- \left(-\dfrac{\{f(y)\}^{-}}{h}\right) \left(\dfrac{\{f(y+h\bm{e}_\rho)\}^{+}}{h}\right) > 0.
\end{equation*}

If $y\in \omega$ is such that $f(y+h\bm{e}_\rho)<0$ and $f(y)>0$ then the integrand (i.e., the argument of the integral under consideration) of interest becomes equal to
\begin{equation*}
-\left(\dfrac{\{f(y+h\bm{e}_\rho)\}^{-}}{h}\right) \left(-\dfrac{\{f(y)\}^{+}}{h}\right) > 0.
\end{equation*}

In conclusion, the integrand is never negative and the integral under examination is always greater or equal than zero, as it was to be proved.
\end{proof}

Let us recall that $\bm{\theta}(y) \cdot \bm{q}>0$ for all $y \in \overline{\omega}$ (Lemma~\ref{lem0}), where the unit-norm vector $\bm{q}$ is given. In view of this, we wonder whether the immersion $\bm{\theta} \in \mathcal{C}^3(\overline{\omega};\mathbb{E}^3)$ admits a prolongation $\tilde{\bm{\theta}}\in\mathcal{C}^3(\overline{\tilde{\omega}};\mathbb{E}^3)$, for some domain $\omega \subset \subset \tilde{\omega}$, prolongation which is associated with the natural covariant and contravariant bases $\{\tilde{\bm{a}}_1, \tilde{\bm{a}}_2, \tilde{\bm{a}}_3\}$ and $\{\tilde{\bm{a}}^1, \tilde{\bm{a}}^2, \tilde{\bm{a}}^3\}$ and which enjoys the following properties:
\begin{itemize}
	\item[(a)] The mapping $\tilde{\bm{\theta}} \in \mathcal{C}^3(\overline{\tilde{\omega}};\mathbb{E}^3)$ is an immersion and $\tilde{\bm{\theta}}\big|_{\overline{\omega}}=\bm{\theta}$;
	\item[(b)] The surface $\tilde{\bm{\theta}}(\overline{\tilde{\omega}})$ is elliptic;
	\item[(c)] If $\min_{y \in \overline{\omega}}(\bm{\theta}(y) \cdot\bm{q}) >0$ then $\min_{y \in \overline{\tilde{\omega}}}(\tilde{\bm{\theta}}(y) \cdot\bm{q}) >0$;
	\item[(d)] If $\min_{y \in \overline{\omega}} (\bm{a}^3(y) \cdot\bm{q})>0$ then $\min_{y \in \overline{\tilde{\omega}}} (\tilde{\bm{a}}^3(y) \cdot\bm{q})>0$.
\end{itemize}

We will say that $\bm{\theta}$ satisfies the ``prolongation property'' if there exists an extension $\tilde{\bm{\theta}}$ satisfying the properties (a)--(d) above.

Thanks to the Whitney's extension theorem (cf., e.g., Theorem~2.3.6 of~\cite{Hormander1990}), we are able to give a \emph{constructive proof} of the fact that the ``prolongation property'' is satisfied by all the elliptic surfaces satisfying the sufficient condition ensuring the ``density property'', thus giving an affirmative answer to the question posed above.

\begin{lemma}
	\label{geometry}
	Let $\omega \subset \mathbb{R}^2$ be a domain and let $\bm{\vartheta} \in \mathcal{C}^2(\overline{\omega};\mathbb{E}^3)$ be an immersion associated with an elliptic surface and satisfying the sufficient condition ensuring the ``density property''. Then $\bm{\vartheta}$ satisfies the ``prolongation property''.
\end{lemma}
\begin{proof}
	Let $\{\bm{e}_i\}_{i=1}^3$ be an orthonormal covariant basis for the Euclidean space $\mathbb{E}^3$. Let $\{\bm{e}^i\}_{i=1}^3$ denote the corresponding contravariant basis of the Euclidean space $\mathbb{E}^3$, and recall that $\bm{e}_i=\bm{e}^i$ for all $1 \le i \le 3$.
	For each $y \in \overline{\omega}$, we can write $\bm{\vartheta}(y)=\vartheta_i(y) \bm{e}^i$. Therefore, each of the components $\vartheta_i$, $1 \le i \le 3$, of the immersion $\bm{\vartheta}$ is clearly of class $\mathcal{C}^2(\overline{\omega})$ since $\vartheta_j=\bm{\vartheta} \cdot \bm{e}_j$, for all $1\le j \le 3$ and the right hand side is of class $\mathcal{C}^2(\overline{\omega})$.
	
	By the Whitney extension theorem (cf., e.g., Theorem~2.3.6 of~\cite{Hormander1990}), for each $1 \le i \le 3$, there exists a function $\tilde{\vartheta}_i \in \mathcal{C}^2(\mathbb{R}^2)$ that extends $\vartheta_i$.
	We can thus define a mapping $\tilde{\bm{\vartheta}}:=\tilde{\vartheta}_i \bm{e}^i \in \mathcal{C}^2(\overline{\tilde{\omega}};\mathbb{E}^3)$ that extends $\bm{\vartheta}$, for all $\tilde{\omega}\supset\supset\omega$.
	
	Observe that the covariant basis $\{\bm{a}_i\}_{i=1}^3$ associated with $\bm{\vartheta}$ satisfies 
	$$
	\det(a_{\alpha\beta}(y))>0,\quad\textup{ for all }y\in \overline{\omega},
	$$
	since $\bm{\vartheta}$ is assumed to be an immersion. Let $\{\tilde{\bm{a}}_i\}_{i=1}^3$ denote the covariant basis of the extension $\tilde{\bm{\vartheta}}$.
	By the continuity of the determinant and the properties of the prolongation $\tilde{\bm{\vartheta}}$ with obvious meaning of the notation we have that, up to shrinking $\tilde{\omega}$:
	$$
	\det(\tilde{a}_{\alpha\beta}(y))>0,\quad\textup{ for all }y\in \overline{\tilde{\omega}},
	$$
	and property (a) is thus established.
	
	Recall that the Gaussian curvature $\kappa$ of the immersion $\bm{\vartheta}$ is defined at each $y\in\overline{\omega}$ by
	\begin{equation*}
		K(y)=\det(b_\alpha^\beta(y)),
	\end{equation*}
	namely, in terms of the invariants of the matrix associated with the mixed components of the second fundamental form of $\bm{\vartheta}$. Let $\tilde{K}$ denote the Gaussian curvature associated with the extension $\tilde{\bm{\vartheta}}$ and observe that $\tilde{K} \in \mathcal{C}^2(\mathbb{R})$, and that $\tilde{K}(y)=K(y)$, for all $y \in \overline{\omega}$.
	
	By the continuity of the mixed components of the second fundamental form (recall that $\bm{\vartheta}$ was assumed to be of class $\mathcal{C}^2(\overline{\omega};\mathbb{E}^3)$) we can thus find a set $\tilde{\omega} \supset\supset\omega$ such that $\tilde{K}>0$ in $\overline{\tilde{\omega}}$. This proves property (b).
	
	Properties (c) and (d) also a direct consequence of the continuity of $\tilde{\bm{\vartheta}}$. 
	
	Up to shrinking $\tilde{\omega}$, we can affirm without loss of generality that the restriction of the mapping $\tilde{\bm{\vartheta}}$ to the set $\tilde{\omega}$ is the sought prolongation of the given immersion $\bm{\vartheta}$, that satisfies properties (a)--(d) of the ``prolongation property''. This completes the proof.
\end{proof}

We are ready to state the main result of this section, that constitutes the first new result in this paper. Note in passing, upon proving the following theorem, we will be able to obtain the conclusion of Theorem~6 in~\cite{Pie-2022-interior} under weaker assumptions on the given term $\bm{p}^\varepsilon$. The main novelty of the approach presented in this paper is that the proof of the augmented regularity of the solution of Problem~\ref{problem1} will be established \emph{without} resorting to the ``density property'' exploited for establishing Theorem~\ref{density}.

\begin{theorem}
\label{aug:int}
Let $\omega_0$ and $\omega_1$ be as in~\eqref{sets}. Assume that there exists a unit norm vector $\bm{q} \in \mathbb{E}^3$ such that
\begin{equation*}
\min_{y \in \overline{\omega}} (\bm{\theta} (y) \cdot \bm{q}) > 0
\textup{ and }
\min_{y \in \overline{\omega}} (\bm{a}_3 (y) \cdot \bm{q}) > 0.
\end{equation*}

Assume also that the vector field $\bm{f}^\varepsilon=(f^{i,\varepsilon})$ defining the applied body force density is of class $L^2(\Omega^\varepsilon)\times L^2(\Omega^\varepsilon)\times H^1(\Omega^\varepsilon)$.
Then, the solution $\bm{\zeta}^\varepsilon_\kappa=(\zeta^\varepsilon_{\kappa,i})$ of Problem~\ref{problem2} is of class $\bm{V}_M(\omega)\cap H^2_{\textup{loc}}(\omega) \times H^2_{\textup{loc}}(\omega) \times H^1_{\textup{loc}}(\omega)$.
\end{theorem}
\begin{proof}
Fix $\varphi\in\mathcal{D}(\omega)$ such that $\text{supp }\varphi \subset\subset \omega_1$ and $0\le \varphi \le 1$. Let $\bm{\zeta}^\varepsilon_\kappa \in \bm{V}_M(\omega)$ be the unique solution of Problem~\ref{problem2}.
Observe that the transverse component $\zeta^\varepsilon_{\kappa,3}$ can be extended outside of $\omega$ by zero, preserving the $L^2(\mathbb{R}^2)$ regularity.
For what concerns the tangential components $\zeta^\varepsilon_{\kappa,\alpha}$, Proposition~9.18 of~\cite{Brez11} states that the only admissible prolongation outside of $\omega$ is the prolongation by zero. Therefore, it makes sense to consider the vector field
\begin{equation*}
(-\varphi \delta_{\rho h}(\varphi\bm{\zeta}^\varepsilon_\kappa)) \in H^1(\mathbb{R}^2) \times H^1(\mathbb{R}^2) \times L^2(\mathbb{R}^2).
\end{equation*}

Since the support of this vector field is compactly contained in $\omega_1$, we obtain that, actually,
\begin{equation*}
(-\varphi \delta_{\rho h}(\varphi\bm{\zeta}^\varepsilon_\kappa)) \in \bm{V}_M(\omega),
\end{equation*}
and we can specialize $\bm{\eta}=-\varphi \delta_{\rho h}(\varphi\bm{\zeta}^\varepsilon_\kappa)$ in the variational equations of Problem~\ref{problem2}.

Let us now evaluate
\begin{equation*}
\begin{aligned}
&\int_{\omega} p^{i,\varepsilon} (-\varphi \delta_{\rho h}(\varphi\zeta^\varepsilon_{\kappa,i})) \sqrt{a} \dd y
=-\int_{\omega_1}(\varphi p^{i,\varepsilon}) (\delta_{\rho h}(\varphi\zeta^\varepsilon_{\kappa,i})) \sqrt{a} \dd y\\
&=-\int_{\omega_1}(\varphi p^{\alpha,\varepsilon}) (\delta_{\rho h}(\varphi\zeta^\varepsilon_{\kappa,\alpha})) \sqrt{a} \dd y
+\int_{\omega}D_{\rho h}(\varphi p^{3,\varepsilon}) (D_{\rho h}(\varphi\zeta^\varepsilon_{\kappa,3})) \sqrt{a} \dd y\\
&=\varepsilon \|\varphi\|_{\mathcal{C}^1(\overline{\omega})}\|\bm{p}\|_{L^2(\omega)\times L^2(\omega) \times H^1(\omega)}\sqrt{a_1}
\|D_{\rho h}(\varphi \bm{\zeta}^\varepsilon_\kappa)\|_{H^1(\omega_1)\times H^1(\omega_1)\times L^2(\omega_1)},
\end{aligned}
\end{equation*}
where the second holds thanks to the integration by parts formula for finite difference quotients (cf. page~293 of~\cite{Evans2010}), and the inequality holds thanks to the H\"older inequality. Note in passing that $\|\bm{p}\|_{L^2(\omega)\times L^2(\omega) \times H^1(\omega)}$ is independent of $\varepsilon$ thanks to the assumptions on the data.

Thanks to these inequalities, we have that
\begin{equation}
\label{int-1}
\begin{aligned}
&\varepsilon\int_{\omega_1} a^{\alpha\beta\sigma\tau} \gamma_{\sigma\tau}(\bm{\zeta}^\varepsilon_\kappa) \gamma_{\alpha \beta}(-\varphi \delta_{\rho h}(\varphi\bm{\zeta}^\varepsilon_\kappa)) \sqrt{a} \dd y
+\dfrac{\varepsilon}{\kappa}\int_{\omega_1}\bm{\beta}(\bm{\zeta}^\varepsilon_\kappa) \cdot (-\varphi \delta_{\rho h}(\varphi\bm{\zeta}^\varepsilon_\kappa)) \dd y\\
&\le \varepsilon \|\varphi\|_{\mathcal{C}^1(\overline{\omega})}\|\bm{p}\|_{L^2(\omega)\times L^2(\omega) \times H^1(\omega)}\sqrt{a_1}
\|D_{\rho h}(\varphi \bm{\zeta}^\varepsilon_\kappa)\|_{H^1(\omega_1)\times H^1(\omega_1)\times L^2(\omega_1)}.
\end{aligned}
\end{equation}

The first step in our analysis consists in showing that:
\begin{equation}
\label{key-relation-2}
\begin{aligned}
&-\varepsilon \int_{\omega_1}a^{\alpha\beta\sigma\tau}\gamma_{\sigma \tau}(\varphi\bm{\zeta}^\varepsilon_\kappa)\gamma_{\alpha\beta}(\delta_{\rho h}(\varphi \bm{\zeta}^\varepsilon_\kappa))\sqrt{a} \dd y\\
&\le -\varepsilon \int_{\omega_1} a^{\alpha\beta\sigma\tau}\gamma_{\sigma \tau}(\bm{\zeta}^\varepsilon_\kappa)\gamma_{\alpha\beta}(\varphi \delta_{\rho h}(\varphi \bm{\zeta}^\varepsilon_\kappa))\sqrt{a} \dd y+C\varepsilon(1+\|D_{\rho h}(\varphi \bm{\zeta}^\varepsilon_\kappa)\|_{H^1(\omega_1)\times H^1(\omega_1)\times L^2(\omega_1)}),
\end{aligned}
\end{equation}
for some $C>0$ independent of $\varepsilon$, $\kappa$ and $h$.

Recalling the definition of the change of metric tensor components $\gamma_{\alpha \beta}$ (cf. section~\ref{sec1}) and recalling that $\bm{\theta} \in \mathcal{C}^3(\overline{\omega};\mathbb{E}^3)$, we have that the integral
\begin{equation*}
-\varepsilon\int_{\omega_1} a^{\alpha\beta\sigma\tau}\gamma_{\sigma \tau}(\varphi\bm{\zeta}^\varepsilon_\kappa)\gamma_{\alpha\beta}(\delta_{\rho h}(\varphi \bm{\zeta}^\varepsilon_\kappa)) \sqrt{a} \dd y
\end{equation*}
can be estimated by estimating the following main nine addends of its. In the evaluation of the following nine terms, the indices are assumed to be fixed, i.e., the summation rule with respect to repeated indices is not enforced in~\eqref{term-1}--\eqref{term-9} below.

Overall, the strategy we resort to is the following: we take into accounts the addends of the linearised change of metric tensor and we apply Green's formula and the integration-by-parts formula for finite difference quotients for suitably arranging the position of the compactly supported function $\varphi$.

First, thanks to an application of Green's formula (cf., e.g., Theorem~6.6-7 of~\cite{PGCLNFAA}), we estimate:
\begin{equation}
\label{term-1}
\begin{aligned}
&\int_{\omega_1} -a^{\alpha\beta\sigma\tau} \partial_\sigma(\varphi \zeta^\varepsilon_{\kappa,\tau}) \partial_\beta(\delta_{\rho h}(\varphi \zeta^\varepsilon_{\kappa,\alpha})) \sqrt{a} \dd y\\
&=\int_{\omega_1} -a^{\alpha\beta\sigma\tau} [(\partial_\sigma \varphi)\zeta^\varepsilon_{\kappa,\tau} +\varphi \partial_\sigma \zeta^\varepsilon_{\kappa,\tau}] \partial_\beta(\delta_{\rho h}(\varphi \zeta^\varepsilon_{\kappa,\alpha})) \sqrt{a} \dd y\\
&=\int_{\omega_1} \partial_\beta(a^{\alpha\beta\sigma\tau} (\partial_\sigma \varphi) \zeta^\varepsilon_{\kappa,\tau} \sqrt{a}) \delta_{\rho h}(\varphi \zeta^\varepsilon_{\kappa,\alpha})  \dd y\\
&\quad+\int_{\omega_1} -a^{\alpha\beta\sigma\tau} \partial_\sigma \zeta^\varepsilon_{\kappa,\tau} [\varphi\partial_\beta(\delta_{\rho h}(\varphi \zeta^\varepsilon_{\kappa,\alpha}))] \sqrt{a} \dd y\\
&\le C \|\zeta^\varepsilon_{\kappa,\tau}\|_{H^1(\omega_1) \times H^1(\omega_1) \times L^2(\omega_1)} \|D_{\rho h}(\varphi \zeta^\varepsilon_{\kappa,\alpha})\|_{H^1(\omega_1)}\\
&\quad+\int_{\omega_1} -a^{\alpha\beta\sigma\tau} \partial_\sigma \zeta^\varepsilon_{\kappa,\tau} \partial_\beta(\varphi \delta_{\rho h}(\varphi \zeta^\varepsilon_{\kappa,\alpha})) \sqrt{a} \dd y\\
&\quad +  \int_{\omega_1} a^{\alpha\beta\sigma\tau} (\partial_\sigma \zeta^\varepsilon_{\kappa,\tau}) (\partial_\beta \varphi) \delta_{\rho h}(\varphi \zeta^\varepsilon_{\kappa,\alpha}) \sqrt{a} \dd y\\
&\le \int_{\omega_1} -a^{\alpha\beta\sigma\tau} (\partial_\sigma \zeta^\varepsilon_{\kappa,\tau}) \partial_\beta(\varphi\delta_{\rho h}(\varphi \zeta^\varepsilon_{\kappa,\alpha})) \sqrt{a} \dd y +C \|D_{\rho h}(\varphi \zeta^\varepsilon_{\kappa,\alpha})\|_{H^1(\omega_1)}.
\end{aligned}
\end{equation}

Second, we estimate:
\begin{equation}
\label{term-2}
\begin{aligned}
&\int_{\omega_1} -a^{\alpha\beta\sigma\tau} (-\Gamma_{\sigma\tau}^\varsigma \varphi \zeta^\varepsilon_{\kappa,\varsigma}) \partial_\beta(\delta_{\rho h}(\varphi \zeta^\varepsilon_{\kappa,\alpha})) \sqrt{a} \dd y\\
&=\int_{\omega_1} a^{\alpha\beta\sigma\tau} \Gamma_{\sigma\tau}^\upsilon \zeta^\varepsilon_{\kappa,\upsilon} \partial_\beta(\varphi \delta_{\rho h}(\varphi \zeta^\varepsilon_{\kappa,\alpha})) \sqrt{a} \dd y\\
&\quad-\int_{\omega_1} a^{\alpha\beta\sigma\tau} \Gamma_{\sigma\tau}^\upsilon \zeta^\varepsilon_{\kappa,\upsilon} (\partial_\beta\varphi) (\delta_{\rho h}(\varphi \zeta^\varepsilon_{\kappa,\alpha})) \sqrt{a} \dd y\\
&\le \int_{\omega_1} a^{\alpha\beta\sigma\tau} \Gamma_{\sigma\tau}^\upsilon \zeta^\varepsilon_{\kappa,\upsilon} \partial_\beta(\varphi \delta_{\rho h}(\varphi \zeta^\varepsilon_{\kappa,\alpha})) \sqrt{a} \dd y +C \|D_{\rho h}(\varphi \zeta^\varepsilon_{\kappa,\alpha})\|_{H^1(\omega_1)},
\end{aligned}
\end{equation}
where the equality holds as a consequence of Green's formula.

Third, we estimate:
\begin{equation}
\label{term-3}
\begin{aligned}
&\int_{\omega_1} -a^{\alpha\beta\sigma\tau} (-b_{\alpha\beta}\varphi \zeta^\varepsilon_{\kappa,3}) \partial_\beta(\delta_{\rho h}(\varphi \zeta^\varepsilon_{\kappa,\alpha})) \sqrt{a} \dd y\\
&=\int_{\omega_1} a^{\alpha\beta\sigma\tau} (b_{\alpha\beta} \zeta^\varepsilon_{\kappa,3}) \partial_\beta(\varphi \delta_{\rho h}(\varphi \zeta^\varepsilon_{\kappa,\alpha})) \sqrt{a} \dd y\\
&\quad-\int_{\omega_1} a^{\alpha\beta\sigma\tau} b_{\alpha\beta} \zeta^\varepsilon_{\kappa,3} (\partial_\beta \varphi) \delta_{\rho h}(\varphi \zeta^\varepsilon_{\kappa,\alpha}) \sqrt{a} \dd y\\
&\le \int_{\omega_1} a^{\alpha\beta\sigma\tau} (b_{\alpha\beta} \zeta^\varepsilon_{\kappa,3}) \partial_\beta(\varphi \delta_{\rho h}(\varphi \zeta^\varepsilon_{\kappa,\alpha})) \sqrt{a} \dd y +C \|D_{\rho h}(\varphi \zeta^\varepsilon_{\kappa,\alpha})\|_{H^1(\omega_1)}.
\end{aligned}
\end{equation}

Fourth, we estimate:
\begin{equation}
\label{term-4}
\begin{aligned}
&\int_{\omega_1} -a^{\alpha\beta\sigma\tau} \partial_\sigma(\varphi \zeta^\varepsilon_{\kappa,\tau}) \Gamma_{\alpha\beta}^\upsilon \delta_{\rho h}(\varphi \zeta^\varepsilon_{\kappa,\upsilon}) \sqrt{a} \dd y\\
&=\int_{\mathcal{U}_1} -a^{\alpha\beta\sigma\tau} (\partial_\sigma\varphi) \zeta^\varepsilon_{\kappa,\tau} \Gamma_{\alpha\beta}^\upsilon \delta_{\rho h}(\varphi \zeta^\varepsilon_{\kappa,\upsilon}) \sqrt{a} \dd y\\
&\quad+\int_{\omega_1} -a^{\alpha\beta\sigma\tau} \varphi (\partial_\sigma \zeta^\varepsilon_{\kappa,\tau}) \Gamma_{\alpha\beta}^\upsilon \delta_{\rho h}(\varphi \zeta^\varepsilon_{\kappa,\upsilon}) \sqrt{a} \dd y\\
&\le C \|D_{\rho h}(\varphi \zeta^\varepsilon_{\kappa,\alpha})\|_{H^1(\omega_1)} 
+\int_{\omega_1} -a^{\alpha\beta\sigma\tau} (\partial_\sigma \zeta^\varepsilon_{\kappa,\tau}) \Gamma_{\alpha\beta}^\upsilon [\varphi\delta_{\rho h}(\varphi \zeta^\varepsilon_{\kappa,\upsilon})] \sqrt{a} \dd y.
\end{aligned}
\end{equation}

Fifth, we straightforwardly observe that:
\begin{equation}
\label{term-5}
\begin{aligned}
&\int_{\omega_1} -a^{\alpha\beta\sigma\tau} (-\Gamma_{\sigma\tau}^\varsigma \zeta^\varepsilon_{\kappa,\varsigma} \varphi) \Gamma_{\alpha\beta}^\upsilon \delta_{\rho h}(\zeta^\varepsilon_{\kappa,\upsilon} \varphi) \sqrt{a} \dd y\\
&\le C(1+\|D_{\rho h}(\varphi \bm{\zeta}^\varepsilon_\kappa)\|_{H^1(\omega_1) \times H^1(\omega_1) \times L^2(\omega_1)})\\
&\quad+ \int_{\omega_1} -a^{\alpha\beta\sigma\tau} (-\Gamma_{\sigma\tau}^\varsigma \zeta^\varepsilon_{\kappa,\varsigma}) \Gamma_{\alpha\beta}^\upsilon [\varphi\delta_{\rho h}(\zeta^\varepsilon_{\kappa,\upsilon} \varphi)] \sqrt{a} \dd y.
\end{aligned}
\end{equation}

Sixth, we straightforwardly observe that:
\begin{equation}
\label{term-6}
\begin{aligned}
&\int_{\omega_1} -a^{\alpha\beta\sigma\tau} (b_{\sigma\tau} \zeta^\varepsilon_{\kappa,3} \varphi) \Gamma_{\alpha\beta}^\upsilon \delta_{\rho h}(\varphi \zeta^\varepsilon_{\kappa,\upsilon}) \sqrt{a} \dd y\\
&\le C(1+\|D_{\rho h}(\varphi \bm{\zeta}^\varepsilon_\kappa)\|_{H^1(\omega_1) \times H^1(\omega_1) \times L^2(\omega_1)})\\
&\quad+\int_{\omega_1} -a^{\alpha\beta\sigma\tau} b_{\sigma\tau} \zeta^\varepsilon_{\kappa,3}  \Gamma_{\alpha\beta}^\upsilon [\varphi\delta_{\rho h}(\varphi \zeta^\varepsilon_{\kappa,\upsilon})] \sqrt{a} \dd y.
\end{aligned}
\end{equation}

Seventh, we straightforwardly observe that:
\begin{equation}
\label{term-7}
\begin{aligned}
&\int_{\omega_1} -a^{\alpha\beta\sigma\tau} (-\Gamma_{\sigma\tau}^\varsigma \varphi \zeta^\varepsilon_{\kappa,\varsigma}) b_{\alpha\beta} \delta_{\rho h}(\zeta^\varepsilon_{\kappa,3} \varphi) \sqrt{a} \dd y\\
&\le C(1+\|D_{\rho h}(\varphi \bm{\zeta}^\varepsilon_\kappa)\|_{H^1(\omega_1) \times H^1(\omega_1) \times L^2(\omega_1)})\\
&\quad+\int_{\omega_1} -a^{\alpha\beta\sigma\tau} (-\Gamma_{\sigma\tau}^\varsigma \zeta^\varepsilon_{\kappa,\varsigma}) b_{\alpha\beta} [\varphi\delta_{\rho h}(\zeta^\varepsilon_{\kappa,3} \varphi)] \sqrt{a} \dd y.
\end{aligned}
\end{equation}

Eighth, we estimate:
\begin{equation}
\label{term-8}
\begin{aligned}
&\int_{\omega_1} -a^{\alpha\beta\sigma\tau} \partial_\sigma(\varphi \zeta^\varepsilon_{\kappa,\tau}) b_{\alpha\beta} \delta_{\rho h}(\zeta^\varepsilon_{\kappa,3} \varphi) \sqrt{a} \dd y\\
&=\int_{\omega_1} -a^{\alpha\beta\sigma\tau} (\partial_\sigma \zeta^\varepsilon_{\kappa,\tau}) b_{\alpha\beta} [\varphi \delta_{\rho h}(\zeta^\varepsilon_{\kappa,3} \varphi)] \sqrt{a} \dd y\\
&\quad+\int_{\omega_1} -a^{\alpha\beta\sigma\tau} (\partial_\sigma \varphi) \zeta^\varepsilon_{\kappa,\tau} b_{\alpha\beta} \delta_{\rho h}(\zeta^\varepsilon_{\kappa,3} \varphi)\sqrt{a} \dd y\\
&=\int_{\omega_1} -a^{\alpha\beta\sigma\tau} (\partial_\sigma \zeta^\varepsilon_{\kappa,\tau}) b_{\alpha\beta} [\varphi \delta_{\rho h}(\zeta^\varepsilon_{\kappa,3} \varphi)] \sqrt{a} \dd y\\
&\quad+\int_{\omega_1} D_{\rho h}(-a^{\alpha\beta\sigma\tau} (\partial_\sigma \varphi) \zeta^\varepsilon_{\kappa,\tau} b_{\alpha\beta} \sqrt{a}) D_{\rho h}(\varphi \zeta^\varepsilon_{\kappa,3})  \dd y\\
&\le \int_{\omega_1} -a^{\alpha\beta\sigma\tau} (\partial_\sigma \zeta^\varepsilon_{\kappa,\tau}) b_{\alpha\beta} [\varphi \delta_{\rho h}(\zeta^\varepsilon_{\kappa,3} \varphi)] \sqrt{a} \dd y\\
&\quad + C(1+\|D_{\rho h}(\varphi \bm{\zeta}^\varepsilon_\kappa)\|_{H^1(\omega_1) \times H^1(\omega_1) \times L^2(\omega_1)}),
\end{aligned}
\end{equation}
where in the last equality we used the integration-by-parts formula for finite difference quotients.

Ninth, and last, we straightforwardly observe that
\begin{equation}
\label{term-9}
\begin{aligned}
&\int_{\omega_1} -a^{\alpha\beta\sigma\tau} (b_{\sigma\tau} \zeta^\varepsilon_{\kappa,3} \varphi) b_{\alpha\beta} \delta_{\rho h}(\zeta^\varepsilon_{\kappa,3} \varphi) \sqrt{a} \dd y\\
&=\int_{\omega_1} -a^{\alpha\beta\sigma\tau} (b_{\sigma\tau} \zeta^\varepsilon_{\kappa,3}) b_{\alpha\beta} [\varphi\delta_{\rho h}(\zeta^\varepsilon_{\kappa,3} \varphi)] \sqrt{a} \dd y\\
&\le C(1+\|D_{\rho h}(\varphi \bm{\zeta}^\varepsilon_\kappa)\|_{H^1(\omega_1) \times H^1(\omega_1) \times L^2(\omega_1)})\\
&\quad+ \int_{\omega_1} -a^{\alpha\beta\sigma\tau} (b_{\sigma\tau} \zeta^\varepsilon_{\kappa,3}) b_{\alpha\beta} [\varphi\delta_{\rho h}(\zeta^\varepsilon_{\kappa,3} \varphi)] \sqrt{a} \dd y.
\end{aligned}
\end{equation}

In conclusion, combining~\eqref{term-1}--\eqref{term-9} together gives~\eqref{key-relation-2}. Combining ~\eqref{int-1} and~\eqref{key-relation-2} gives that there exists a constant $C>0$ independent of $\varepsilon$, $\kappa$ and $h$ such that
\begin{equation*}
\begin{aligned}
&-\varepsilon \int_{\omega_1} a^{\alpha\beta\sigma\tau} \gamma_{\sigma\tau}(\varphi\bm{\zeta}^\varepsilon_\kappa) \gamma_{\alpha \beta}(\delta_{\rho h}(\varphi\bm{\zeta}^\varepsilon_\kappa)) \sqrt{a} \dd y\\
&\quad+\dfrac{\varepsilon}{\kappa}\int_{\omega_1}\bm{\beta}(\bm{\zeta}^\varepsilon_\kappa) \cdot (-\varphi \delta_{\rho h}(\varphi\bm{\zeta}^\varepsilon_\kappa)) \dd y \le C\varepsilon(1+\|D_{\rho h}(\varphi\bm{\zeta}^\varepsilon_\kappa)\|_{H^1(\omega_1)\times H^1(\omega_1)\times L^2(\omega_1)}).
\end{aligned}
\end{equation*}

An application of the integration-by-parts formula for finite difference quotients (cf., e.g., page~293 of~\cite{Evans2010}) and~\eqref{D+} turn the latter into:
\begin{equation*}
\begin{aligned}
&\varepsilon \int_{\omega_1} a^{\alpha\beta\sigma\tau} \gamma_{\sigma\tau}(D_{\rho h}(\varphi\bm{\zeta}^\varepsilon_\kappa)) \gamma_{\alpha \beta}(D_{\rho h}(\varphi\bm{\zeta}^\varepsilon_\kappa)) \sqrt{a} \dd y
+\varepsilon \int_{\omega_1} D_{\rho h}(a^{\alpha\beta\sigma\tau}\sqrt{a}) E_{\rho h}\left(\gamma_{\sigma\tau}(\varphi\bm{\zeta}^\varepsilon_\kappa)\right) \gamma_{\alpha \beta}(D_{\rho h}(\varphi\bm{\zeta}^\varepsilon_\kappa)) \dd y\\
&\quad+\dfrac{\varepsilon}{\kappa}\int_{\omega_1}\bm{\beta}(\bm{\zeta}^\varepsilon_\kappa) \cdot (-\varphi \delta_{\rho h}(\varphi\bm{\zeta}^\varepsilon_\kappa)) \dd y\\
&=\varepsilon \int_{\omega_1} D_{\rho h}\left(a^{\alpha\beta\sigma\tau} \gamma_{\sigma\tau}(\varphi\bm{\zeta}^\varepsilon_\kappa)\sqrt{a}\right) \gamma_{\alpha \beta}(D_{\rho h}(\varphi\bm{\zeta}^\varepsilon_\kappa))  \dd y
+\dfrac{\varepsilon}{\kappa}\int_{\omega_1}\bm{\beta}(\bm{\zeta}^\varepsilon_\kappa) \cdot (-\varphi \delta_{\rho h}(\varphi\bm{\zeta}^\varepsilon_\kappa)) \dd y\\
&=-\varepsilon \int_{\omega_1} a^{\alpha\beta\sigma\tau} \gamma_{\sigma\tau}(\varphi\bm{\zeta}^\varepsilon_\kappa) D_{-\rho h}\left(\gamma_{\alpha \beta}(D_{\rho h}(\varphi\bm{\zeta}^\varepsilon_\kappa))\right) \sqrt{a} \dd y
+\dfrac{\varepsilon}{\kappa}\int_{\omega_1}\bm{\beta}(\bm{\zeta}^\varepsilon_\kappa) \cdot (-\varphi \delta_{\rho h}(\varphi\bm{\zeta}^\varepsilon_\kappa)) \dd y\\
&=-\varepsilon \int_{\omega_1} a^{\alpha\beta\sigma\tau} \gamma_{\sigma\tau}(\varphi\bm{\zeta}^\varepsilon_\kappa) \gamma_{\alpha \beta}(\delta_{\rho h}(\varphi\bm{\zeta}^\varepsilon_\kappa)) \sqrt{a} \dd y
+\dfrac{\varepsilon}{\kappa}\int_{\omega_1}\bm{\beta}(\bm{\zeta}^\varepsilon_\kappa) \cdot (-\varphi \delta_{\rho h}(\varphi\bm{\zeta}^\varepsilon_\kappa)) \dd y\\
&\le C\varepsilon(1+\|D_{\rho h}(\varphi\bm{\zeta}^\varepsilon_\kappa)\|_{H^1(\omega_1)\times H^1(\omega_1)\times L^2(\omega_1)}),
\end{aligned}
\end{equation*}
for some $C>0$ independent of $\varepsilon$, $\kappa$ and $h$.

We then have that the fact that $\varphi$ has compact support in $\omega_1$, Korn's inequality (Theorem~\ref{korn}), the definition of $d$ (viz. \eqref{d}) give
\begin{equation*}
\begin{aligned}
&\dfrac{\varepsilon\sqrt{a_0}}{c_0 c_e}\|D_{\rho h}(\varphi \bm{\zeta}^\varepsilon_\kappa)\|_{H^1(\omega_1)\times H^1(\omega_1)\times L^2(\omega_1)}^2
+\dfrac{\varepsilon}{\kappa}\int_{\omega_1}\bm{\beta}(\bm{\zeta}^\varepsilon_\kappa) \cdot (-\varphi \delta_{\rho h}(\varphi\bm{\zeta}^\varepsilon_\kappa)) \dd y\\
&\le\varepsilon \int_{\omega_1} a^{\alpha\beta\sigma\tau} \gamma_{\sigma\tau}(D_{\rho h}(\varphi\bm{\zeta}^\varepsilon_\kappa)) \gamma_{\alpha \beta}(D_{\rho h}(\varphi\bm{\zeta}^\varepsilon_\kappa)) \sqrt{a} \dd y
+\dfrac{\varepsilon}{\kappa}\int_{\omega_1}\bm{\beta}(\bm{\zeta}^\varepsilon_\kappa) \cdot (-\varphi \delta_{\rho h}(\varphi\bm{\zeta}^\varepsilon_\kappa)) \dd y\\
&\le C\varepsilon(1+\|D_{\rho h}(\varphi\bm{\zeta}^\varepsilon_\kappa)\|_{H^1(\omega_1)\times H^1(\omega_1)\times L^2(\omega_1)})
-\varepsilon \int_{\omega_1} D_{\rho h}(a^{\alpha\beta\sigma\tau}\sqrt{a}) E_{\rho h}\left(\gamma_{\sigma\tau}(\varphi\bm{\zeta}^\varepsilon_\kappa)\right) \gamma_{\alpha \beta}(D_{\rho h}(\varphi\bm{\zeta}^\varepsilon_\kappa)) \dd y\\
&\le C\varepsilon(1+\|D_{\rho h}(\varphi\bm{\zeta}^\varepsilon_\kappa)\|_{H^1(\omega_1)\times H^1(\omega_1)\times L^2(\omega_1)})\\
&\quad+\varepsilon \left(\max_{\alpha,\beta,\sigma,\tau \in \{1,2\}}\{\|a^{\alpha\beta\sigma\tau}\sqrt{a}\|_{\mathcal{C}^1(\overline{\omega})}\}\right) \left(\max_{\alpha,\beta\in\{1,2\}}\|\gamma_{\alpha\beta}(D_{\rho h}(\varphi\bm{\zeta}^\varepsilon_\kappa)\|_{L^2(\omega_1)}\right) \left(\max_{\sigma,\tau\in\{1,2\}}\|E_{\rho h}(\gamma_{\sigma\tau}(\varphi \bm{\zeta}^\varepsilon_\kappa))\|_{L^2(\omega_1)}\right)\\
&= C\varepsilon(1+\|D_{\rho h}(\varphi\bm{\zeta}^\varepsilon_\kappa)\|_{H^1(\omega_1)\times H^1(\omega_1)\times L^2(\omega_1)})\\
&\quad+\varepsilon \left(\max_{\alpha,\beta,\sigma,\tau \in \{1,2\}}\{\|a^{\alpha\beta\sigma\tau}\sqrt{a}\|_{\mathcal{C}^1(\overline{\omega})}\}\right) \left(\max_{\alpha,\beta\in\{1,2\}}\|\gamma_{\alpha\beta}(D_{\rho h}(\varphi\bm{\zeta}^\varepsilon_\kappa)\|_{L^2(\omega_1)}\right) \left(\max_{\sigma,\tau\in\{1,2\}}\|\gamma_{\sigma\tau}(\varphi \bm{\zeta}^\varepsilon_\kappa)\|_{L^2(\omega_0)}\right)\\
&= C\varepsilon(1+\|D_{\rho h}(\varphi\bm{\zeta}^\varepsilon_\kappa)\|_{H^1(\omega_1)\times H^1(\omega_1)\times L^2(\omega_1)})\\
&\quad+\varepsilon \left(\max_{\alpha,\beta,\sigma,\tau \in \{1,2\}}\{\|a^{\alpha\beta\sigma\tau}\sqrt{a}\|_{\mathcal{C}^1(\overline{\omega})}\}\right) 
\|D_{\rho h}(\varphi\bm{\zeta}^\varepsilon_\kappa)\|_{H^1(\omega_1)\times H^1(\omega_1)\times L^2(\omega_1)}
\|\bm{\zeta}^\varepsilon_\kappa\|_{H^1(\omega)\times H^1(\omega)\times L^2(\omega)}\\
&\le C\varepsilon(1+\|D_{\rho h}(\varphi\bm{\zeta}^\varepsilon_\kappa)\|_{H^1(\omega_1)\times H^1(\omega_1)\times L^2(\omega_1)}),
\end{aligned}
\end{equation*}
where, once again, the constant $C>0$ is independent of $\varepsilon$, $\kappa$ and $h$. The latter computations summarize in the following result
\begin{equation}
\label{checkpoint-1}
\begin{aligned}
&\dfrac{\varepsilon\sqrt{a_0}}{c_0 c_e}\|D_{\rho h}(\varphi \bm{\zeta}^\varepsilon_\kappa)\|_{H^1(\omega_1)\times H^1(\omega_1)\times L^2(\omega_1)}^2
+\dfrac{\varepsilon}{\kappa}\int_{\omega_1}\bm{\beta}(\bm{\zeta}^\varepsilon_\kappa) \cdot (-\varphi \delta_{\rho h}(\varphi\bm{\zeta}^\varepsilon_\kappa)) \dd y\\
&\le C\varepsilon(1+\|D_{\rho h}(\varphi\bm{\zeta}^\varepsilon_\kappa)\|_{H^1(\omega_1)\times H^1(\omega_1)\times L^2(\omega_1)}),
\end{aligned}
\end{equation}
for some constant $C>0$ is independent of $\varepsilon$, $\kappa$ and $h$.

Let us now estimate the penalty term. Thanks to the equations of Problem~\ref{problem1}, we have that
\begin{equation*}
\label{bdd-1}
\dfrac{\varepsilon}{\kappa}\int_{\omega}\bm{\beta}(\bm{\zeta}^\varepsilon_\kappa) \cdot\bm{\eta} \dd y=-\varepsilon\int_{\omega}a^{\alpha\beta\sigma\tau}\gamma_{\sigma \tau}(\bm{\zeta}^\varepsilon_\kappa) \gamma_{\alpha\beta}(\bm{\eta}) \sqrt{a} \dd y+\int_{\omega}p^{i,\varepsilon} \eta_i\sqrt{a} \dd y,\quad\textup{ for all }\bm{\eta}=(\eta_i) \in \bm{V}_M(\omega).
\end{equation*}

An application of the triangle inequality and the continuity of the components $\gamma_{\alpha\beta}$ of the linearized change of metric tensor gives
\begin{equation*}
\left|\dfrac{\varepsilon}{\kappa}\int_{\omega}\bm{\beta}(\bm{\zeta}^\varepsilon_\kappa) \cdot\bm{\eta} \dd y\right|\le \varepsilon \left(\max_{\alpha,\beta,\sigma,\tau \in \{1,2\}}\|a^{\alpha\beta\sigma\tau}\|_{\mathcal{C}^0(\overline{\omega})}\right)\|\bm{\zeta}^\varepsilon_\kappa\|_{\bm{V}_M(\omega)}\|\bm{\eta}\|_{\bm{V}_M(\omega)} \sqrt{a}_1+\varepsilon\|\bm{p}\|_{\bm{L}^2(\omega)}\|\bm{\eta}\|_{\bm{L}^2(\omega)}\sqrt{a_1},
\end{equation*}
for all $\bm{\eta}=(\eta_i) \in\bm{V}_M(\omega)$.

Passing to the supremum over all the vector fields $\bm{\eta}=(\eta_i) \in\bm{V}_M(\omega)$ with $\|\bm{\eta}\|_{\bm{V}_M(\omega)}=1$ gives
\begin{equation*}
\sup_{\substack{\bm{\eta}\in\bm{V}_M(\omega)\\\|\bm{\eta}\|_{\bm{V}_M(\omega)}=1}} \left|\dfrac{1}{\kappa}\int_{\omega}\bm{\beta}(\bm{\zeta}^\varepsilon_\kappa) \cdot\bm{\eta} \dd y\right|
\le \sqrt{a_1}\left(\left(\max_{\alpha,\beta,\sigma,\tau \in \{1,2\}}\|a^{\alpha\beta\sigma\tau}\|_{\mathcal{C}^0(\overline{\omega})}\right)\|\bm{\zeta}^\varepsilon_\kappa\|_{\bm{V}_M(\omega)}+\|\bm{p}\|_{\bm{L}^2(\omega)}\right),
\end{equation*}
where, by Theorem~\ref{ex-un-kappa}, the right hand side is bounded independently of $\varepsilon$ and $\kappa$. In conclusion, we have shown that there exists a constant $M_1>0$ independent of $\varepsilon$ and $\kappa$ (and clearly $h$) such that:
\begin{equation}
\label{bdd-2}
\dfrac{1}{\kappa}\|\bm{\beta}(\bm{\zeta}^\varepsilon_\kappa)\|_{\bm{V}'_M(\omega)} \le M_1.
\end{equation}

The fact that we identified $L^2(\omega)$ with its dual, the assumption $\min_{y \in \overline{\omega}}(\bm{a}^3\cdot\bm{q})>0$, and~\eqref{bdd-2} give
\begin{equation*}
\begin{aligned}
M_1&\ge\dfrac{1}{\kappa}\|\bm{\beta}(\bm{\zeta}^\varepsilon_\kappa)\|_{\bm{V}'_M(\omega)}
=\dfrac{1}{\kappa}\Bigg\{\left\|-\{(\bm{\theta}+\zeta^\varepsilon_{\kappa,j}\bm{a}^j)\cdot\bm{q}\}^{-}\left(\dfrac{\bm{a}^1\cdot\bm{q}}{\sqrt{\sum_{\ell=1}^3|\bm{a}^\ell\cdot\bm{q}|^2}}\right)\right\|_{H^{-1}(\omega)}^2\\
&\quad+\left\|-\{(\bm{\theta}+\zeta^\varepsilon_{\kappa,j}\bm{a}^j)\cdot\bm{q}\}^{-}\left(\dfrac{\bm{a}^2\cdot\bm{q}}{\sqrt{\sum_{\ell=1}^3|\bm{a}^\ell\cdot\bm{q}|^2}}\right)\right\|_{H^{-1}(\omega)}^2\\
&\quad+\left\|-\{(\bm{\theta}+\zeta^\varepsilon_{\kappa,j}\bm{a}^j)\cdot\bm{q}\}^{-}\left(\dfrac{\bm{a}^3\cdot\bm{q}}{\sqrt{\sum_{\ell=1}^3|\bm{a}^\ell\cdot\bm{q}|^2}}\right)\right\|_{L^{2}(\omega)}^2
\Bigg\}^{1/2}\\
&\ge\dfrac{1}{\kappa} \left\|-\{(\bm{\theta}+\zeta^\varepsilon_{\kappa,j}\bm{a}^j)\cdot\bm{q}\}^{-}\left(\dfrac{\bm{a}^3\cdot\bm{q}}{\sqrt{\sum_{\ell=1}^3|\bm{a}^\ell\cdot\bm{q}|^2}}\right)\right\|_{L^{2}(\omega)}\\
&\ge\dfrac{\left(\min_{y \in \overline{\omega}}(\bm{a}^3\cdot\bm{q})\right)}{\kappa\sqrt{3\max\left\{\|\bm{a}^\ell \cdot\bm{q}\|_{\mathcal{C}^0(\overline{\omega})}^2;1\le\ell\le 3\right\}}}
\left(\int_{\omega} |-\{(\bm{\theta}+\zeta^\varepsilon_{\kappa,j}\bm{a}^j)\cdot\bm{q}\}^{-}|^2 \dd y\right)^{1/2}\\
&\ge \dfrac{\left(\min_{y \in \overline{\omega}}(\bm{a}^3\cdot\bm{q})\right)}{\kappa\sqrt{3\max\left\{\|\bm{a}^\ell \cdot\bm{q}\|_{\mathcal{C}^0(\overline{\omega})}^2;1\le\ell\le 3\right\}}} \|-\{(\bm{\theta}+\zeta^\varepsilon_{\kappa,j}\bm{a}^j)\cdot\bm{q}\}^{-}\|_{L^2(\omega)},
\end{aligned}
\end{equation*}
so that we have the following estimate:
\begin{equation}
\label{bdd-3}
\|-\{(\bm{\theta}+\zeta^\varepsilon_{\kappa,j}\bm{a}^j)\cdot\bm{q}\}^{-}\|_{L^2(\omega)}\le \kappa\dfrac{M_1\sqrt{3\max\left\{\|\bm{a}^\ell \cdot\bm{q}\|_{\mathcal{C}^0(\overline{\omega})}^2;1\le\ell\le 3\right\}}}{\left(\min_{y \in \overline{\omega}}(\bm{a}^3\cdot\bm{q})\right)}.
\end{equation}

Let us now evaluate the penalty term in the governing equations of Problem~\ref{problem2}. An application of formulas~\eqref{D+}, \eqref{D-}, \eqref{delta+}, Lemma~\ref{fdq-neg-part},  Lemma~\ref{geometry} and~\eqref{bdd-3} gives:
\begin{align*}
&\dfrac{1}{\kappa}\int_{\omega_1}\bm{\beta}(\bm{\zeta}^\varepsilon_\kappa) \cdot (-\varphi \delta_{\rho h}(\varphi \bm{\zeta}^\varepsilon_\kappa))\dd y =-\dfrac{1}{\kappa}\int_{\omega_1}\left[-\varphi\{(\bm{\theta}+\zeta^\varepsilon_{\kappa,j}\bm{a}^j)\cdot\bm{q}\}^{-}\left(\dfrac{\bm{a}^i\cdot\bm{q}}{\sqrt{\sum_{\ell=1}^3|\bm{a}^\ell\cdot\bm{q}|^2}}\right)\right]\delta_{\rho h}(\varphi\zeta^\varepsilon_{\kappa,i})\dd y\\
&=\dfrac{1}{\kappa}\int_{\omega_1}D_{\rho h}\left(-\varphi\{(\bm{\theta}+\zeta^\varepsilon_{\kappa,j}\bm{a}^j)\cdot\bm{q}\}^{-} \left(\dfrac{\bm{a}^i\cdot\bm{q}}{\sqrt{\sum_{\ell=1}^3|\bm{a}^\ell\cdot\bm{q}|^2}}\right)\right) D_{\rho h}(\varphi\zeta^\varepsilon_{\kappa,i})\dd y\\
&=\dfrac{1}{\kappa}\int_{\omega_1}\left[D_{\rho h}\left(-\{(\bm{\theta}+\zeta^\varepsilon_{\kappa,j}\bm{a}^j)\cdot\bm{q}\}^{-} \varphi\right) E_{\rho h}\left(\dfrac{\bm{a}^i\cdot\bm{q}}{\sqrt{\sum_{\ell=1}^3|\bm{a}^\ell\cdot\bm{q}|^2}}\right)\right] D_{\rho h}(\varphi\zeta^\varepsilon_{\kappa,i})\dd y\\
&\quad+\dfrac{1}{\kappa}\int_{\omega_1} \left[\left(-\{(\bm{\theta}+\zeta^\varepsilon_{\kappa,j}\bm{a}^j)\cdot\bm{q}\}^{-} \varphi\right) D_{\rho h}\left(\dfrac{\bm{a}^i\cdot\bm{q}}{\sqrt{\sum_{\ell=1}^3|\bm{a}^\ell\cdot\bm{q}|^2}}\right)\right] D_{\rho h}(\varphi\zeta^\varepsilon_{\kappa,i})\dd y\\
&=\dfrac{1}{\kappa}\int_{\omega_1}E_{\rho h}\left(\dfrac{1}{\sqrt{\sum_{\ell=1}^3|\bm{a}^\ell\cdot\bm{q}|^2}}\right)\left[D_{\rho h}\left(-\{(\bm{\theta}+\zeta^\varepsilon_{\kappa,j}\bm{a}^j)\cdot\bm{q}\}^{-} \varphi\right)\right] D_{\rho h}\left(\varphi \zeta^\varepsilon_{\kappa,i}\bm{a}^i\cdot\bm{q}\right)\dd y\\
&\quad-\dfrac{1}{\kappa}\int_{\omega_1}E_{\rho h}\left(\dfrac{1}{\sqrt{\sum_{\ell=1}^3|\bm{a}^\ell\cdot\bm{q}|^2}}\right)\left[D_{\rho h}\left(-\{(\bm{\theta}+\zeta^\varepsilon_{\kappa,j}\bm{a}^j)\cdot\bm{q}\}^{-} \varphi\right)\right] (\varphi \zeta^\varepsilon_{\kappa,i}) D_{\rho h}\left(\bm{a}^i\cdot\bm{q}\right)\dd y\\
&\quad+\dfrac{1}{\kappa}\int_{\omega_1} \left(-\{(\bm{\theta}+\zeta^\varepsilon_{\kappa,j}\bm{a}^j)\cdot\bm{q}\}^{-} \varphi\right)\left(D_{\rho h}\left(\dfrac{\bm{a}^i\cdot\bm{q}}{\sqrt{\sum_{\ell=1}^3|\bm{a}^\ell\cdot\bm{q}|^2}}\right) D_{\rho h}(\varphi\zeta^\varepsilon_{\kappa,i})\right)\dd y\\
&=\dfrac{1}{\kappa}\int_{\omega_1}E_{\rho h}\left(\dfrac{1}{\sqrt{\sum_{\ell=1}^3|\bm{a}^\ell\cdot\bm{q}|^2}}\right)\left[D_{\rho h}\left(-\{(\bm{\theta}+\zeta^\varepsilon_{\kappa,j}\bm{a}^j)\cdot\bm{q}\}^{-} \varphi\right)\right] D_{\rho h}\left(\varphi (\bm{\theta}+\zeta^\varepsilon_{\kappa,i}\bm{a}^i)\cdot\bm{q}\right)\dd y\\
&\quad-\dfrac{1}{\kappa}\int_{\omega_1}E_{\rho h}\left(\dfrac{1}{\sqrt{\sum_{\ell=1}^3|\bm{a}^\ell\cdot\bm{q}|^2}}\right)\left[D_{\rho h}\left(-\{(\bm{\theta}+\zeta^\varepsilon_{\kappa,j}\bm{a}^j)\cdot\bm{q}\}^{-} \varphi\right)\right] D_{\rho h}\left(\varphi \bm{\theta}\cdot\bm{q}\right)\dd y\\
&\quad+\dfrac{1}{\kappa}\int_{\omega_1}\left(-\{(\bm{\theta}+\zeta^\varepsilon_{\kappa,j}\bm{a}^j)\cdot\bm{q}\}^{-} \varphi\right) D_{-\rho h}\left(E_{\rho h}\left(\dfrac{1}{\sqrt{\sum_{\ell=1}^3|\bm{a}^\ell\cdot\bm{q}|^2}}\right)(\varphi \zeta^\varepsilon_{\kappa,i}) D_{\rho h}\left(\bm{a}^i\cdot\bm{q}\right)\right)\dd y\\
&\quad+\dfrac{1}{\kappa}\int_{\omega_1} \left(-\{(\bm{\theta}+\zeta^\varepsilon_{\kappa,j}\bm{a}^j)\cdot\bm{q}\}^{-} \varphi\right)\left(D_{\rho h}\left(\dfrac{\bm{a}^i\cdot\bm{q}}{\sqrt{\sum_{\ell=1}^3|\bm{a}^\ell\cdot\bm{q}|^2}}\right) D_{\rho h}(\varphi\zeta^\varepsilon_{\kappa,i})\right)\dd y\\
&=\dfrac{1}{\kappa}\int_{\omega_1} E_{\rho h}\left(\dfrac{1}{\sqrt{\sum_{\ell=1}^3|\bm{a}^\ell\cdot\bm{q}|^2}}\right)\left[D_{\rho h}\left(-\{(\bm{\theta}+\zeta^\varepsilon_{\kappa,j}\bm{a}^j)\cdot\bm{q}\}^{-} \varphi\right)\right] D_{\rho h}\left(\varphi \{(\bm{\theta}+\zeta^\varepsilon_{\kappa,i}\bm{a}^i)\cdot\bm{q}\}^{+}\right)\dd y\\
&\quad+\dfrac{1}{\kappa}\int_{\omega_1} E_{\rho h}\left(\dfrac{1}{\sqrt{\sum_{\ell=1}^3|\bm{a}^\ell\cdot\bm{q}|^2}}\right)\left|D_{\rho h}\left(-\{(\bm{\theta}+\zeta^\varepsilon_{\kappa,j}\bm{a}^j)\cdot\bm{q}\}^{-} \varphi\right)\right|^2\dd y\\
&\quad+\dfrac{1}{\kappa}\int_{\omega_1}(-\varphi \{(\bm{\theta}+\zeta^\varepsilon_{\kappa,j}\bm{a}^j)\cdot\bm{q}\}^{-}) D_{-\rho h}\left[E_{\rho h}\left(\dfrac{1}{\sqrt{\sum_{\ell=1}^3|\bm{a}^\ell\cdot\bm{q}|^2}}\right) D_{\rho h}(\varphi \bm{\theta}\cdot\bm{q})\right] \dd y\\
&\quad+\dfrac{1}{\kappa}\int_{\omega_1}(-\varphi \{(\bm{\theta}+\zeta^\varepsilon_{\kappa,j}\bm{a}^j)\cdot\bm{q}\}^{-}) D_{-\rho h}\left[E_{\rho h}\left(\dfrac{1}{\sqrt{\sum_{\ell=1}^3|\bm{a}^\ell\cdot\bm{q}|^2}}\right) \left(D_{\rho h}(\bm{a}^i\cdot\bm{q})\right) (\varphi\zeta^\varepsilon_{\kappa,i})\right] \dd y\\
&\quad+\dfrac{1}{\kappa}\int_{\omega_1} \left(-\{(\bm{\theta}+\zeta^\varepsilon_{\kappa,j}\bm{a}^j)\cdot\bm{q}\}^{-} \varphi\right)\left(D_{\rho h}\left(\dfrac{\bm{a}^i\cdot\bm{q}}{\sqrt{\sum_{\ell=1}^3|\bm{a}^\ell\cdot\bm{q}|^2}}\right) D_{\rho h}(\varphi\zeta^\varepsilon_{\kappa,i})\right)\dd y.
\end{align*}

Applying the latter computations, \eqref{bdd-3}, the fact that $\bm{\theta}\in\mathcal{C}^3(\overline{\omega};\mathbb{E}^3)$, Lemma~\ref{fdq-neg-part}, Lemma~\ref{geometry}, the assumption according to which $\min_{y \in \overline{\omega}}(\bm{a}^3\cdot\bm{q})>0$ and the fact that $\textup{supp }\varphi \subset\subset \omega_1$ to~\eqref{checkpoint-1} gives:
\begin{align*}
	&\dfrac{\varepsilon\sqrt{a_0}}{c_0 c_e}\|D_{\rho h}(\varphi \bm{\zeta}^\varepsilon_\kappa)\|_{H^1(\omega_1)\times H^1(\omega_1)\times L^2(\omega_1)}^2\\
	&\quad+\varepsilon\dfrac{\left(3\max\{\|\tilde{\bm{a}}^\ell\cdot\bm{q}\|_{\mathcal{C}^0(\overline{\tilde{\omega}})}^2;1\le \ell \le 3\}\right)^{-1/2}}{\kappa}\int_{\omega_1}\left|D_{\rho h}\left(-\{(\bm{\theta}+\zeta^\varepsilon_{\kappa,j}\bm{a}^j)\cdot\bm{q}\}^{-} \varphi\right)\right|^2\dd y\\
	&\le C\varepsilon(1+\|D_{\rho h}(\varphi\bm{\zeta}^\varepsilon_\kappa)\|_{H^1(\omega_1)\times H^1(\omega_1)\times L^2(\omega_1)})\\
	&\quad-\dfrac{\varepsilon}{\kappa}\int_{\omega_1}(-\varphi \{(\bm{\theta}+\zeta^\varepsilon_{\kappa,j}\bm{a}^j)\cdot\bm{q}\}^{-}) D_{-\rho h}\left[E_{\rho h}\left(\dfrac{1}{\sqrt{\sum_{\ell=1}^3|\bm{a}^\ell\cdot\bm{q}|^2}}\right) D_{\rho h}(\varphi \bm{\theta}\cdot\bm{q})\right] \dd y\\
	&\quad-\dfrac{\varepsilon}{\kappa}\int_{\omega_1}(-\varphi \{(\bm{\theta}+\zeta^\varepsilon_{\kappa,j}\bm{a}^j)\cdot\bm{q}\}^{-}) D_{-\rho h}\left[E_{\rho h}\left(\dfrac{1}{\sqrt{\sum_{\ell=1}^3|\bm{a}^\ell\cdot\bm{q}|^2}}\right) \left(D_{\rho h}(\bm{a}^i\cdot\bm{q})\right) (\varphi\zeta^\varepsilon_{\kappa,i})\right] \dd y\\
	&\quad-\dfrac{\varepsilon}{\kappa}\int_{\omega_1} \left(-\{(\bm{\theta}+\zeta^\varepsilon_{\kappa,j}\bm{a}^j)\cdot\bm{q}\}^{-} \varphi\right)\left(D_{\rho h}\left(\dfrac{\bm{a}^i\cdot\bm{q}}{\sqrt{\sum_{\ell=1}^3|\bm{a}^\ell\cdot\bm{q}|^2}}\right) D_{\rho h}(\varphi\zeta^\varepsilon_{\kappa,i})\right)\dd y\\
	&\le C\varepsilon(1+\|D_{\rho h}(\varphi\bm{\zeta}^\varepsilon_\kappa)\|_{H^1(\omega_1)\times H^1(\omega_1)\times L^2(\omega_1)}),
\end{align*}
for some constant $C>0$ independent of $\varepsilon$, $\kappa$ and $h$. In conclusion, the latter computations can be summarized as follows:
\begin{equation}
\label{checkpoint-2}
\begin{aligned}
&\dfrac{\varepsilon\sqrt{a_0}}{c_0 c_e}\|D_{\rho h}(\varphi \bm{\zeta}^\varepsilon_\kappa)\|_{H^1(\omega_1)\times H^1(\omega_1)\times L^2(\omega_1)}^2\\
&\quad+\varepsilon\dfrac{\left(3\max\{\|\tilde{\bm{a}}^\ell\cdot\bm{q}\|_{\mathcal{C}^0(\overline{\tilde{\omega}})}^2;1\le \ell \le 3\}\right)^{-1/2}}{\kappa}\int_{\omega_1}\left|D_{\rho h}\left(-\{(\bm{\theta}+\zeta^\varepsilon_{\kappa,j}\bm{a}^j)\cdot\bm{q}\}^{-} \varphi\right)\right|^2\dd y\\
&\le C\varepsilon(1+\|D_{\rho h}(\varphi\bm{\zeta}^\varepsilon_\kappa)\|_{H^1(\omega_1)\times H^1(\omega_1)\times L^2(\omega_1)}).
\end{aligned}
\end{equation}

A consequence of~\eqref{checkpoint-2} is that
\begin{equation}
\label{conclusion-1}
\dfrac{\sqrt{a_0}}{c_0 c_e}\|D_{\rho h}(\varphi \bm{\zeta}^\varepsilon_\kappa)\|_{H^1(\omega_1)\times H^1(\omega_1)\times L^2(\omega_1)}^2
-C\|D_{\rho h}(\varphi\bm{\zeta}^\varepsilon_\kappa)\|_{H^1(\omega_1)\times H^1(\omega_1)\times L^2(\omega_1)}-C\le 0.
\end{equation}

Regarding $\|D_{\rho h}(\varphi \bm{\zeta}^\varepsilon_\kappa)\|_{H^1(\omega_1)\times H^1(\omega_1)\times L^2(\omega_1)}$ as the variable of the corresponding second-degree polynomial $\frac{\sqrt{a_0}}{c_0 c_e} x^2 -C x -C$, we have that its discriminant is positive. Therefore, we have that the inequality~\eqref{conclusion-1} is satisfied for
\begin{equation}
	\label{conclusion-2}
0\le \|D_{\rho h}(\varphi \bm{\zeta}^\varepsilon_\kappa)\|_{H^1(\omega_1)\times H^1(\omega_1)\times L^2(\omega_1)} \le \dfrac{C+\sqrt{C^2+4\frac{C\sqrt{a_0}}{c_0 c_e}}}{\frac{2 \sqrt{a_0}}{c_0 c_e}},
\end{equation}
where the upper bound is independent of $\varepsilon$, $\kappa$ and $h$. Applying~\eqref{conclusion-2} to~\eqref{checkpoint-2} gives that
\begin{equation}
\label{conclusion-3}
\dfrac{1}{\kappa}\left\|D_{\rho h}\left(-\{(\bm{\theta}+\zeta^\varepsilon_{\kappa,j}\bm{a}^j)\cdot\bm{q}\}^{-} \varphi\right)\right\|_{L^2(\omega_1)}^2 \le C,
\end{equation}
for some $C>0$ independent of $\varepsilon$, $\kappa$ and $h$.

An application of Theorem~3 of Section~5.8.2 of~\cite{Evans2010}, together with the fact that $\varphi$ in a way such that its support has nonempty interior in $\omega$ and that there exists a nonzero measure set $U\subset \textup{supp }\varphi$ such that $\varphi\equiv 1$ in $U$ shows that the sequence $\{\bm{\zeta}^\varepsilon_\kappa\}_{\kappa>0}$ is bounded in $H^2_{\textup{loc}}(\omega) \times H^2_{\textup{loc}}(\omega) \times H^1_{\textup{loc}}(\omega)$ independently of $\kappa$ as well as that $\{(\bm{\theta}+\zeta^\varepsilon_{\kappa,j}\bm{a}^j)\cdot\bm{q}\}^{-} \in H^1_{\textup{loc}}(\omega)$, and
\begin{equation*}
\|-\{(\bm{\theta}+\zeta^\varepsilon_{\kappa,j}\bm{a}^j)\cdot\bm{q}\}^{-}\|_{H^1(U)} \le C\sqrt{\kappa}.
\end{equation*}

Exploiting the fact that $(\bm{a}^i\cdot\bm{q}) \in \mathcal{C}^1(\overline{\omega})$ for all $1\le i \le 3$ and the assumption $\min_{y \in \overline{\omega}}(\bm{a}^3\cdot\bm{q})>0$, we have that an application of the product rule in Sobolev spaces (cf., e.g., Proposition~9.4 of~\cite{Brez11}) together with~\eqref{conclusion-3} implies that each component of the vector field $\bm{\beta}(\varphi\bm{\zeta}^\varepsilon_\kappa)$ is of class $H^1_{\textup{loc}}(\omega)$ and that the following estimate holds
\begin{equation*}
\label{conclusion-4}
\|\bm{\beta}(\bm{\zeta}^\varepsilon_\kappa)\|_{\bm{H}^1(U)}^2\le C\kappa,
\end{equation*}
for some $C>0$ independent of $\varepsilon$, $\kappa$ and $h$. This completes the proof.
\end{proof}

As a remark, we observe that the higher regularity of the negative part of the constraint has been established without resorting by any means to Stampacchia's theorem~\cite{Stampacchia1965}. Moreover, we showed that the negative part approaches zero as $\kappa \to 0^+$ \emph{more rapidly} than what inferred in the energy estimates in Theorem~\ref{ex-un-kappa}.

A straightforward consequence of~\eqref{conclusion-3} is that
\begin{equation*}
\bm{\zeta}^\varepsilon_\kappa \rightharpoonup \bm{\zeta}^\varepsilon,\quad\textup{ in }H^2(\omega_1) \times H^2(\omega_1) \times H^1(\omega_1) \textup{ as }\kappa\to0^+, 
\end{equation*}
thus showing an alternative proof of the interior regularity for the solution of Problem~\ref{problem1} \emph{without} resorting, as it was instead done in~\cite{Pie-2022-interior}, to the ``density property'' recalled in Theorem~\ref{density} (although in the proof of Theorem~\ref{aug:int} we exploited the \emph{sufficient conditions} ensuring the validity of the ``density property'') and \emph{without} assuming additional regularity for the tangential components of $\bm{p}^\varepsilon$.

The result established in Theorem~\ref{aug:int} actually shows that the solution of Problem~\ref{problem1} is the weak limit of the sequence of solutions of Problem~\ref{problem2} in the space $H^2(\omega_1) \times H^2(\omega_1) \times H^1(\omega_1)$.

Let us now show that the solution $\bm{\zeta}^\varepsilon_\kappa$ of Problem~\ref{problem2} enjoys the higher regularity established in Theorem~\ref{aug:int} up to the boundary of the domain $\omega$. In order to establish this result, we will need to make the assumption that the solution $\bm{\zeta}^\varepsilon_\kappa$ of Problem~\ref{problem2} does not violate the constraint under consideration near the boundary of the integration domain $\omega$. This assumption is physically feasible, since this limit model is derived as a result of asymptotic analyses of models whose solutions have vanishing trace along the boundary (cf.~\cite{CiaPie2018b,CiaPie2018bCR,CiaMarPie2018b,CiaMarPie2018}).

\begin{theorem}
\label{aug:bdry}
Assume that the boundary $\gamma$ of the domain $\omega$ is of class $\mathcal{C}^4$ and that the immersion $\bm{\theta}$ is of class $\mathcal{C}^4(\overline{\omega};\mathbb{E}^3)$.
Assume that there exists a unit-norm vector $\bm{q} \in \mathbb{E}^3$ such that
\begin{equation*}
\min_{y \in \overline{\omega}} (\bm{\theta} (y) \cdot \bm{q}) > 0 \quad
\textup{ and } \quad
\min_{y \in \overline{\omega}} (\bm{a}_3 (y) \cdot \bm{q}) > 0.
\end{equation*}

Assume also that the vector field $\bm{f}^\varepsilon=(f^{i,\varepsilon})$ defining the applied body force density is such that $\bm{p}^\varepsilon=(p^{i,\varepsilon}) \in L^2(\omega) \times L^2(\omega) \times H^1(\omega)$. Define $\bm{H}(\omega):=H^2(\omega) \times H^2(\omega) \times H^1(\omega)$.

Finally, assume that the solution $\bm{\zeta}^\varepsilon_\kappa$ of Problem~\ref{problem2} is such that there exists a neighbourhood $U\subset \gamma$ independent of $\varepsilon$ and $\kappa$ such that
\begin{equation}
\label{conclusion-5.5}
\bm{\beta}(\bm{\zeta}^\varepsilon_\kappa) =\bm{0} \textup{ for a.a. points in } U \cap \omega. 
\end{equation}

Then, the solution $\bm{\zeta}^\varepsilon_\kappa=(\zeta^\varepsilon_{\kappa,i})$ of Problem~\ref{problem2} is of class $\bm{V}_M(\omega)\cap \bm{H}(\omega)$.
\end{theorem}
\begin{proof}
Let $\bm{\zeta}^\varepsilon_\kappa=(\zeta^\varepsilon_{\kappa,i})$ be the solution of Problem~\ref{problem2}.
Combining the assumption according to which $\bm{\beta}(\bm{\zeta}^\varepsilon_\kappa)=\bm{0}$ for a.a. points in $U\cap \omega$ with the conclusion of Theorem~\ref{aug:int} according to which $\bm{\beta}(\bm{\zeta}^\varepsilon_\kappa) \in \bm{H}^1_{\textup{loc}}(\omega)$, we straightforwardly infer that $\bm{\beta}(\bm{\zeta}^\varepsilon_\kappa) \in \bm{H}^1_0(\omega)$.

Keeping in mind the boundary value problem~\eqref{BVP} we recovered beforehand, we apply the elliptic augmentation of regularity argument near the boundary proposed in~\cite{Genevey1996} after observing that:
$$
\left(p^{3,\varepsilon}-\dfrac{\varepsilon}{\kappa\sqrt{a}}\beta_3(\bm{\zeta}^\varepsilon_\kappa)\right) \in H^1(\omega).
$$

This completes the proof.
\end{proof}

The boundary value problem recovered in~\eqref{BVP} enters, in the same spirit of Theorem~4 on page~334 of~\cite{Evans2010}, the proof of Theorem~\ref{aug:bdry} to show the augmented regularity in the nearness of a flat boundary for the \emph{reduced problem}.

As a remark, we observe that an application of Theorem~\ref{aug:int} and Theorem~\ref{aug:bdry} gives
\begin{equation}
	\label{conclusion-5}
	\bm{\zeta}^\varepsilon_\kappa \rightharpoonup \bm{\zeta}^\varepsilon,\quad\textup{ in }H^2(\omega)\times H^2(\omega) \times H^1(\omega) \textup{ as }\kappa \to 0^+,
\end{equation}
where we recall that $\bm{\zeta}^\varepsilon$ is the solution of Problem~\ref{problem1}.

Furthermore, the estimate~\eqref{conclusion-2} can be extended up to the boundary, so that, exploiting the compactness of $\overline{\omega}$ gives
\begin{equation}
	\label{conclusion-6}
	0\le \|\bm{\zeta}^\varepsilon_\kappa\|_{H^2(\omega)\times H^2(\omega)\times H^1(\omega)} \le \dfrac{C+\sqrt{C^2+4\frac{C\sqrt{a_0}}{c_0 c_e}}}{\frac{2 \sqrt{a_0}}{c_0 c_e}},
\end{equation}
for some $C>0$ independent of $\varepsilon$, $\kappa$ and $h$. Combining the lower semicontinuity of $\|\cdot\|_{H^2(\omega)\times H^2(\omega)\times H^1(\omega)}$ with~\eqref{conclusion-5} and~\eqref{conclusion-6} gives that
\begin{equation}
\label{conclusion-7}
\|\bm{\zeta}^\varepsilon\|_{H^2(\omega)\times H^2(\omega)\times H^1(\omega)} \le \liminf_{\kappa\to 0^+}\|\bm{\zeta}^\varepsilon_\kappa\|_{H^2(\omega)\times H^2(\omega)\times H^1(\omega)} \le \dfrac{C+\sqrt{C^2+4\frac{C\sqrt{a_0}}{c_0 c_e}}}{\frac{2 \sqrt{a_0}}{c_0 c_e}},
\end{equation}
thus asserting that the solution of Problem~\ref{problem1} is of class $\bm{H}(\omega)=H^2(\omega)\times H^2(\omega)\times H^1(\omega)$ and which is bounded in $\bm{H}(\omega)$ independently of $\varepsilon$.

The results established in Theorem~\ref{aug:int} and Theorem~\ref{aug:bdry} actually improve Theorem~\ref{ex-un-kappa} as the solution of Problem~\ref{problem1} is proved to be the weak limit of the sequence of solutions of Problem~\ref{problem2} in the space $H^2(\omega) \times H^2(\omega) \times H^1(\omega)$.

Finally, we recall that the augmentation of regularity up to the boundary holds for domains with Lipschitz continuous boundary provided that $\omega$ is convex (viz. \cite{Eggleston1958} and~\cite{Grisvard2011}).

\section{Approximation of the solution of Problem~\ref{problem1} via the Penalty Method}
\label{approx:original}

In this section, we exploit the augmentation of regularity established in Theorem~\ref{aug:int}, Theorem~\ref{aug:bdry} as well as the subsequent remarks to sharpen the convergence~\eqref{beta-5} obtained as a result of Theorem~\ref{ex-un-kappa}. 

\begin{theorem}
\label{th:beta-6}
Let $\kappa>0$ be given.
Let $\bm{\zeta}^\varepsilon \in \bm{V}_M(\omega) \cap \bm{H}(\omega)$ be the solution of Problem~\ref{problem1} and let $\bm{\zeta}^\varepsilon_\kappa \in \bm{V}_M(\omega) \cap \bm{H}(\omega)$ be the solution of Problem~\ref{problem2}. Then, there exists a constant $C>0$ independent of $\varepsilon$ and $\kappa$ such that
\begin{equation*}
\|\bm{\zeta}^\varepsilon-\bm{\zeta}^\varepsilon_\kappa\|_{\bm{V}_M(\omega)}\le C \sqrt{\kappa}.
\end{equation*}
\end{theorem}
\begin{proof}
For each $\bm{\eta}\in \bm{L}^2(\omega)$, define
\begin{equation*}
\tilde{\bm{\beta}}(\bm{\eta}):=\left(-\{(\bm{\theta}+\eta_j\bm{a}^j)\cdot\bm{q}\}^{-}\left(\dfrac{\bm{a}^i\cdot\bm{q}}{\sum_{\ell=1}^3|\bm{a}^\ell\cdot\bm{q}|^2}\right)\right)_{i=1}^3.
\end{equation*}

Define $P(\bm{\zeta}^\varepsilon_\kappa):=\bm{\zeta}^\varepsilon_\kappa-\tilde{\bm{\beta}}(\bm{\zeta}^\varepsilon_\kappa)$, and observe that $P(\bm{\zeta}^\varepsilon_\kappa) \in\bm{U}_M(\omega)$. Indeed, a direct computation gives
\begin{equation*}
\begin{aligned}
&\left(\bm{\theta}+\left[\zeta^\varepsilon_{\kappa,i}-\dfrac{-\{(\bm{\theta}+\zeta^\varepsilon_{\kappa,i}\bm{a}^i)\cdot\bm{q}\}^{-}(\bm{a}^i\cdot\bm{q})}{\sum_{\ell=1}^3|\bm{a}^\ell\cdot\bm{q}|^2}\right]\bm{a}^i\right)\cdot\bm{q}
=((\bm{\theta}+\zeta^\varepsilon_{\kappa,i}\bm{a}^i)\cdot\bm{q})+\{(\bm{\theta}+\zeta^\varepsilon_{\kappa,i}\bm{a}^i)\cdot\bm{q}\}^{-}\\
&=\{(\bm{\theta}+\zeta^\varepsilon_{\kappa,i}\bm{a}^i)\cdot\bm{q}\}^{+}\ge0,
\end{aligned}
\end{equation*}
thus proving the claim.
Let us estimate
\begin{equation*}
\label{est-1}
\|\bm{\zeta}^\varepsilon_\kappa-\bm{\zeta}^\varepsilon\|_{\bm{V}_M(\omega)} \le \|\tilde{\bm{\beta}}(\bm{\zeta}^\varepsilon_\kappa)\|_{\bm{H}^1(\omega)}
+\|\bm{\zeta}^\varepsilon_\kappa-\tilde{\bm{\beta}}(\bm{\zeta}^\varepsilon_\kappa)-\bm{\zeta}^\varepsilon\|_{\bm{V}_M(\omega)} \le C \sqrt{\kappa}+\|\bm{\zeta}^\varepsilon_\kappa-\tilde{\bm{\beta}}(\bm{\zeta}^\varepsilon_\kappa)-\bm{\zeta}^\varepsilon\|_{\bm{V}_M(\omega)},
\end{equation*}
where the latter inequality holds thanks to~\eqref{conclusion-5.5}. Since $P(\bm{\zeta}^\varepsilon_\kappa)\in \bm{U}_M(\omega)$, an application of the uniform positive definiteness of the fourth order two-dimensional elasticity tensor $(a^{\alpha\beta\sigma\tau})$ (Theorem~3.1-1 of~\cite{Ciarlet2000}), Korn's inequality (Theorem~\ref{korn}) and~\eqref{conclusion-5.5} gives
\begin{align*}
&\dfrac{\varepsilon\sqrt{a_0}}{c_e c_0}\|P(\bm{\zeta}^\varepsilon_\kappa)-\bm{\zeta}^\varepsilon\|_{\bm{V}_M(\omega)}
\le \varepsilon\int_{\omega}a^{\alpha\beta\sigma\tau} \gamma_{\sigma \tau}(P(\bm{\zeta}^\varepsilon_\kappa)-\bm{\zeta}^\varepsilon) \gamma_{\alpha\beta}(P(\bm{\zeta}^\varepsilon_\kappa)-\bm{\zeta}^\varepsilon) \sqrt{a} \dd y\\
&\le-\int_{\omega} \bm{p}^{\varepsilon} \cdot (P(\bm{\zeta}^\varepsilon_\kappa)-\bm{\zeta}^\varepsilon) \sqrt{a} \dd y
+\varepsilon\int_{\omega} a^{\alpha\beta\sigma\tau} \gamma_{\sigma \tau}(P(\bm{\zeta}^\varepsilon_\kappa)) \gamma_{\alpha\beta}(P(\bm{\zeta}^\varepsilon_\kappa)-\bm{\zeta}^\varepsilon) \sqrt{a} \dd y\\
&=-\int_{\omega} \bm{p}^{\varepsilon} \cdot (P(\bm{\zeta}^\varepsilon_\kappa)-\bm{\zeta}^\varepsilon) \sqrt{a} \dd y
-\dfrac{\varepsilon}{\kappa} \int_{\omega} \bm{\beta}(\bm{\zeta}^\varepsilon_\kappa) \cdot (P(\bm{\zeta}^\varepsilon_\kappa)-\bm{\zeta}^\varepsilon) \dd y
+\int_{\omega} \bm{p}^{\varepsilon} \cdot (P(\bm{\zeta}^\varepsilon_\kappa)-\bm{\zeta}^\varepsilon) \sqrt{a} \dd y\\
&\quad-\varepsilon\int_{\omega}a^{\alpha\beta\sigma\tau} \gamma_{\sigma \tau}(\tilde{\bm{\beta}}(\bm{\zeta}^\varepsilon_\kappa)) \gamma_{\alpha\beta}(P(\bm{\zeta}^\varepsilon_\kappa)-\bm{\zeta}^\varepsilon) \sqrt{a} \dd y\\
&=\dfrac{\varepsilon}{\kappa}\int_{\omega}\left(-\{(\bm{\theta}+\zeta^\varepsilon_{\kappa,j}\bm{a}^j)\cdot\bm{q}\}^{-}\right) \left( \dfrac{\zeta^\varepsilon_i\bm{a}^i\cdot\bm{q}}{\sqrt{\sum_{\ell=1}^3|\bm{a}^\ell\cdot\bm{q}|^2}}\right) \dd y
-\dfrac{\varepsilon}{\kappa}\int_{\omega}\left(-\{(\bm{\theta}+\zeta^\varepsilon_{\kappa,j}\bm{a}^j)\cdot\bm{q}\}^{-}\right) \left( \dfrac{\zeta^\varepsilon_{\kappa,i}\bm{a}^i\cdot\bm{q}}{\sqrt{\sum_{\ell=1}^3|\bm{a}^\ell\cdot\bm{q}|^2}}\right) \dd y\\
&\quad+\dfrac{\varepsilon}{\kappa}\int_{\omega} \left(-\{(\bm{\theta}+\zeta^\varepsilon_{\kappa,j}\bm{a}^j)\cdot\bm{q}\}^{-} \dfrac{\bm{a}^i\cdot\bm{q}}{\sqrt{\sum_{\ell=1}^3|\bm{a}^\ell\cdot\bm{q}|^2}}\right)_{i=1}^3 \cdot \left(-\{(\bm{\theta}+\zeta^\varepsilon_{\kappa,j}\bm{a}^j)\cdot\bm{q}\}^{-} \dfrac{\bm{a}^i\cdot\bm{q}}{\sum_{\ell=1}^3|\bm{a}^\ell\cdot\bm{q}|^2}\right)_{i=1}^3 \dd y\\
&\quad-\varepsilon\int_{\omega}a^{\alpha\beta\sigma\tau} \gamma_{\sigma \tau}(\tilde{\bm{\beta}}(\bm{\zeta}^\varepsilon_\kappa)) \gamma_{\alpha\beta}(P(\bm{\zeta}^\varepsilon_\kappa)-\bm{\zeta}^\varepsilon) \sqrt{a} \dd y\\
&\le-\dfrac{\varepsilon}{\kappa}\int_{\omega}\left(-\{(\bm{\theta}+\zeta^\varepsilon_{\kappa,j}\bm{a}^j)\cdot\bm{q}\}^{-}\right) \left( \dfrac{\bm{\theta}\cdot\bm{q}}{\sqrt{\sum_{\ell=1}^3|\bm{a}^\ell\cdot\bm{q}|^2}}\right) \dd y
+\dfrac{\varepsilon}{\kappa}\int_{\omega}\left(-\{(\bm{\theta}+\zeta^\varepsilon_{\kappa,j}\bm{a}^j)\cdot\bm{q}\}^{-}\right) \left( \dfrac{\bm{\theta}\cdot\bm{q}}{\sqrt{\sum_{\ell=1}^3|\bm{a}^\ell\cdot\bm{q}|^2}}\right) \dd y\\
&\quad-\dfrac{\varepsilon}{\kappa}\int_{\omega}\dfrac{|-\{(\bm{\theta}+\zeta^\varepsilon_{\kappa,j}\bm{a}^j)\cdot\bm{q}\}^{-}|^2}{\sqrt{\sum_{\ell=1}^3|\bm{a}^\ell\cdot\bm{q}|^2}} \dd y+\dfrac{\varepsilon}{\kappa}\int_{\omega}\dfrac{|-\{(\bm{\theta}+\zeta^\varepsilon_{\kappa,j}\bm{a}^j)\cdot\bm{q}\}^{-}|^2}{\sqrt{\sum_{\ell=1}^3|\bm{a}^\ell\cdot\bm{q}|^2}} \dd y\\
&\quad-\varepsilon\int_{\omega}a^{\alpha\beta\sigma\tau} \gamma_{\sigma \tau}(\tilde{\bm{\beta}}(\bm{\zeta}^\varepsilon_\kappa)) \gamma_{\alpha\beta}(P(\bm{\zeta}^\varepsilon_\kappa)-\bm{\zeta}^\varepsilon) \sqrt{a} \dd y\\
&=-\varepsilon\int_{\omega}a^{\alpha\beta\sigma\tau} \gamma_{\sigma \tau}(\tilde{\bm{\beta}}(\bm{\zeta}^\varepsilon_\kappa)) \gamma_{\alpha\beta}(P(\bm{\zeta}^\varepsilon_\kappa)-\bm{\zeta}^\varepsilon) \sqrt{a} \dd y\le M \varepsilon \|\tilde{\bm{\beta}}(\bm{\zeta}^\varepsilon_\kappa)\|_{\bm{H}^1_0(\omega)} \|P(\bm{\zeta}^\varepsilon_\kappa)-\bm{\zeta}^\varepsilon\|_{\bm{V}_M(\omega)} \sqrt{a_1}\\
&\le M C \sqrt{a_1} \varepsilon \kappa \|P(\bm{\zeta}^\varepsilon_\kappa)-\bm{\zeta}^\varepsilon\|_{\bm{V}_M(\omega)}.
\end{align*}

In conclusion, we have that
\begin{equation*}
\|P(\bm{\zeta}^\varepsilon_\kappa)-\bm{\zeta}^\varepsilon\|_{\bm{V}_M(\omega)} \le M C c_0 c_e\dfrac{\sqrt{a_1}}{\sqrt{a}_0} \kappa,
\end{equation*}
so that
\begin{equation*}
\|\bm{\zeta}^\varepsilon_\kappa-\bm{\zeta}^\varepsilon\|_{\bm{V}_M(\omega)} \le C \sqrt{\kappa},
\end{equation*}
for some $C>0$ independent of $\varepsilon$ and $\kappa$.
\end{proof}

We note in passing that the proof of Theorem~\ref{th:beta-6} was established by just assuming that $\min_{y \in \overline{\omega}}(\bm{a}^3\cdot\bm{q})>0$.
Our assumption appears to be more realistic than the abstract assumption $(\ast)$ introduced on page~299 of Scholz's seminal paper~\cite{Scholz1984}.
Moreover, we notice that the conclusions in Lemma~3 and Theorem~4 of~\cite{Scholz1984} continue to hold in the vector-valued case.

\section{Numerical approximation of the solution of Problem~\ref{problem2} via the Finite Element Method}
\label{approx:penalty}

In this section we present a suitable Finite Element Method to approximate the solution to Problem~\ref{problem1}. 
Following~\cite{PGCFEM} and~\cite{Brenner2008} (see also~\cite{ChaBat2011}, \cite{CheGloLi2003}, \cite{Ganesan2017} and~\cite{LiHuaAHuaQ2015}), we recall some basic terminology and definitions. 
In what follows the letter $h$ denotes a quantity approaching zero. For brevity, the same notation $C$ (with or without subscripts) designates a positive constant independent of $\varepsilon$, $\kappa$ and $h$, which can take different values at different places.
We denote by $(\mathcal{T}_h)_{h>0}$ a \emph{family of triangulations of the polygonal domain} $\overline{\omega}$ made of triangles and we let $T$ denote any element of such a family.
Let us first recall, following~\cite{Brenner2008} and~\cite{PGCFEM}, the \emph{rigorous} definition of \emph{finite element} in $\mathbb{R}^n$, where $n \ge 1$ is an integer. A \emph{finite element} in $\mathbb{R}^n$ is a \emph{triple}
$(T,P, \mathcal{N})$ where:

(i) $T$ is a closed subset of $\mathbb{R}^n$ with non-empty interior and Lipschitz-continuous boundary,

(ii) $P$ is a finite dimensional space of real-valued functions defined over $T$,

(iii) $\mathcal{N}$ is is a finite set of linearly independent linear forms $N_i$, $1 \le i \le \dim P$, defined over the space $P$.

By definition, it is assumed that the set $\mathcal{N}$ is \emph{$P$-unisolvent} in the following sense: given any real scalars $\alpha_i$, $1\le i \le \dim P$, there exists a unique function $g \in P$ which satisfies
$$
N_i(g)=\alpha_i, \quad 1 \le i \le \dim P.
$$

It is henceforth assumed that the \emph{degrees of freedom}, $N_i$ , lie in the dual space of a function space larger than $P$ like, for instance, a Sobolev space (see~\cite{Brenner2008}).
For brevity we shall conform our terminology to the one of~\cite{PGCFEM}, calling the sole set $T$ a finite element.
Define the \emph{diameter} of any finite element $T$ as follows:
$$
h_T=\text{diam }T:= \max_{x,y \in T} |x-y|.
$$

Let us also define
$$
\rho_T:=\sup\{\text{diam }B; B \textup{ is a ball contained in }T\}.
$$

A triangulation $\mathcal{T}_h$ is said to be \emph{regular} (cf., e.g., \cite{PGCFEM}) if:

(i) There exists a constant $\sigma>0$, independent of $h$, such that
$$
\textup{for all }T \in \mathcal{T}_h,\quad \dfrac{h_T}{\rho_T} \le \sigma.
$$

(ii) The quantity $h:=\max\{h_T>0; T \in \mathcal{T}_h\} $ approaches zero.

A triangulation $\mathcal{T}_h$ is said to satisfy \emph{an inverse assumption} (cf., e.g., \cite{PGCFEM}) if there exists a constant $\upsilon>0$ such that
$$
\textup{for all }T \in \mathcal{T}_h,\quad \dfrac{h}{h_T} \le \upsilon.
$$

We assume that the finite elements $(K, P_K, \Sigma_K)$, $K \in \bigcup_{h>0}\mathcal{T}_h$, are of class $\mathcal{C}^0$ and are affine (cf. Section~2.3 of~\cite{PGCFEM}), in the sense that they are affine equivalent to a single reference element $(\hat{K}, \hat{P}, \hat{\Sigma})$.


The forthcoming finite element analysis will be carried out using triangles of type $(1)$ (see Figure~2.2.1 of~\cite{PGCFEM}) to approximate the components of the solution of Problem~\ref{problem2}. In this case, the set $\mathcal{V}_h$ consists of all the vertices of the triangulation $\mathcal{T}_h$.

Let $V_{1,h}$, $V_{2,h}$ and $V_{3,h}$ be three finite dimensional spaces such that $V_{\alpha,h}\subset H^1_0(\omega)$ and $V_{3,h} \subset L^2(\omega)$.
Define
$$
\bm{V}_h:=V_{1,h} \times V_{2,h}\times V_{3,h},
$$
and observe that $\bm{V}_h \subset \bm{V}_M(\omega)$.

Let us now define the $\bm{V}_h$ interpolation operator $\bm{\Pi}_h:\bm{\mathcal{C}}^0(\overline{\omega})\to\bm{V}_h$ as follows
$$
\bm{\Pi}_h \bm{\xi}:=\left(\Pi_{1,h} \xi_1, \Pi_{2,h} \xi_2, \Pi_{3,h} \xi_3\right)\quad\textup{ for all }\bm{\xi}=(\xi_i)\in \bm{\mathcal{C}}^0(\overline{\omega}),
$$
where $\Pi_{i,h}$ is the standard $V_{i,h}$ interpolation operator (cf., e.g., \cite{PGCFEM} and~\cite{Brenner2008}). 
It thus results that the interpolation operator $\bm{\Pi}_h$ satisfies the following properties
\begin{equation*}
(\Pi_{j,h} \xi_j)(p)=\xi_j(p)\quad\textup{ for all integers }1 \le j \le 3 \textup{ and all vertices }p \in \mathcal{V}_h,
\end{equation*}
where $\nu_e$ is outer unit normal vector to the edge $e$.
Recall that
$$
\bm{H}(\omega)=H^2(\omega)\times H^2(\omega) \times H^1(\omega)
$$
and that it is equipped with the norm:
$$
\|\bm{\xi}\|_{\bm{H}(\omega)}=\|\xi_1\|_{H^2(\omega)}+\|\xi_2\|_{H^2(\omega)}+\|\xi_3\|_{H^1(\omega)} \quad\textup{ for all } \bm{\xi}=(\xi_i)\in \bm{H}(\omega).
$$

An application of Theorem~3.2.1 of~\cite{PGCFEM} (see also Theorem~4.4.20 of~\cite{Brenner2008}) yields
\begin{equation}
\label{Pih}
\|\bm{\xi}-\bm{\Pi}_h \bm{\xi}\|_{\bm{V}_M(\omega)} \le C h |\bm{\xi}|_{\bm{H}(\omega)},
\end{equation}
for all $\bm{\xi}\in \bm{H}(\omega)\cap \bm{V}_M(\omega)$, where $|\cdot|_{\bm{H}(\omega)}$ denotes the semi-norm associated with the norm $\|\cdot\|_{\bm{H}(\omega)}$.

For each $h>0$, denote the discretization of the elliptic operator $\bm{A}^\varepsilon:\bm{V}_M(\omega) \to \bm{V}'_M(\omega)$ over the triangulation $\mathcal{T}_h$ by $\bm{A}^{\varepsilon,h}$. We have that the linear mapping $\bm{A}^{\varepsilon,h}: \bm{V}_h \to \bm{V}_h$ si defined by
\begin{equation*}
\langle \bm{A}^{\varepsilon,h}\bm{\eta},\bm{\xi}\rangle_{\bm{V}'_M(\omega), \bm{V}_M(\omega)}:=\varepsilon\int_{\omega} a^{\alpha\beta\sigma\tau} \gamma_{\sigma\tau}(\bm{\eta})\gamma_{\alpha\beta}(\bm{\xi})\sqrt{a}\dd y,\quad\textup{ for all }\bm{\eta}, \bm{\xi} \in \bm{V}_h.
\end{equation*}

For each $h>0$, denote the projection of $\bm{L}^2(\omega)$ onto $\bm{V}_h$ by $\bm{P}^h$. We have that the mapping $\bm{P}^h:\bm{L}^2(\omega) \to \bm{V}_h$ is defined by
\begin{equation*}
\bm{P}^h(\bm{\eta}):=\sum_{\ell=1}^{\dim \bm{V}_h} \left(\int_{\omega} \bm{\eta}\cdot\bm{e}_\ell\dd y\right)\bm{e}_\ell,
\end{equation*}
where $\{\bm{e}_\ell\}_{\ell=1}^\infty$ is a Hilbert basis in $\bm{L}^2(\omega)$. We observe that the projection is defined in terms of the Fourier series of $\bm{\eta}$ (viz. Theorem~4.9-1 of~\cite{PGCLNFAA}).

The discretized version of Problem~\ref{problem2} is formulated as follows.

\begin{customprob}{$\mathcal{P}_{M,\kappa}^{\varepsilon,h}(\omega)$}
	\label{problem3}
	Find $\bm{\zeta}^{\varepsilon,h}_\kappa=(\zeta^{\varepsilon,h}_{\kappa,i}) \in \bm{V}_h$ satisfying the following variational equations:
	\begin{equation*}
	\varepsilon \int_\omega a^{\alpha \beta \sigma \tau} \gamma_{\sigma \tau}(\bm{\zeta}^{\varepsilon,h}_\kappa) \gamma_{\alpha \beta} (\bm{\eta}) \sqrt{a} \dd y 
	+\dfrac{\varepsilon}{\kappa}\int_{\omega} \bm{\beta}(\bm{\zeta}^{\varepsilon,h}_\kappa) \cdot \bm{\eta} \dd y
	= \int_\omega p^{i,\varepsilon} \eta_i \sqrt{a} \dd y,
	\end{equation*}
	for all $\bm{\eta} = (\eta_i) \in \bm{V}_h$.
	\bqed
\end{customprob}

It can be shown, thanks to an argument similar to the one exploited for establishing Theorem~\ref{ex-un-kappa}, that Problem~\ref{problem3} admits a unique solution $\bm{\zeta}^{\varepsilon,h}_\kappa \in\bm{V}_h$.

\begin{theorem}
\label{th:conv}
Let $\bm{\zeta}^\varepsilon_\kappa \in \bm{V}_M(\omega) \cap \bm{H}(\omega)$ be the solution of Problem~\ref{problem2}, and let $\bm{\zeta}^{\varepsilon,h}_\kappa \in \bm{V}_h$ be the solution of Problem~\ref{problem3}. Then there exists a constant $\hat{C}>0$ independent of $\varepsilon$, $\kappa$ and $h$ for which the following estimate holds
\begin{equation*}
\label{est:conv}
\|\bm{\zeta}^\varepsilon_\kappa-\bm{\zeta}^{\varepsilon,h}_\kappa\|_{\bm{V}_M(\omega)}
\le \hat{C} h\left(1+\dfrac{1}{\sqrt{\kappa}}\right).
\end{equation*}
\end{theorem}
\begin{proof}
Thanks to the boundedness of the sequences $\{\bm{\zeta}^\varepsilon_\kappa\}_{\kappa>0}$ and $\{\bm{\zeta}^\varepsilon\}_{\varepsilon>0}$ in $\bm{H}(\omega)$ (Theorem~\ref{ex-un-kappa}, Theorem~\ref{t:4}, Theorem~\ref{aug:int}, Theorem~\ref{aug:bdry}, \eqref{conclusion-7} and~\eqref{Pih}), we have that
\begin{equation*}
\|\bm{\zeta}^\varepsilon_\kappa-\bm{\Pi}_h \bm{\zeta}^\varepsilon_\kappa\|_{\bm{V}_M(\omega)} \le C h |\bm{\zeta}^\varepsilon_\kappa|_{\bm{H}(\omega)},
\end{equation*}
and the semi-norm on the right-hand side is bounded independently of $\kappa$ and $\varepsilon$ (see the remark after Theorem~\ref{aug:bdry}).

Thanks to the calculations carried out in Lemma~\ref{lem:beta} for establishing the monotonicity of the operator $\bm{\beta}$, we have that
\begin{align*}
&\dfrac{\sqrt{a_0}\varepsilon}{c_0 c_e}\|\bm{\zeta}^\varepsilon_\kappa - \bm{\zeta}^{\varepsilon,h}_\kappa\|_{\bm{V}_M(\omega)}^2 +
\dfrac{\varepsilon\kappa^{-1}}{\sqrt{3 \max\{\|\bm{a}^\ell \cdot\bm{q}\|_{\mathcal{C}^0(\overline{\omega})}^2;1\le \ell \le 3\}}}\int_{\omega}\left|\left(-\{(\bm{\theta}+\zeta^\varepsilon_{\kappa,j}\bm{a}^j)\cdot\bm{q}\}^{-}\right) - \left(-\{(\bm{\theta}+\zeta^{\varepsilon,h}_{\kappa,j}\bm{a}^j)\cdot\bm{q}\}^{-}\right)\right|^2\dd y\\
&\le\varepsilon\int_{\omega} a^{\alpha\beta\sigma\tau}\gamma_{\sigma\tau}(\bm{\zeta}^\varepsilon_\kappa - \bm{\zeta}^{\varepsilon,h}_\kappa)\gamma_{\alpha \beta}(\bm{\zeta}^\varepsilon_\kappa - \bm{\zeta}^{\varepsilon,h}_\kappa)\sqrt{a} \dd y\\
&\quad+\dfrac{\varepsilon}{\kappa}\int_{\omega} \left(\left[-\{(\bm{\theta}+\zeta^\varepsilon_{\kappa,j}\bm{a}^j)\cdot\bm{q}\}^{-}\right] - \left[-\{(\bm{\theta}+\zeta^{\varepsilon,h}_{\kappa,j}\bm{a}^j)\cdot\bm{q}\}^{-}\right]\right) \left(\dfrac{(\zeta^\varepsilon_{\kappa,i}-\zeta^{\varepsilon,h}_{\kappa,i})\bm{a}^i\cdot\bm{q}}{\sqrt{\sum_{\ell=1}^{3}|\bm{a}^\ell \cdot\bm{q}|^2}}\right) \dd y\\
&\le \dfrac{\sqrt{a_0}\varepsilon}{2 c_0 c_e}\|\bm{\zeta}^\varepsilon_\kappa-\bm{\zeta}^{\varepsilon,h}_\kappa\|_{\bm{V}_M(\omega)}^2
+\dfrac{\varepsilon c_0 c_e a_1}{2\sqrt{a_0}}\|\bm{\zeta}^\varepsilon_\kappa-\bm{\zeta}^{\varepsilon,h}_\kappa\|_{\bm{V}_M(\omega)}^2\\
&\quad+\dfrac{\varepsilon\kappa^{-1}}{2\sqrt{3 \max\{\|\bm{a}^\ell \cdot\bm{q}\|_{\mathcal{C}^0(\overline{\omega})}^2;1\le \ell \le 3\}}}
\int_{\omega}\left|\left[-\{(\bm{\theta}+\zeta^\varepsilon_{\kappa,j}\bm{a}^j)\cdot\bm{q}\}^{-}\right] - \left[-\{(\bm{\theta}+\zeta^{\varepsilon,h}_{\kappa,j}\bm{a}^j)\cdot\bm{q}\}^{-}\right]\right|^2\dd y\\
&\quad+\varepsilon\dfrac{3 \max\{\|\bm{a}^\ell \cdot\bm{q}\|_{\mathcal{C}^0(\overline{\omega})}^2;1\le \ell \le 3\}}{2\kappa \min_{y \in \overline{\omega}}(\bm{a}^3\cdot\bm{q})}\|\bm{\zeta}^\varepsilon_\kappa-\bm{\zeta}^{\varepsilon,h}_\kappa\|_{\bm{L}^2(\omega)}^2.
\end{align*}

Combining the latter inequalities with Cea's lemma (cf., e.g., Theorem~2.4.1 of~\cite{PGCFEM}) and~\eqref{Pih} gives
\begin{equation*}
\begin{aligned}
&\dfrac{\sqrt{a_0}\varepsilon}{2c_0 c_e}\|\bm{\zeta}^\varepsilon_\kappa - \bm{\zeta}^{\varepsilon,h}_\kappa\|_{\bm{V}_M(\omega)}^2 +
\dfrac{\varepsilon\kappa^{-1}}{2\sqrt{3 \max\{\|\bm{a}^\ell \cdot\bm{q}\|_{\mathcal{C}^0(\overline{\omega})}^2;1\le \ell \le 3\}}}\int_{\omega}\left|\left(-\{(\bm{\theta}+\zeta^\varepsilon_{\kappa,j}\bm{a}^j)\cdot\bm{q}\}^{-}\right) - \left(-\{(\bm{\theta}+\zeta^{\varepsilon,h}_{\kappa,j}\bm{a}^j)\cdot\bm{q}\}^{-}\right)\right|^2\dd y\\
&\le\dfrac{\varepsilon c_0 c_e a_1}{2\sqrt{a_0}}\|\bm{\zeta}^\varepsilon_\kappa-\bm{\Pi}_h\bm{\zeta}^\varepsilon_\kappa\|_{\bm{V}_M(\omega)}^2
+\varepsilon\dfrac{3 \max\{\|\bm{a}^\ell \cdot\bm{q}\|_{\mathcal{C}^0(\overline{\omega})}^2;1\le \ell \le 3\}}{2\kappa \min_{y \in \overline{\omega}}(\bm{a}^3\cdot\bm{q})}\|\bm{\zeta}^\varepsilon_\kappa-\bm{\Pi}_h\bm{\zeta}^\varepsilon_\kappa\|_{\bm{L}^2(\omega)}^2.
\end{aligned}
\end{equation*}

Letting
$$
\tilde{C}^2:=\max\left\{\dfrac{c_0 c_e a_1}{\sqrt{a_0}},\dfrac{3 \max\{\|\bm{a}^\ell \cdot\bm{q}\|_{\mathcal{C}^0(\overline{\omega})}^2;1\le \ell \le 3\}}{ \min_{y \in \overline{\omega}}(\bm{a}^3\cdot\bm{q})}\right\}C,
$$
where $C>0$ is the constant appearing in~\eqref{Pih} or, equivalently, in Theorem~3.2.1 of~\cite{PGCFEM}, we obtain the estimate
\begin{equation*}
\|\bm{\zeta}^\varepsilon_\kappa - \bm{\zeta}^{\varepsilon,h}_\kappa\|_{\bm{V}_M(\omega)}^2
\le \tilde{C}^2 h^2\left(1+\dfrac{1}{\kappa}\right)|\bm{\zeta}^\varepsilon_\kappa|_{\bm{H}(\omega)}^2,
\end{equation*}
which, together with Lemma~3.2 page~260 of~\cite{Nec67}, straightforwardly leads to the conclusion.
\end{proof}

\section{Numerical approximation of the solution of Problem~\ref{problem2} via the Brezis-Sibony iteration scheme}
\label{approx:BrezisSibony}

In view of Theorem~\ref{th:conv}, we are in position to define the discrete nonlinear operator $\bm{N}^\varepsilon_h:\bm{V}_h \to\bm{V}_h$ by:
\begin{equation*}
\bm{N}^\varepsilon_h(\bm{\eta})=\bm{A}^\varepsilon_h\bm{\eta}+\dfrac{\varepsilon}{h^q}\bm{P}_h(\bm{\beta}(\bm{\eta}))-\bm{P}_h(\bm{p}^\varepsilon\sqrt{a}),
\end{equation*}
where the specialization $\kappa:=h^q$, with $0< q <2$ ensures the convergence of the sequence of solutions of Problem~\ref{problem3} to the solution of Problem~\ref{problem2} (Theorem~\ref{th:conv}).

We have that if $\bm{\zeta}^{\varepsilon,h}_\kappa$ is the solution of Problem~\ref{problem3}, then we have that $\bm{N}^\varepsilon_h(\bm{\zeta}^{\varepsilon,h}_\kappa)=\bm{0}$.

In this section we extend the validity of the scheme proposed by Brezis \& Sibony in~\cite{BrezisSibony1968} to approximate the solution of Problem~\ref{problem3} by means of an iterative pattern.

Critical to establishing the sought convergence is the inverse assumption stated in section~\ref{approx:penalty} which, we notice, was not exploitd to carry out the proof of Theorem~\ref{th:conv}. As a consequence of Theorem~3.2.6 of~\cite{PGCFEM} we have that the following \emph{inverse inequality} holds.

\begin{lemma}
\label{inv:ineq}
Let $h>0$ be given and let $\mathcal{T}_h$ be a regular triangulation of $\omega$ made of affine elements of class $\mathcal{C}^0$ (viz. section~\ref{approx:penalty}).
Then, the following inverse inequality holds
\begin{equation*}
\left(\sum_{K\in\mathcal{T}_h}|\bm{\eta}_h|_{\bm{V}_M(K)}^2\right)^{1/2} \le C_{\textup{inv}} h^{-1} \left(\sum_{K\in\mathcal{T}_h}|\bm{\eta}_h|_{\bm{L}^2(K)}^2\right)^{1/2},\quad\textup{ for all }\bm{\eta}_h\in \bm{V}_h,
\end{equation*}
for some $C_{\textup{inv}}>0$ independent of $h$.
\end{lemma}
\begin{proof}
An application of Theorem~3.2.6 of~\cite{PGCFEM} gives
\begin{equation*}
\left(\sum_{K\in\mathcal{T}_h}|\eta_h|_{H^1(K)}^2\right)^{1/2} \le C h^{-1} \left(\sum_{K\in\mathcal{T}_h}|\eta_h|_{L^2(K)}^2\right)^{1/2},
\end{equation*}
and the sought estimate derives straightforwardly.
\end{proof}

We are thus in position to establish the main result of this section, namely, the convergence of the Brezis-Sibony scheme for Problem~\ref{problem3}.

\begin{theorem}
\label{BrezisSybony}
Let us define, for the sake of simplicity, the vector field $\hat{\bm{\psi}}$ as follows
\begin{equation}
\label{psi:hat}
\hat{\bm{\psi}}:=\bm{\zeta}^{\varepsilon,h}_\kappa,
\end{equation}
and we let $\bm{\psi}_0 \in \bm{V}_h$ be arbitrarily chosen. Let $c_0>0$ be the constant of Korn's inequality (Theorem~\ref{korn}), let $c_e>0$ the constant associated with the uniform positive-definiteness of the fourth order two-dimensional elasticity tensor $(a^{\alpha\beta\sigma\tau})$, let $C_{\textup{inv}}>0$ be the constant associated with the inverse property (Theorem~\ref{inv:ineq}), let $M>0$ be the sup norm of the fourth order two-dimensional elasticity tensor $(a^{\alpha\beta\sigma\tau})$, and let $a_0>0$ and $a_1>0$ be, respectively, the minimum and maximum of the function $a=\det(a_{\alpha\beta})$ introduced in section~\ref{sec1}.

Then, there exists a positive number $\Xi>0$ such that the sequence of vector fields $\{\bm{\psi}_k\}_{k=0}^\infty \subset \bm{V}_h$ defined by
\begin{equation}
\label{iterate}
\bm{\psi}_{k+1}:=\bm{\psi}_k-\Xi h^4 \bm{N}^\varepsilon_h(\bm{\psi}_k),
\end{equation}
satisfies
\begin{equation}
\label{conv:iter}
\|\hat{\bm{\psi}}-\bm{\psi}_{k+1}\|_{\bm{L}^2(\omega)}\le \sqrt{1-\rho'}\|\hat{\bm{\psi}}-\bm{\psi}_k\|_{\bm{L}^2(\omega)}, \quad\textup{ for all }k\ge 0,
\end{equation}
for some $\rho'=\rho'(h,\Xi) \in (0,1)$, whenever $h>0$ is such that 
\begin{equation}
\label{h:bound}
h <\sqrt{\dfrac{c_0 c_e\left(MC_{\textup{inv}}^2\sqrt{a_1}+1\right)}{\sqrt{a_0}}},
\end{equation}
and $\Xi>0$ is such that
\begin{equation}
\label{xi:bound}
\Xi<\frac{2\sqrt{a_0}}{c_0 c_e \left(MC_{\textup{inv}}^2\sqrt{a_1}+1\right)^2}.
\end{equation}
\end{theorem}
\begin{proof}
To begin with, thanks to~\eqref{iterate} and the fact that $\bm{N}^\varepsilon_h(\hat{\bm{\psi}})=\bm{0}$ by~\eqref{psi:hat}, we compute
\begin{align*}
&\hat{\bm{\psi}}-\bm{\psi}_{k+1}=\hat{\bm{\psi}}-\bm{\psi}_k+h^4 \Xi \bm{N}^\varepsilon_h(\bm{\psi}_k)
=\hat{\bm{\psi}}-\bm{\psi}_k-h^4 \Xi \left(\bm{N}^\varepsilon_h(\hat{\bm{\psi}})-\bm{N}^\varepsilon_h(\bm{\psi}_k)\right)\\
&=\hat{\bm{\psi}}-\bm{\psi}_k-h^4 \Xi\left[\left(\bm{A}^\varepsilon_h \hat{\bm{\psi}} +\varepsilon h^{-q} \bm{P}_h(\bm{\beta}(\hat{\bm{\psi}}))-\bm{P}_h(\bm{p}^\varepsilon\sqrt{a})\right)-\left(\bm{A}^\varepsilon_h \bm{\psi}_k +\varepsilon h^{-q} \bm{P}_h(\bm{\beta}(\bm{\psi}_k))-\bm{P}_h(\bm{p}^\varepsilon \sqrt{a})\right)\right]\\
&=\hat{\bm{\psi}}-\bm{\psi}_k-h^4 \Xi\left[\bm{A}^\varepsilon_h(\hat{\bm{\psi}}-\bm{\psi}_k)+\varepsilon h^{-q}\bm{P}_h\left(\bm{\beta}(\hat{\bm{\psi}})-\bm{\beta}(\bm{\psi}_k)\right)\right].
\end{align*}

Define the operator $\bm{Q}_h:\bm{V}_h \to\bm{V}_h$ by
\begin{equation*}
\bm{Q}_h(\bm{\eta}):=\bm{A}^\varepsilon_h\bm{\eta}+\varepsilon h^{-q}\bm{P}_h\left(\bm{\beta}(\bm{\eta})\right),\quad\textup{ for all }\bm{\eta}\in\bm{V}_h.
\end{equation*}

Thanks to this newly introduced definition we can thus write
\begin{equation}
\label{BS1}
\hat{\bm{\psi}}-\bm{\psi}_{k+1}=\hat{\bm{\psi}}-\bm{\psi}_k-h^4 \Xi\left[\bm{Q}_h(\hat{\bm{\psi}})-\bm{Q}_h(\bm{\psi}_k)\right],\quad\textup{ for all }k\ge0.
\end{equation}

In view of~\eqref{BS1}, the uniform positive-definiteness of the fourth order two-dimensional elasticity tensor $(a^{\alpha\beta\sigma\tau})$ (Theorem~3.3-1 of~\cite{Ciarlet2000}), Korn's inequality (Theorem~\ref{korn}), we compute
\begin{equation}
\label{BS1.1}
\begin{aligned}
&\|\hat{\bm{\psi}}-\bm{\psi}_{k+1}\|_{\bm{L}^2(\omega)}^2=\|\hat{\bm{\psi}}-\bm{\psi}_k\|_{\bm{L}^2(\omega)}^2
+h^8 \Xi^2\|\bm{Q}_h(\hat{\bm{\psi}})-\bm{Q}_h(\bm{\psi}_k)\|_{\bm{L}^2(\omega)}^2\\
&\quad-2h^4 \Xi \int_{\omega} (\hat{\bm{\psi}}-\bm{\psi}_k) \cdot \left(\bm{Q}_h(\hat{\bm{\psi}})-\bm{Q}_h(\bm{\psi}_k)\right) \dd y\\
&=\|\hat{\bm{\psi}}-\bm{\psi}_k\|_{\bm{L}^2(\omega)}^2
+h^8 \Xi^2\|\bm{Q}_h(\hat{\bm{\psi}})-\bm{Q}_h(\bm{\psi}_k)\|_{\bm{L}^2(\omega)}^2
-2h^4 \Xi \int_{\omega} (\hat{\bm{\psi}}-\bm{\psi}_k) \cdot \bm{A}^\varepsilon_h(\hat{\bm{\psi}}-\bm{\psi}_k) \dd y\\
&\quad-2\varepsilon h^{4-q} \Xi \int_{\omega} (\hat{\bm{\psi}}-\bm{\psi}_k) \cdot \left(\bm{P}_h(\bm{\beta}(\hat{\bm{\psi}}))-\bm{P}_h(\bm{\beta}(\bm{\psi}_k))\right) \dd y\\
&\le \|\hat{\bm{\psi}}-\bm{\psi}_k\|_{\bm{L}^2(\omega)}^2 +h^8 \Xi^2\|\bm{Q}_h(\hat{\bm{\psi}})-\bm{Q}_h(\bm{\psi}_k)\|_{\bm{L}^2(\omega)}^2
-2 h^4 \Xi \dfrac{\varepsilon\sqrt{a_0}}{c_0 c_e}\|\hat{\bm{\psi}}-\bm{\psi}_k\|_{\bm{L}^2(\omega)}^2\\
&\quad-2 \varepsilon h^{4-q} \Xi \int_{\omega} (\hat{\bm{\psi}}-\bm{\psi}_k) \cdot \bm{P}_h(\bm{\beta}(\hat{\bm{\psi}})-\bm{\beta}(\bm{\psi}_k)) \dd y.
\end{aligned}
\end{equation}

Let $\{\bm{e}_\ell\}_{\ell=1}^{\infty}$ be a Hilbert basis in $\bm{L}^2(\omega)$.
By the theory of Fourier series (cf., e.g., Theorem~4.9-1 of~\cite{PGCLNFAA}), we have that the last integral term can be rewritten as follows:
\begin{equation}
\label{BS1.2}
\begin{aligned}
&\int_{\omega} (\hat{\bm{\psi}}-\bm{\psi}_k) \cdot \left(\bm{P}_h(\bm{\beta}(\hat{\bm{\psi}})-\bm{\beta}(\bm{\psi}_k))\right) \dd y
=\int_{\omega} (\hat{\bm{\psi}}-\bm{\psi}_k) \cdot \left(\sum_{\ell=1}^{\dim\bm{V}_h} \left(\int_{\omega} (\bm{\beta}(\hat{\bm{\psi}})-\bm{\beta}(\bm{\psi}_k)) \cdot \bm{e}_\ell\dd s\right)\bm{e}_\ell\right) \dd y\\
&=\sum_{\ell=1}^{\dim\bm{V}_h}\left\{\left(\int_{\omega} (\hat{\bm{\psi}}-\bm{\psi}_k) \cdot \bm{e}_\ell \dd y\right)\left(\int_{\omega} (\bm{\beta}(\hat{\bm{\psi}})-\bm{\beta}(\bm{\psi}_k)) \cdot \bm{e}_\ell\dd s\right)\right\}\\
&=\int_{\omega}(\bm{\beta}(\hat{\bm{\psi}})-\bm{\beta}(\bm{\psi}_k)) \cdot \left(\sum_{\ell=1}^{\dim\bm{V}_h} \left(\int_{\omega}(\hat{\bm{\psi}}-\bm{\psi}_k) \cdot\bm{e}_\ell \dd y\right)\bm{e}_\ell\right) \dd s
=\int_{\omega}(\bm{\beta}(\hat{\bm{\psi}})-\bm{\beta}(\bm{\psi}_k)) \cdot (\hat{\bm{\psi}}-\bm{\psi}_k)\dd s \ge 0,
\end{aligned}
\end{equation}
where the last equality holds thanks to the fact that $\hat{\bm{\psi}}, \bm{\psi}_k \in\bm{V}_h$.

For each $k\ge 0$, let us now estimate
\begin{align*}
&\|\bm{Q}_h(\hat{\bm{\psi}})-\bm{Q}_h(\bm{\psi}_k)\|_{\bm{L}^2(\omega)}^2
=\left\|\left(\bm{A}^\varepsilon_h\hat{\bm{\psi}}+\varepsilon h^{-q}\bm{P}_h(\bm{\beta}(\hat{\bm{\psi}}))\right)-\left(\bm{A}^\varepsilon_h\bm{\psi}_k+\varepsilon h^{-q}\bm{P}_h(\bm{\beta}(\bm{\psi}_k))\right)\right\|_{\bm{L}^2(\omega)}^2\\
&=\left\|\bm{A}^\varepsilon_h(\hat{\bm{\psi}}-\bm{\psi}_k)+\varepsilon h^{-q}\bm{P}_h(\bm{\beta}(\hat{\bm{\psi}})-\bm{\beta}(\bm{\psi}_k))\right\|_{\bm{L}^2(\omega)}^2
=\|\bm{A}^\varepsilon_h(\hat{\bm{\psi}}-\bm{\psi}_k)\|_{\bm{L}^2(\omega)}^2\\
&\quad+ \varepsilon h^{-2q}\|\bm{P}_h(\bm{\beta}(\hat{\bm{\psi}})-\bm{\beta}(\bm{\psi}_k))\|_{\bm{L}^2(\omega)}^2
+2\varepsilon h^{-q}\int_{\omega} \left(\bm{A}^\varepsilon_h(\hat{\bm{\psi}}-\bm{\psi}_k)\right) \cdot\left(\bm{P}_h(\bm{\beta}(\hat{\bm{\psi}})-\bm{\beta}(\bm{\psi}_k))\right) \dd y\\
&=\varepsilon\int_{\omega}a^{\alpha\beta\sigma\tau} \gamma_{\sigma \tau}(\hat{\bm{\psi}}-\bm{\psi}_k) \gamma_{\alpha\beta}(\bm{A}^\varepsilon_h(\hat{\bm{\psi}}-\bm{\psi}_k))\sqrt{a}\dd y\\
&\quad+\varepsilon h^{-2q}\|\bm{P}_h(\bm{\beta}(\hat{\bm{\psi}})-\bm{\beta}(\bm{\psi}_k))\|_{\bm{L}^2(\omega)}^2
+2\varepsilon h^{-q}\int_{\omega} \left(\bm{A}^\varepsilon_h(\hat{\bm{\psi}}-\bm{\psi}_k)\right) \cdot\left(\bm{P}_h(\bm{\beta}(\hat{\bm{\psi}})-\bm{\beta}(\bm{\psi}_k))\right) \dd y\\
&\le\varepsilon\int_{\omega}a^{\alpha\beta\sigma\tau} \gamma_{\sigma \tau}(\hat{\bm{\psi}}-\bm{\psi}_k) \gamma_{\alpha\beta}(\bm{A}^\varepsilon_h(\hat{\bm{\psi}}-\bm{\psi}_k))\sqrt{a}\dd y
+\varepsilon h^{-2q}\|\bm{P}_h(\bm{\beta}(\hat{\bm{\psi}})-\bm{\beta}(\bm{\psi}_k))\|_{\bm{L}^2(\omega)}^2\\
&\quad+2\varepsilon h^{-q}\|\bm{A}^\varepsilon_h(\hat{\bm{\psi}}-\bm{\psi}_k)\|_{\bm{L}^2(\omega)} \|\bm{P}_h(\bm{\beta}(\hat{\bm{\psi}})-\bm{\beta}(\bm{\psi}_k))\|_{\bm{L}^2(\omega)}\\
&\le M\varepsilon\sqrt{a}_1\|\hat{\bm{\psi}}-\bm{\psi}_k\|_{\bm{V}_M(\omega)} \|\bm{A}^\varepsilon_h(\hat{\bm{\psi}}-\bm{\psi}_k)\|_{\bm{V}_M(\omega)}
+\varepsilon h^{-2q}\|\hat{\bm{\psi}}-\bm{\psi}_k\|_{\bm{L}^2(\omega)}^2\\
&\quad+2\varepsilon h^{-q}\|\bm{A}^\varepsilon_h(\hat{\bm{\psi}}-\bm{\psi}_k)\|_{\bm{L}^2(\omega)} \|\hat{\bm{\psi}}-\bm{\psi}_k\|_{\bm{L}^2(\omega)},
\end{align*}
where the second last estimate is due to the continuity of the bilinear form, and the last estimate is due to the fact that the projection $\bm{P}_h$ and the operator $\bm{\beta}$ are non-expansive mappings (cf., e.g., Theorem~4.3-1(c) of~\cite{PGCLNFAA} and Lemma~\ref{lem:beta}). To sum up, we have shown that
\begin{equation}
\label{BS2}
\begin{aligned}
&\|\bm{Q}_h(\hat{\bm{\psi}})-\bm{Q}_h(\bm{\psi}_k)\|_{\bm{L}^2(\omega)}^2\le M\varepsilon\sqrt{a}_1\|\hat{\bm{\psi}}-\bm{\psi}_k\|_{\bm{V}_M(\omega)} \|\bm{A}^\varepsilon_h(\hat{\bm{\psi}}-\bm{\psi}_k)\|_{\bm{V}_M(\omega)}
+\varepsilon h^{-2q}\|\hat{\bm{\psi}}-\bm{\psi}_k\|_{\bm{L}^2(\omega)}^2\\
&\quad+2\varepsilon h^{-q}\|\bm{A}^\varepsilon_h(\hat{\bm{\psi}}-\bm{\psi}_k)\|_{\bm{L}^2(\omega)} \|\hat{\bm{\psi}}-\bm{\psi}_k\|_{\bm{L}^2(\omega)}.
\end{aligned}
\end{equation}

Thanks to the inverse property (Lemma~\ref{inv:ineq}), we have that
\begin{equation}
\label{BS3}
\|\bm{A}^\varepsilon_h(\hat{\bm{\psi}}-\bm{\psi}_k)\|_{\bm{V}_M(\omega)}\le \dfrac{C_{\textup{inv}}}{h}\|\bm{A}^\varepsilon_h(\hat{\bm{\psi}}-\bm{\psi}_k)\|_{\bm{L}^2(\omega)}.
\end{equation}

An application of~\eqref{BS3} gives
\begin{align*}
&\|\bm{A}^\varepsilon_h(\hat{\bm{\psi}}-\bm{\psi}_k)\|_{\bm{L}^2(\omega)}^2
=\varepsilon\int_{\omega} a^{\alpha\beta\sigma\tau} \gamma_{\sigma\tau}(\hat{\bm{\psi}}-\bm{\psi}_k) \gamma_{\alpha\beta}(\bm{A}^\varepsilon_h(\hat{\bm{\psi}}-\bm{\psi}_k)) \sqrt{a}\dd y\\
&\le M\varepsilon\sqrt{a_1}\|\hat{\bm{\psi}}-\bm{\psi}_k\|_{\bm{V}_M(\omega)} \|\bm{A}^\varepsilon_h(\hat{\bm{\psi}}-\bm{\psi}_k)\|_{\bm{V}_M(\omega)}\\
&\le \dfrac{MC_{\textup{inv}}^2\varepsilon\sqrt{a_1}}{h^2}\|\hat{\bm{\psi}}-\bm{\psi}_k\|_{\bm{L}^2(\omega)}\|\bm{A}^\varepsilon_h(\hat{\bm{\psi}}-\bm{\psi}_k)\|_{\bm{L}^2(\omega)}.
\end{align*}

The latter in turn implies that
\begin{equation}
\label{BS4}
\|\bm{A}^\varepsilon_h(\hat{\bm{\psi}}-\bm{\psi}_k)\|_{\bm{L}^2(\omega)}
\le \dfrac{MC_{\textup{inv}}^2\varepsilon\sqrt{a_1}}{h^2}\|\hat{\bm{\psi}}-\bm{\psi}_k\|_{\bm{L}^2(\omega)}.
\end{equation}

Thanks to~\eqref{BS3}, \eqref{BS4}, the inverse property (Theorem~) and the fact that $0< q<2$, we are able to estimate the right-hand side of~\eqref{BS2} as follows:
\begin{equation*}
\begin{aligned}
&M\varepsilon\sqrt{a}_1\|\hat{\bm{\psi}}-\bm{\psi}_k\|_{\bm{V}_M(\omega)} \|\bm{A}^\varepsilon_h(\hat{\bm{\psi}}-\bm{\psi}_k)\|_{\bm{V}_M(\omega)}
+\varepsilon h^{-2q}\|\hat{\bm{\psi}}-\bm{\psi}_k\|_{\bm{L}^2(\omega)}^2+2\varepsilon h^{-q}\|\bm{A}^\varepsilon_h(\hat{\bm{\psi}}-\bm{\psi}_k)\|_{\bm{L}^2(\omega)} \|\hat{\bm{\psi}}-\bm{\psi}_k\|_{\bm{L}^2(\omega)}\\
&\le \dfrac{MC_{\textup{inv}}^2\varepsilon\sqrt{a_1}}{h^2}\|\hat{\bm{\psi}}-\bm{\psi}_k\|_{\bm{L}^2(\omega)} \|\bm{A}^\varepsilon_h(\hat{\bm{\psi}}-\bm{\psi}_k)\|_{\bm{L}^2(\omega)}
+ \varepsilon h^{-2q}\|\hat{\bm{\psi}}-\bm{\psi}_k\|_{\bm{L}^2(\omega)}^2+\dfrac{2MC_{\textup{inv}}^2\varepsilon^2\sqrt{a_1}}{h^{2+q}} \|\hat{\bm{\psi}}-\bm{\psi}_k\|_{\bm{L}^2(\omega)}^2\\
&\le \left(\dfrac{M^2C_{\textup{inv}}^4\varepsilon^2 a_1}{h^4}+\varepsilon h^{-2q}+\dfrac{2MC_{\textup{inv}}^2\varepsilon^2\sqrt{a_1}}{h^{2+q}}\right)\|\hat{\bm{\psi}}-\bm{\psi}_k\|_{\bm{L}^2(\omega)}^2
\le h^{-4}\left(M^2C_{\textup{inv}}^4\varepsilon^2 a_1+\varepsilon+2MC_{\textup{inv}}^2\varepsilon^2\sqrt{a_1}\right)\|\hat{\bm{\psi}}-\bm{\psi}_k\|_{\bm{L}^2(\omega)}^2\\
&\le h^{-4}\varepsilon\left(MC_{\textup{inv}}^2\sqrt{a_1}+1\right)^2\|\hat{\bm{\psi}}-\bm{\psi}_k\|_{\bm{L}^2(\omega)}^2.
\end{aligned}
\end{equation*}

Combining the latter inequality with~\eqref{BS2} gives at once:
\begin{equation}
\label{BS5}
\|\bm{Q}_h(\hat{\bm{\psi}})-\bm{Q}_h(\bm{\psi}_k)\|_{\bm{L}^2(\omega)}^2
\le h^{-4}\varepsilon\left(MC_{\textup{inv}}^2\sqrt{a_1}+1\right)^2\|\hat{\bm{\psi}}-\bm{\psi}_k\|_{\bm{L}^2(\omega)}^2.
\end{equation}

Therefore, combining~\eqref{BS1.1}, \eqref{BS1.2} and~\eqref{BS5} gives
\begin{equation}
\label{BS6}
\begin{aligned}
&\|\hat{\bm{\psi}}-\bm{\psi}_{k+1}\|_{\bm{L}^2(\omega)}^2
\le \|\hat{\bm{\psi}}-\bm{\psi}_k\|_{\bm{L}^2(\omega)}^2 +h^8 \Xi^2\|\bm{Q}_h(\hat{\bm{\psi}})-\bm{Q}_h(\bm{\psi}_k)\|_{\bm{L}^2(\omega)}^2
-2 h^4 \Xi \dfrac{\varepsilon\sqrt{a_0}}{c_0 c_e}\|\hat{\bm{\psi}}-\bm{\psi}_k\|_{\bm{L}^2(\omega)}^2\\
&\le \left(1-2 h^4  \dfrac{\varepsilon\sqrt{a_0}}{c_0 c_e}\Xi+ \varepsilon h^4 \left(MC_{\textup{inv}}^2\sqrt{a_1}+1\right)^2\Xi^2\right) \|\hat{\bm{\psi}}-\bm{\psi}_k\|_{\bm{L}^2(\omega)}^2.
\end{aligned}
\end{equation}

Let us now consider the polynomial $p(\Xi):=1-2 \varepsilon h^4 \dfrac{\sqrt{a_0}}{c_0 c_e}\Xi+ \varepsilon h^4 \left(MC_{\textup{inv}}^2\sqrt{a_1}+1\right)^2\Xi^2$, and let us observe that its discriminant is such that
\begin{equation*}
\dfrac{\Delta}{4}=h^4\left(\varepsilon^2 \dfrac{a_0}{c_0^2 c_e^2} h^4-\varepsilon\left(MC_{\textup{inv}}^2\sqrt{a_1}+1\right)^2\right)
< \varepsilon h^4 \left(\dfrac{a_0}{c_0^2 c_e^2} h^4-\left(MC_{\textup{inv}}^2\sqrt{a_1}+1\right)^2\right)
\end{equation*}
and it is negative when
\begin{equation*}
h <\sqrt{\dfrac{c_0 c_e\left(MC_{\textup{inv}}^2\sqrt{a_1}+1\right)}{\sqrt{a_0}}}.
\end{equation*}

Therefore, thanks to~\eqref{h:bound}, we have that $p(\Xi)>0$ for all $\Xi\in\mathbb{R}$, on the one hand.

On the other hand, we have that $p(\Xi)<1$ if and only if
\begin{equation*}
\Xi<\frac{2\sqrt{a_0}}{c_0 c_e \left(MC_{\textup{inv}}^2\sqrt{a_1}+1\right)^2},
\end{equation*}
as per our assumption~\eqref{xi:bound}. This means that, under the assumptions~\eqref{h:bound} and~\eqref{xi:bound}, the coefficient on the right-hand side of~\eqref{BS6} is a number between $0$ and $1$.
We thus define the number
\begin{equation*}
\rho':=1-\left(1-2 h^4 \Xi \dfrac{\varepsilon\sqrt{a_0}}{c_0 c_e}+ h^4 \varepsilon\Xi^2 \left(MC_{\textup{inv}}^2\sqrt{a_1}+1\right)^2\right) \in (0,1),
\end{equation*}
and~\eqref{BS6} becomes
\begin{equation*}
\label{BS7}
\|\hat{\bm{\psi}}-\bm{\psi}_{k+1}\|_{\bm{L}^2(\omega)}
\le \sqrt{1-\rho'} \|\hat{\bm{\psi}}-\bm{\psi}_k\|_{\bm{L}^2(\omega)},
\end{equation*}
and the proof is complete.
\end{proof}

Note in passing that iterating~\eqref{conv:iter} gives
\begin{equation*}
\label{conv:iter:2}
\|\hat{\bm{\psi}}-\bm{\psi}_{k+1}\|_{\bm{L}^2(\omega)}\le (1-\rho')^{\frac{k+1}{2}}\|\hat{\bm{\psi}}-\bm{\psi}_0\|_{\bm{L}^2(\omega)} \to 0,
\end{equation*}
as $k\to\infty$, being $(1-\rho') \in (0,1)$.

As a final remark, we observe that the iterative scheme~\eqref{iterate} is expected to converge very slowly.
This is due to the presence of the $h^4$ multiplicative term, which \emph{dampens} the convergence by making the norm $\|\bm{\psi}_{k+1}-\bm{\psi}_k\|_{\bm{L}^2(\omega)}$ small for all $k\ge 0$. 
The dampening is due to the fact that the $h^4$ term neglects the effects of the term $\kappa=h^{-q}$, $0<q<2$, appearing in the penalty term.
This means that the iterates will slowly depart from the initialisation $\bm{\psi}_0$ which is customarily chosen to be either $\bm{0}$ (viz. \cite{Scholz1984}) or the solution of the linearised version of the problem under consideration.

\section{Numerical Simulations}
\label{numerics}

In this last section of the paper, we implement numerical simulations aiming to test the convergence of the algorithms presented in section~\ref{approx:original} and in section~\ref{approx:penalty}.

Let $R>0$ be given. We consider as a domain a circle of radius $r_A:=\frac{R}{2}$
\begin{equation*}
\omega:=\left\{y=(y_\alpha)\in \mathbb{R}^2;\sqrt{y_1^2+y_2^2}<r_A\right\}.
\end{equation*}

The middle surface of the membrane shell under consideration is a non-hemispherical spherical cap which is not in contact with the plane $\{x_3=0\}$. The parametrization we choose is $\bm{\theta} \in \mathcal{C}^2(\overline{\omega};\mathbb{E}^3)$ defined by:
\begin{equation}
\label{middlesurf}
\bm{\theta}(y):=\left(y_1, y_2, \sqrt{R^2-y_1^2-y_2^2}-0.85\right),\quad\textup{ for all } y=(y_\alpha) \in \overline{\omega}.
\end{equation}

Throughout this section, the values of $\varepsilon$, $\lambda$, $\mu$ and $R$ are fixed as follows
\begin{equation*}
	\begin{aligned}
		\varepsilon&=0.001,\\
		\lambda&=0.4,\\
		\mu&=0.012,\\
		R&=1.0.
	\end{aligned}
\end{equation*}

The applied body force density $\bm{p}^\varepsilon=(p^{i,\varepsilon})$ entering the first two batches of experiments is given by $\bm{p}^\varepsilon=(0,0,g(y))$, where
$$
g(y):=
\begin{cases}
-\frac{2\varepsilon}{25}(-5.0 y_1^2-5.0 y_2^2+0.295), &\textup{ if } |y|< 0.060,\\
0, &\textup{otherwise}.
\end{cases}
$$

We let $\bm{q}=(0,0,1)$.
We observe that even though $g$ defined as above is not of class $H^1(\omega)$, the numerical results we obtained comply with the theoretical results obtained in Theorem~\ref{ex-un-kappa} and Theorem~\ref{th:conv}.

The expressions of the geometrical parameters (i.e., the covariant and contravariant bases, the first fundamental form in covariant and contravariant components, the second fundamental form in covariant and mixed components, etc.) associated with the middle surface~\eqref{middlesurf} were computed by means of the symbolic computer provided by MATLAB~\cite{DPSY2023}.
The numerical simulations are performed by means of the software FEniCS~\cite{Fenics2016} and the visualization is performed by means of the software ParaView~\cite{Ahrens2005}.
The plots were created by means of the \verb*|matplotlib| libraries from a Python~3.9.8 installation.

The first batch of numerical experiments is meant to validate the claim of Theorem~\ref{ex-un-kappa}. We fix the mesh size $0<h<<1$ and we let $\kappa=h^q$ in Problem~\ref{problem2}. Consider a sequence of exponents $\{q_\ell\}_{\ell=1}^\infty$ such that $q_\ell \to \infty$ as $\ell \to\infty$ and let $\bm{\zeta}^{\varepsilon,h}_{h^{q_n}}$ and $\bm{\zeta}^{\varepsilon,h}_{h^{q_m}}$ be the solutions of Problem~\ref{problem2} corresponding to $\kappa=h^{q_n}$ and $\kappa=h^{q_m}$ respectively.
The experiments whose results are shown in Figures~\ref{fig:1}--\ref{fig:4} an Tables~\ref{table:1}--\ref{table:4} below show that $\|\bm{\zeta}^{\varepsilon,h}_{h^{q_n}}-\bm{\zeta}^{\varepsilon,h}_{h^{q_m}}\|_{\bm{V}_M(\omega)} \to 0$ as $m,n \to\infty$. The algorithm stops when $\|\bm{\zeta}^{\varepsilon,h}_{h^{q_n}}-\bm{\zeta}^{\varepsilon,h}_{h^{q_m}}\|_{\bm{V}_M(\omega)}< 2.0 \times 10^{-6}$.

Each component $\zeta^{\varepsilon}_{\kappa,i}$ of Problem~\ref{problem2} is discretized by Lagrange triangles (cf., e.g., \cite{PGCFEM}) and homogeneous Dirichlet boundary conditions are imposed for all the components. The reason why the transverse component $\zeta^{\varepsilon}_{\kappa,3}$ was imposed to be subjected to this boundary condition is that Problem~\ref{problem1} is derived as a result of a rigorous asymptotic analysis starting from Koiter's model~\cite{CiaPie2018bCR,CiaPie2018b}. The fact that the transverse component of the solution of Koiter's model is of class $H^2_0(\omega)$ makes a boundary layer appear (viz. Section~7.3 of~\cite{Ciarlet2000}) and justifies our choice for this boundary condition, without which the boundary would be pushed down to the obstacle when, clearly, this is not the case. The higher regularity of the solution of Problem~\ref{problem2} (viz.~\eqref{conclusion-7}) and the higher regularity of the solution of Koiter's model for elliptic membranes subject to an obstacle, which can be derived by adapting the argument of Theorem~\ref{aug:int} and Theorem~\ref{aug:bdry} to the proof in~\cite{Iosifescu1994} justify the choice for the boundary condition of the transverse component.
At each iteration, Problem~\ref{problem3} is solved by Newton's method.

\begin{table}[H]
	\begin{varwidth}[b]{0.6\linewidth}
		\centering
		\begin{tabular}{ c c c l }
			\toprule
			Iteration& $q_n$ &$q_m$& Error \\
			\midrule
			1&0.5 &1.0& 0.0009870505482918299\\
			2&1.0 &1.5& 0.0005716399376703707\\
			3&1.5 &2.0& 0.0003259806690885746\\
			4&2.0 &2.5& 0.0001851239908727575\\
			5&2.5 &3.0& 0.00010447749338622102\\
			6&3.0 &3.5& 5.8930703167247946e-05\\
			7&3.5 &4.0& 3.2967335748701205e-05\\
			8&4.0 &4.5& 1.928599264387323e-05\\
			9&4.5 &5.0& 1.0777591612766316e-05\\
			10&5.0 &5.5& 5.9221185866507025e-06\\
			11&5.5 &6.0& 3.545734351957462e-06\\
			12&6.0 &6.5& 2.595888616762957e-06\\
			13&6.5 &7.0& 2.1107364521837126e-06\\
			14&7.0 &7.5& 1.958867087445544e-06\\
			\bottomrule
		\end{tabular}
		\caption{Verification of Theorem~\ref{ex-un-kappa} for $h=0.03123779990753546$ fixed and $q$ varying}
		\label{table:1}
	\end{varwidth}%
\hspace{0.1cm}
	\begin{minipage}[b]{0.4\linewidth}
		\centering
		\includegraphics[width=\textwidth]{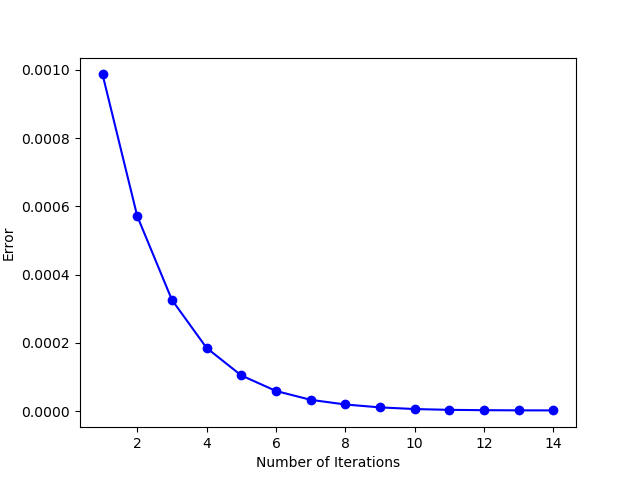}
		\captionof{figure}{The residual $\|\bm{\zeta}^{\varepsilon,h}_{h^{q_n}}-\bm{\zeta}^{\varepsilon,h}_{h^{q_m}}\|_{\bm{V}_M(\omega)}$ becomes lower than the tolerance after fourteen iterations.}
		\label{fig:1}
	\end{minipage}
	\hspace{0.1cm}
	\begin{minipage}[b]{0.4\linewidth}
		\centering
		\includegraphics[width=\textwidth]{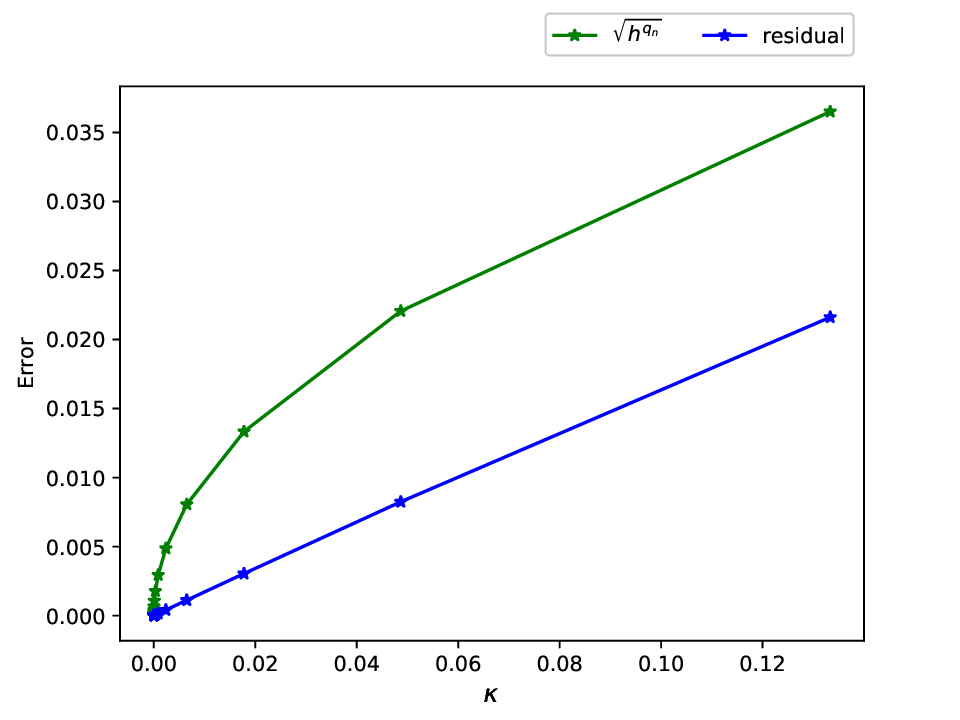}
		\captionof{figure}{Comparison between the residual $\|\bm{\zeta}^{\varepsilon,h}_{h^{q_n}}-\bm{\zeta}^{\varepsilon,h}_{h^{q_m}}\|_{\bm{V}_M(\omega)}$, in blue, and the function $\frac{\sqrt{h^{q_n}}}{10}$, in green.}
		\label{fig:1b}
	\end{minipage}
\end{table}

\begin{table}[H]
	\begin{varwidth}[b]{0.6\linewidth}
		\centering
		\begin{tabular}{ c c c l }
			\toprule
			Iteration&$q_n$ &$q_m$& Error \\
			\midrule
			1&0.5 &1.0& 0.0008589020335743345\\
			2&1.0 &1.5& 0.00024578598359837673\\
			3&1.5 &2.0& 6.925018104565528e-05\\
			4&2.0 &2.5& 1.943169742921457e-05\\
			5&2.5 &3.0& 5.4843823503599594e-06\\
			6&3.0 &3.5& 1.5246502664061824e-06\\
			\bottomrule
		\end{tabular}
		\caption{Verification of Theorem~\ref{ex-un-kappa} for $h=0.007812398571396802$ fixed and $q$ varying}
		\label{table:3}
	\end{varwidth}%
	\hspace{0.1cm}
	\begin{minipage}[b]{0.4\linewidth}
		\centering
		\includegraphics[width=\textwidth]{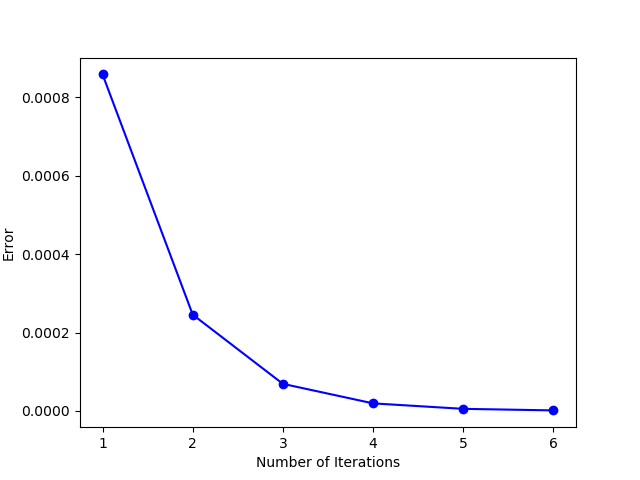}
		\captionof{figure}{The residual $\|\bm{\zeta}^{\varepsilon,h}_{h^{q_n}}-\bm{\zeta}^{\varepsilon,h}_{h^{q_m}}\|_{\bm{V}_M(\omega)}$ becomes lower than the tolerance after six iterations.}
		\label{fig:3}
	\end{minipage}
\hspace{0.1cm}
\begin{minipage}[b]{0.4\linewidth}
	\centering
	\includegraphics[width=\textwidth]{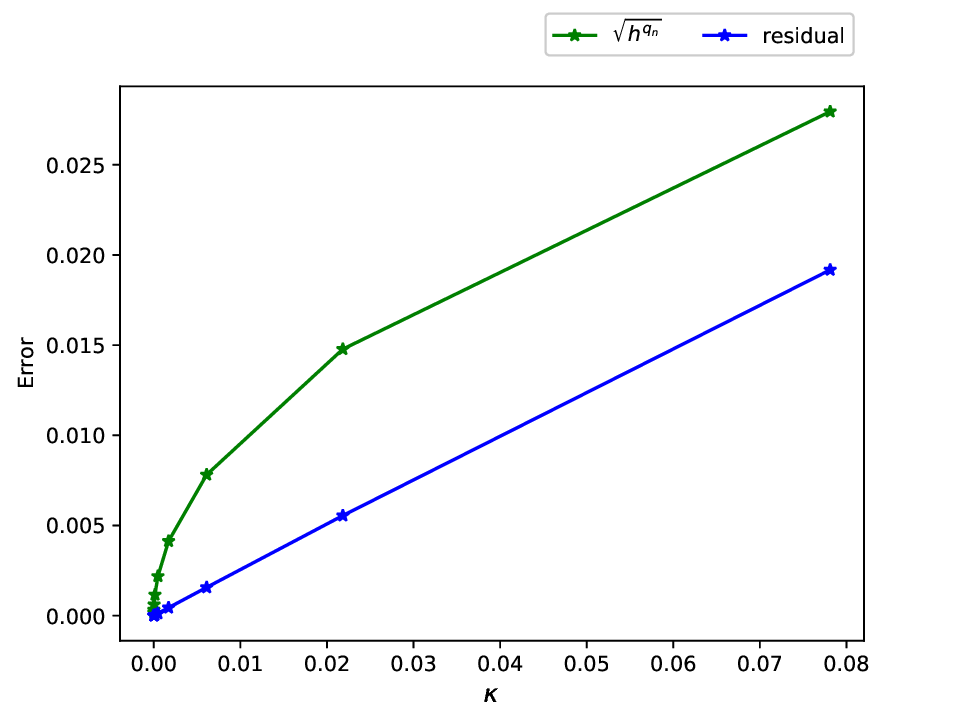}
	\captionof{figure}{Comparison between the residual $\|\bm{\zeta}^{\varepsilon,h}_{h^{q_n}}-\bm{\zeta}^{\varepsilon,h}_{h^{q_m}}\|_{\bm{V}_M(\omega)}$, in blue, and the function $\frac{\sqrt{h^{q_n}}}{10}$, in green.}
	\label{fig:3b}
\end{minipage}
\end{table}

\begin{table}[H]
	\begin{varwidth}[b]{0.6\linewidth}
		\centering
		\begin{tabular}{ c c c l }
			\toprule
			Iteration&$q_n$ &$q_m$& Error \\
			\midrule
			1&0.5 &1.0& 0.0006853937021067343\\
			2&1.0 &1.5& 0.0001389653237533636\\
			3&1.5 &2.0& 2.8767966192506784e-05\\
			4&2.0 &2.5& 6.876855587433019e-06\\
			5&2.5 &3.0& 1.1588084240614098e-06\\
			\bottomrule
		\end{tabular}
		\caption{Verification of Theorem~\ref{ex-un-kappa} for $h=0.0039062328553237536$ fixed and $q$ varying}
		\label{table:4}
	\end{varwidth}%
	\hspace{0.1cm}
	\begin{minipage}[b]{0.4\linewidth}
		\centering
		\includegraphics[width=\textwidth]{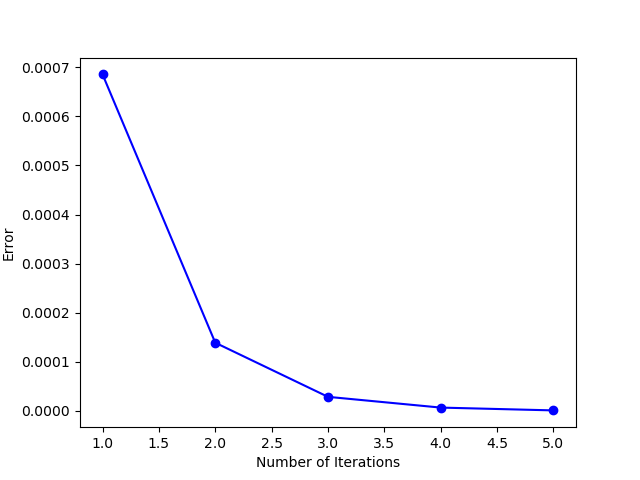}
		\captionof{figure}{The residual $\|\bm{\zeta}^{\varepsilon,h}_{h^{q_n}}-\bm{\zeta}^{\varepsilon,h}_{h^{q_m}}\|_{\bm{V}_M(\omega)}$ becomes lower than the tolerance after five iterations.}
		\label{fig:4}
	\end{minipage}
\hspace{0.1cm}
\begin{minipage}[b]{0.4\linewidth}
	\centering
	\includegraphics[width=\textwidth]{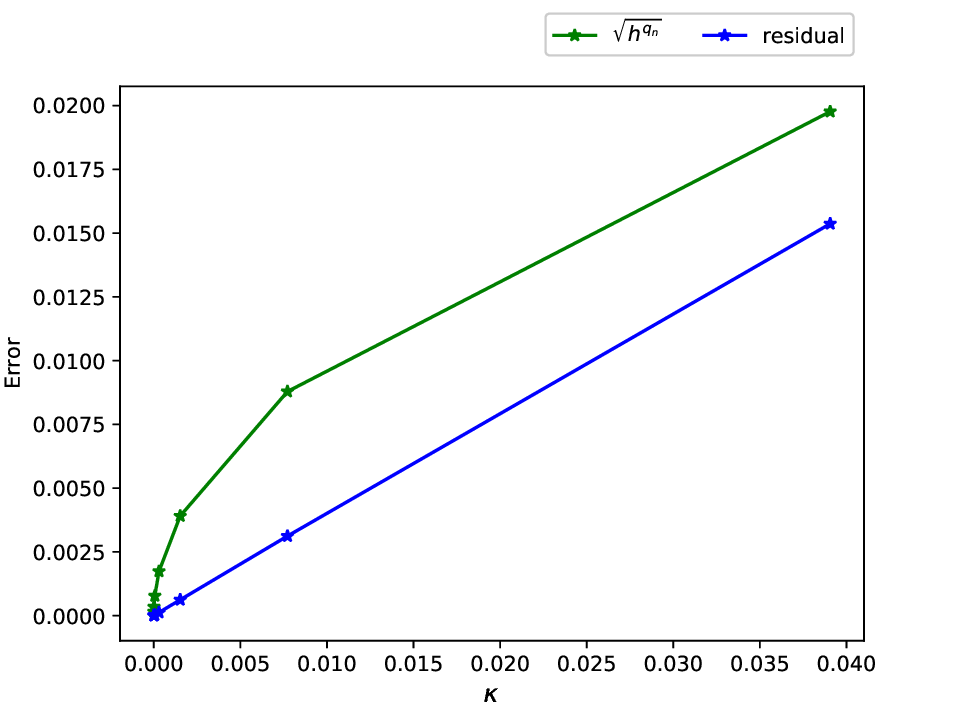}
	\captionof{figure}{Comparison between the residual $\|\bm{\zeta}^{\varepsilon,h}_{h^{q_n}}-\bm{\zeta}^{\varepsilon,h}_{h^{q_m}}\|_{\bm{V}_M(\omega)}$, in blue, and the function $\frac{\sqrt{h^{q_n}}}{10}$, in green.}
	\label{fig:4b}
\end{minipage}
\end{table}

From the data patterns in Figures~\ref{fig:1}--\ref{fig:4b} we observe that as $h$ decreases (and so $\kappa$ increases) less iterations are needed to reach the tolerance triggering the stopping criterion. This is coherent with the conclusion of Theorem~\ref{ex-un-kappa}.

The second batch of numerical experiments is meant to validate the claim of Theorem~\ref{th:conv}. 
We show that, for a fixed $0< q <2$, the error $\|\bm{\zeta}^{\varepsilon,h_1}_{h_1^{q}}-\bm{\zeta}^{\varepsilon,h_2}_{h_2^{q}}\|_{\bm{V}_M(\omega)}$ tends to zero as $h_1, h_2 \to 0^+$.
The results of these experiments are reported in Figure~\ref{fig:5} below.

\begin{figure}[H]
	\centering
	\begin{subfigure}[b]{0.3\linewidth}
		\includegraphics[width=1.0\linewidth]{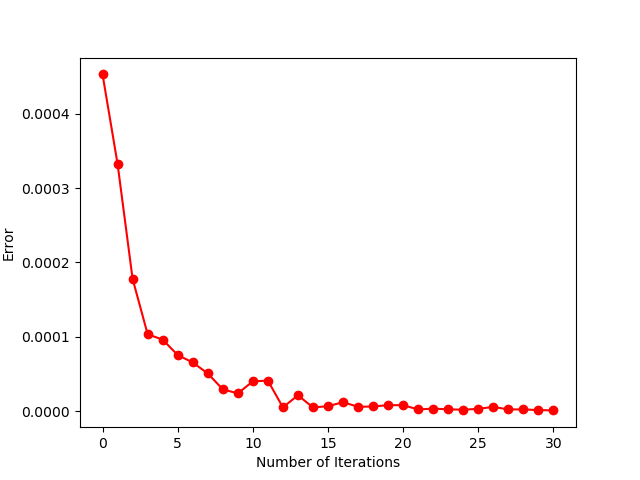}
		\subcaption{For $q=0.4$ the stopping criterion of the Cauchy sequence is reached when $h=0.00268$}
	\end{subfigure}%
	\hspace{0.5cm}
	\begin{subfigure}[b]{0.3\linewidth}
		\includegraphics[width=1.0\linewidth]{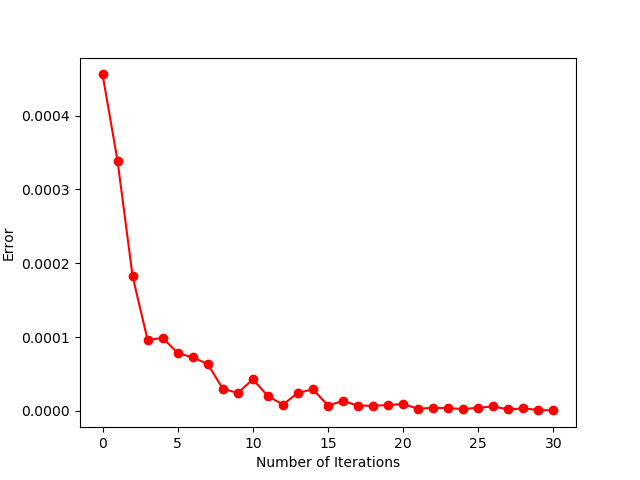}
		\subcaption{For $q=0.5$ the stopping criterion of the Cauchy sequence is reached when $h=0.00268$}
	\end{subfigure}%
	\hspace{0.5cm}
	\begin{subfigure}[b]{0.3\linewidth}
		\includegraphics[width=1.0\linewidth]{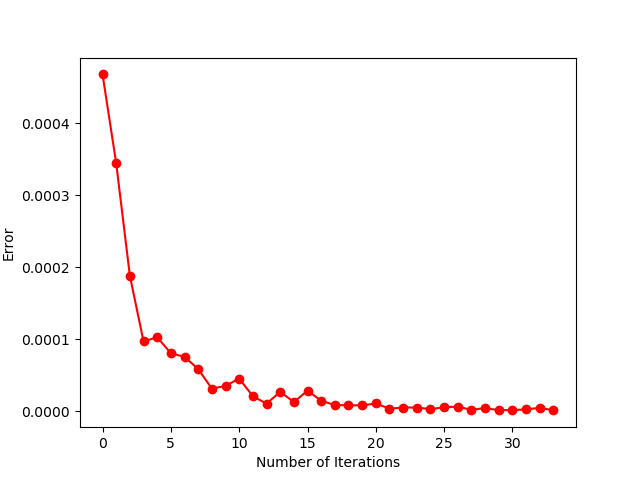}
		\subcaption{For $q=0.6$ the stopping criterion of the Cauchy sequence is reached when $h=0.00249$}
	\end{subfigure}%
\end{figure}

\begin{figure}[H]
	\ContinuedFloat
	\centering
	\begin{subfigure}[b]{0.3\linewidth}
		\includegraphics[width=1.0\linewidth]{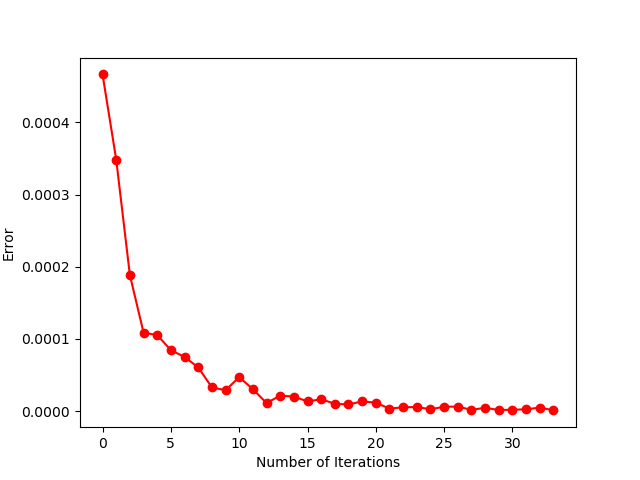}
		\subcaption{For $q=0.7$ the stopping criterion of the Cauchy sequence is reached when $h=0.00249$}
	\end{subfigure}%
	\hspace{0.5cm}
	\begin{subfigure}[b]{0.3\linewidth}
		\includegraphics[width=1.0\linewidth]{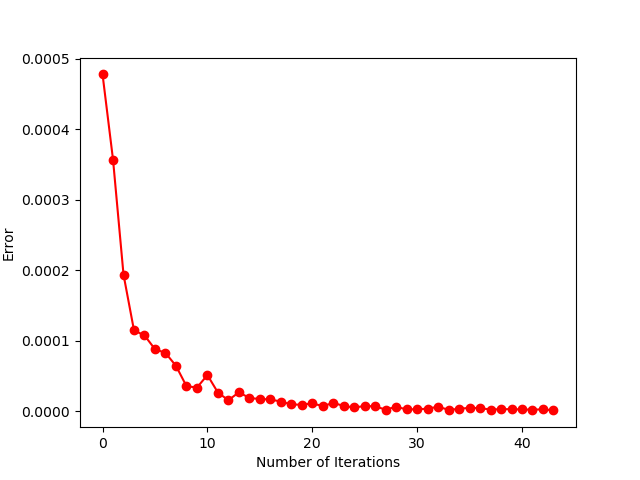}
		\subcaption{For $q=1.0$ the stopping criterion of the Cauchy sequence is reached when $h=0.00199$}
	\end{subfigure}%
	\hspace{0.5cm}
	\begin{subfigure}[b]{0.3\linewidth}
		\includegraphics[width=1.0\linewidth]{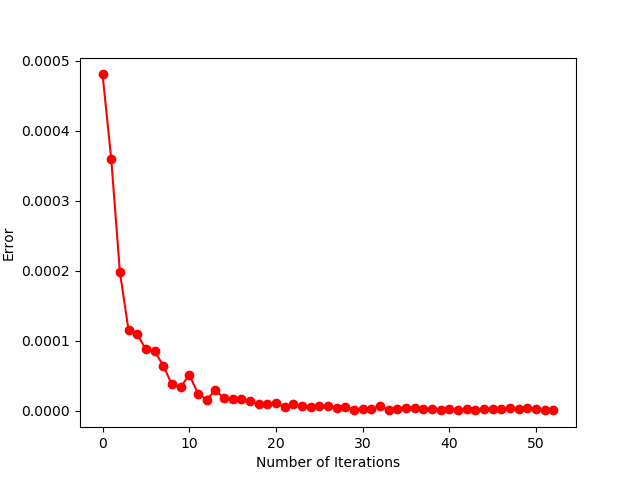}
		\subcaption{For $q=1.3$ the stopping criterion of the Cauchy sequence is reached when $h=0.00169$}
	\end{subfigure}%
	\caption{Given $0<q<2$, the error $\|\bm{\zeta}^{\varepsilon,h_1}_{h_1^{q}}-\bm{\zeta}^{\varepsilon,h_2}_{h_2^{q}}\|_{\bm{V}_M(\omega)}$ converges to zero as $h_1, h_2 \to0^+$. The value of $h$ for which the algorithm stops decreases as $q$ increases.}
	\label{fig:5}
\end{figure}

The third batch of numerical experiments validates the genuineness of the model.
We observe that the presented data exhibits the pattern that, for a fixed $0<h<<1$ and a fixed $0< q <2$, the contact area increases as the applied body force intensity increases.
For the third batch of experiments, the applied body force density $\bm{p}^\varepsilon=(p^{i,\varepsilon})$ entering the model is given by $\bm{p}^\varepsilon=(0,0,g_\ell(y))$, where $\ell$ is a nonnegative integer and
$$
g_\ell(y):=
\begin{cases}
-\frac{2\varepsilon}{25}(-5.0 y_1^2-5.0 y_2^2+(1+0.05 \ell)\times 0.295), &\textup{ if } |y|< 0.060,\\
0, &\textup{otherwise}.
\end{cases}
$$

We observe that even though $g_\ell$ defined as above is not of class $H^1(\omega)$, the numerical results we obstained comply with the Physics, in the sense that the contact area increases as the intensity of the applied body force increases.
The results of these experiments are reported in Figure~\ref{fig:6} below.

\newpage

\begin{figure}[H]
	\centering
	\begin{subfigure}[b]{0.4\linewidth}
		\includegraphics[width=\textwidth]{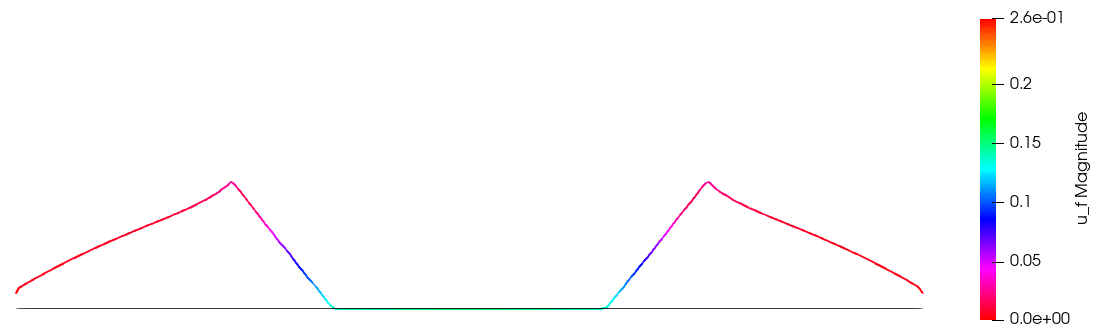}
		\subcaption{$\ell=0$}
	\end{subfigure}%
	\hspace{0.5cm}
	\begin{subfigure}[b]{0.4\linewidth}
		\includegraphics[width=1.0\textwidth]{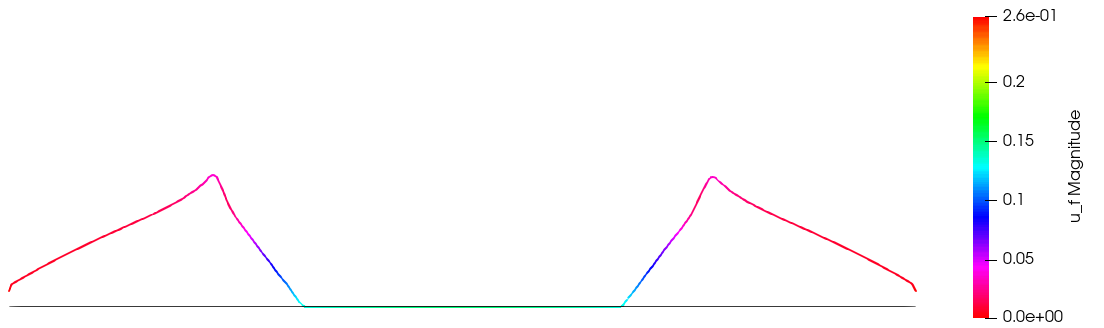}
		\subcaption{$\ell=3$}
	\end{subfigure}%
\end{figure}

\begin{figure}[H]
	\ContinuedFloat
	\centering
	\begin{subfigure}[b]{0.4\linewidth}
		\includegraphics[width=\textwidth]{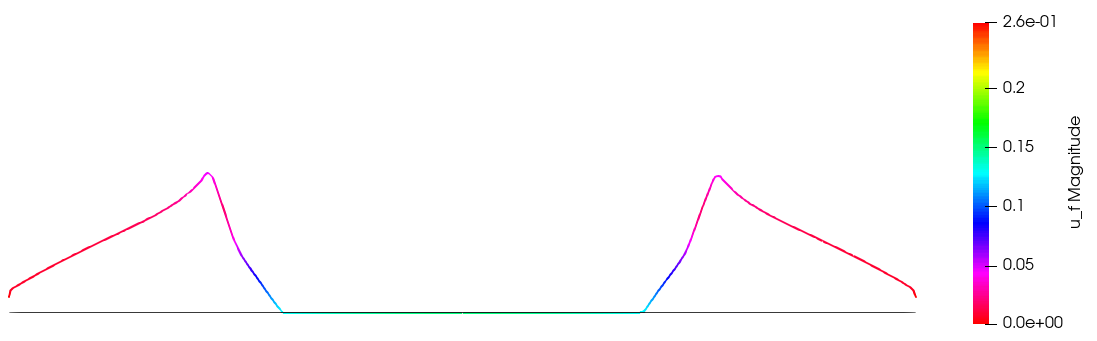}
		\subcaption{$\ell=6$}
	\end{subfigure}%
	\hspace{0.5cm}
	\begin{subfigure}[b]{0.4\linewidth}
		\includegraphics[width=\textwidth]{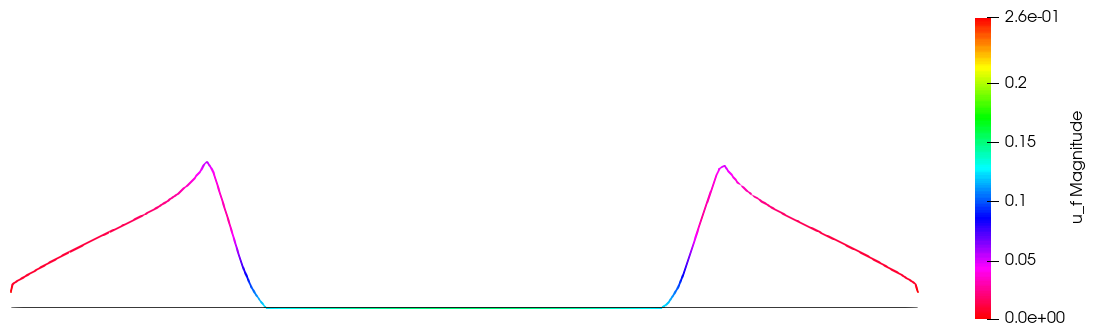}
		\subcaption{$\ell=9$}
	\end{subfigure}%
	\caption{Cross sections of a deformed membrane shell subjected not to cross a given planar obstacle.
		Given $0<h<<1$ and $0<q<2$ we observe that as the applied body force magnitude increases the contact area increases.}
	\label{fig:6}
\end{figure}

\section*{Conclusions and Commentary}

In this paper we established the convergence of a numerical scheme based on the Finite Element Method for approximating the solution of a set of variational inequalities modelling the deformation of a linearly elastic elliptic membrane shell subject to remaining confined in a prescribed half space. 

Instead of directly approximating the solution of the variational inequalities, we approximate the solution of the corresponding penalized variational formulation with respect to the norm of the space where the solution of this penalized problem is sought. Moreover, we also show that the iterative method proposed by Brezis and Sibony can be applied to approximate the solution of the discrete penalized problem under consideration with respect, however, to a weaker norm.

The main novelty introduced in this paper is the overcoming of the condition $(\ast)$ introduced by Scholz~\cite{Scholz1984}.
Indeed, since the second order differential operator we are considering takes into account all the components of the solution, which is a vector field with values in the Euclidean space $\mathbb{E}^3$, it is not straightforward to re-write the condition $(\ast)$ introduced by Scholz~\cite{Scholz1984} in a vectorial context. We instead assume that the middle surface of the linearly elastic shell under consideration satisfies a certain geometrical assumption, which is the same assumption ensuring the validity of the ``density property'' introduced in~\cite{CiaMarPie2018b,CiaMarPie2018}.

The method we presented in this paper is, however, in general not applicable to fourth order obstacle problems like the one studied by L\'eger \& Miara~\cite{Leger2008,Leger2010}, and for which a suitable numerical scheme was studied in~\cite{PS}. The reason why the methodology presented in this paper is not applicable to fourth order problems is due to the fact that the solution of fourth order obstacle problems is not in general of class $H^4$ over its definition domain. This limitation was established by Caffarelli and his associates in the papers~\cite{Caffarelli1979,CafFriTor1982}.

In order to study the convergence of the finite element analysis addressed in the paper~\cite{PS}, an interior $\mathcal{C}^0$ penalty method based on a nonconforming finite element of Morley type had to be exploited. The choice of the nonconforming finite element of Morley type is motivated by the fact that the highest regularity one can achieve for the considered problem is $H^3$ over the definition domain. One such regularity is sufficient to apply a suitable Green's formula for establishing the convergence of the finite element scheme in~\cite{PS}.

We also observe that the penalty method discussed in this paper is, in the context of a finite element analysis, more easily applicable than the primal-dual active set method~\cite{SunYuan2006}. The latter is particularly amenable in the context of the optimization of problems the solution of which is a real-valued functions or a vector field for which the constraint bears on the transverse component~\cite{PWDT2D, PWDT3D}.

\section*{Declarations}

\noindent\textbf{Authors’ Contribution}. All authors have contributed to the realisation of this manuscript in equal manner.

\vspace{.5cm}

\noindent\textbf{Acknowledgements}. Not applicable

\vspace{.5cm}

\noindent\textbf{Ethical Approval}. Not applicable.

\vspace{.5cm}

\noindent\textbf{Availability of Supporting Data}. Not applicable.

\vspace{.5cm}

\noindent\textbf{Competing Interests}. All authors certify that they have no affiliations with or involvement in any organi- zation or entity with any competing interests in the subject matter or materials discussed in this manuscript.

\vspace{.5cm}

\noindent\textbf{Funding}. P.P. and A.M. were partly supported by the Research Fund of Indiana University and by the National Science Foundation under Grant Number DMS-2051032.

\bibliographystyle{abbrvnat} 
\bibliography{references}	

\begin{thebibliography}{59}
\providecommand{\natexlab}[1]{#1}
\providecommand{\url}[1]{\texttt{#1}}
\expandafter\ifx\csname urlstyle\endcsname\relax
  \providecommand{\doi}[1]{doi: #1}\else
  \providecommand{\doi}{doi: \begingroup \urlstyle{rm}\Url}\fi

\bibitem[Agmon et~al.(1959)Agmon, Douglis, and Nirenberg]{AgmDouNir1959}
S.~Agmon, A.~Douglis, and L.~Nirenberg.
\newblock Estimates near the boundary for solutions of elliptic partial
  differential equations satisfying general boundary conditions. {I}.
\newblock \emph{Comm. Pure Appl. Math.}, 12:\penalty0 623--727, 1959.

\bibitem[Agmon et~al.(1964)Agmon, Douglis, and Nirenberg]{AgmDouNir1964}
S.~Agmon, A.~Douglis, and L.~Nirenberg.
\newblock Estimates near the boundary for solutions of elliptic partial
  differential equations satisfying general boundary conditions. {II}.
\newblock \emph{Comm. Pure Appl. Math.}, 17:\penalty0 35--92, 1964.

\bibitem[Ahrens et~al.(2005)Ahrens, Geveci, and Law]{Ahrens2005}
J.~Ahrens, B.~Geveci, and C.~Law.
\newblock \emph{ParaView: {A}n {E}nd-{U}ser {T}ool for {L}arge {D}ata
  {V}isualization}.
\newblock Visualization Handbook, Elsevier, 2005.
\newblock ISBN-13: 978-0123875822.

\bibitem[Alexandrescu(1994)]{Iosifescu1994}
O.~Alexandrescu.
\newblock Th\'{e}or\`eme d'existence pour le mod\`ele bidimensionnel de coque
  non lin\'{e}aire de {W}. {T}. {K}oiter.
\newblock \emph{C. R. Acad. Sci. Paris S\'{e}r. I Math.}, 319\penalty0
  (8):\penalty0 899--902, 1994.

\bibitem[Brenner and Scott(2008)]{Brenner2008}
S.~Brenner and L.~R. Scott.
\newblock \emph{The mathematical theory of finite element methods}.
\newblock Springer, New York, third edition, 2008.

\bibitem[Brezis(2011)]{Brez11}
H.~Brezis.
\newblock \emph{Functional {A}nalysis, {S}obolev {S}paces and {P}artial
  {D}ifferential {E}quations}.
\newblock Springer, New York, 2011.

\bibitem[Brezis and Sibony(1967/1968)]{BrezisSibony1968}
H.~Brezis and M.~Sibony.
\newblock M\'{e}thodes d'approximation et d'it\'{e}ration pour les
  op\'{e}rateurs monotones.
\newblock \emph{Arch. Rational Mech. Anal.}, 28:\penalty0 59--82, 1967/1968.

\bibitem[Brezis and Stampacchia(1968)]{BrezisStampacchia1968}
H.~Brezis and G.~Stampacchia.
\newblock Sur la r\'{e}gularit\'{e} de la solution d'in\'{e}quations
  elliptiques.
\newblock \emph{Bull. Soc. Math. France}, 96:\penalty0 153--180, 1968.

\bibitem[Caffarelli and Friedman(1979)]{Caffarelli1979}
L.~A. Caffarelli and A.~Friedman.
\newblock The obstacle problem for the biharmonic operator.
\newblock \emph{{A}nn. {S}cuola {N}orm. {S}up. {P}isa {C}l. {S}ci. (4)},
  6:\penalty0 151--184, 1979.

\bibitem[Caffarelli et~al.(1982)Caffarelli, Friedman, and
  Torelli]{CafFriTor1982}
L.~A. Caffarelli, A.~Friedman, and A.~Torelli.
\newblock The two-obstacle problem for the biharmonic operator.
\newblock \emph{{P}acific {J}. {M}ath.}, 103:\penalty0 325--335, 1982.

\bibitem[Chapelle and Bathe(2011)]{ChaBat2011}
D.~Chapelle and K.-J. Bathe.
\newblock \emph{The finite element analysis of shells - fundamentals}.
\newblock Springer-Verlag Berlin Heidelberg, {S}econd edition, 2011.

\bibitem[Chen et~al.(2003)Chen, Glowinski, and Li]{CheGloLi2003}
Z.~Chen, R.~Glowinski, and K.~Li.
\newblock \emph{Current trends in scientific computing: {ICM} 2002 {B}eijing
  {S}atellite {C}onference on {S}cientific {C}omputing, {A}ugust 15-18, 2002,
  {X}i'an {J}iaotong {U}niversity, {X}i'an, {C}hina}.
\newblock American Mathematical Society, Providence, R.I., 2003.

\bibitem[Ciarlet(1978)]{PGCFEM}
P.~G. Ciarlet.
\newblock \emph{The Finite Element Method for Elliptic Problems}.
\newblock North-Holland, Amsterdam, 1978.

\bibitem[Ciarlet(1988)]{Ciarlet1988}
P.~G. Ciarlet.
\newblock \emph{Mathematical Elasticity. Vol. I: Three-Dimensional Elasticity}.
\newblock North-Holland, Amsterdam, 1988.

\bibitem[Ciarlet(2000)]{Ciarlet2000}
P.~G. Ciarlet.
\newblock \emph{Mathematical Elasticity. Vol. III: {T}heory of Shells.}
\newblock North-Holland, Amsterdam, 2000.

\bibitem[Ciarlet(2005)]{Ciarlet2005}
P.~G. Ciarlet.
\newblock \emph{An {I}ntroduction to {D}ifferential {G}eometry with
  {A}pplications to {E}lasticity}.
\newblock Springer, Dordrecht, 2005.

\bibitem[Ciarlet(2013)]{PGCLNFAA}
P.~G. Ciarlet.
\newblock \emph{Linear and Nonlinear Functional Analysis with Applications}.
\newblock Society for Industrial and Applied Mathematics, Philadelphia, 2013.

\bibitem[Ciarlet and Destuynder(1979)]{CiaDes1979}
P.~G. Ciarlet and P.~Destuynder.
\newblock A justification of the two-dimensional linear plate model.
\newblock \emph{{J}. {M}\'ecanique}, 18:\penalty0 315--344, 1979.

\bibitem[Ciarlet and Lods(1996{\natexlab{a}})]{CiaLods1996a}
P.~G. Ciarlet and V.~Lods.
\newblock On the ellipticity of linear membrane shell equations.
\newblock \emph{{J}. {M}ath. {P}ures {A}ppl.}, 75:\penalty0 107--124,
  1996{\natexlab{a}}.

\bibitem[Ciarlet and Lods(1996{\natexlab{b}})]{Ciarlet1996}
P.~G. Ciarlet and V.~Lods.
\newblock Asymptotic analysis of linearly elastic shells. {I}. {J}ustification
  of membrane shell equations.
\newblock \emph{Arch. {R}ational {M}ech. {A}nal.}, 136\penalty0 (2):\penalty0
  119--161, 1996{\natexlab{b}}.

\bibitem[Ciarlet and Piersanti(2019{\natexlab{a}})]{CiaPie2018b}
P.~G. Ciarlet and P.~Piersanti.
\newblock Obstacle problems for {K}oiter's shells.
\newblock \emph{{M}ath. {M}ech. {S}olids}, 24:\penalty0 3061--3079,
  2019{\natexlab{a}}.

\bibitem[Ciarlet and Piersanti(2019{\natexlab{b}})]{CiaPie2018bCR}
P.~G. Ciarlet and P.~Piersanti.
\newblock A confinement problem for a linearly elastic {K}oiter's shell.
\newblock \emph{C.R. Acad. Sci. Paris, S\'{e}r. I}, 357:\penalty0 221--230,
  2019{\natexlab{b}}.

\bibitem[Ciarlet and Sanchez-Palencia(1996)]{CiaSanPan1996}
P.~G. Ciarlet and E.~Sanchez-Palencia.
\newblock An existence and uniqueness theorem for the two-dimensional linear
  membrane shell equations.
\newblock \emph{{J}. {M}ath. {P}ures {A}ppl.}, 75:\penalty0 51--67, 1996.

\bibitem[Ciarlet et~al.(2018)Ciarlet, Mardare, and Piersanti]{CiaMarPie2018b}
P.~G. Ciarlet, C.~Mardare, and P.~Piersanti.
\newblock Un probl\`eme de confinement pour une coque membranaire
  lin\'eairement \'elastique de type elliptique.
\newblock \emph{{C}. {R}. {M}ath. {A}cad. {S}ci. {P}aris}, 356\penalty0
  (10):\penalty0 1040--1051, 2018.

\bibitem[Ciarlet et~al.(2019)Ciarlet, Mardare, and Piersanti]{CiaMarPie2018}
P.~G. Ciarlet, C.~Mardare, and P.~Piersanti.
\newblock An obstacle problem for elliptic membrane shells.
\newblock \emph{{M}ath. {M}ech. {S}olids}, 24\penalty0 (5):\penalty0
  1503--1529, 2019.

\bibitem[Duan et~al.(To appear)Duan, Piersanti, Shen, and Yang]{DPSY2023}
W.~Duan, P.~Piersanti, X.~Shen, and Q.~Yang.
\newblock Numerical corroboration of koiter's model for all the main types of
  linearly elastic shells in the static case.
\newblock \emph{Math. Mech. Solids}, To appear.

\bibitem[Eggleston(1958)]{Eggleston1958}
H.~G. Eggleston.
\newblock \emph{Convexity}.
\newblock Cambridge Tracts in Mathematics and Mathematical Physics, No. 47.
  Cambridge University Press, New York, 1958.

\bibitem[Evans(2010)]{Evans2010}
L.~C. Evans.
\newblock \emph{Partial {D}ifferential {E}quations}.
\newblock American Mathematical Society, Providence, {S}econd edition, 2010.

\bibitem[Falk(1974)]{Falk1974}
R.~S. Falk.
\newblock Error estimates for the approximation of a class of variational
  inequalities.
\newblock \emph{{M}ath. {C}omp.}, 28:\penalty0 963--971, 1974.

\bibitem[Frehse(1971)]{Frehse1971}
J.~Frehse.
\newblock {Z}um {D}ifferenzierbarkeitsproblem bei {V}ariationsungleichungen
  höherer {O}rdnung. ({G}erman).
\newblock \emph{{A}bh. {M}ath. {S}em. {U}niv. {H}amburg}, 36:\penalty0
  140--149, 1971.

\bibitem[Frehse(1973)]{Frehse1973}
J.~Frehse.
\newblock On the regularity of the solution of the biharmonic variational
  inequality.
\newblock \emph{{M}anuscripta {M}ath.}, 9:\penalty0 91--103, 1973.

\bibitem[Ganesan and Tobiska(2017)]{Ganesan2017}
S.~Ganesan and L.~Tobiska.
\newblock \emph{Finite Elements: Theory and Algorithms}.
\newblock Cambridge University Press, 2017.

\bibitem[Genevey(1996)]{Genevey1996}
K.~Genevey.
\newblock A regularity result for a linear membrane shell problem.
\newblock \emph{Math. Modelling Numer.}, 30:\penalty0 467--488, 1996.

\bibitem[Geymonat(1966)]{Geymonat1965}
G.~Geymonat.
\newblock Sui problemi ai limiti per i sistemi lineari ellittici.
\newblock In \emph{Atti del {C}onvegno su le {E}quazioni alle {D}erivate
  {P}arziali ({N}ervi, 1965)}, pages 60--65. Edizioni Cremonese, Rome, 1966.

\bibitem[Grisvard(2011)]{Grisvard2011}
P.~Grisvard.
\newblock \emph{Elliptic problems in nonsmooth domains}, volume~69 of
  \emph{Classics in Applied Mathematics}.
\newblock Society for Industrial and Applied Mathematics (SIAM), Philadelphia,
  PA, 2011.
\newblock Reprint of the 1985 original [ MR0775683], With a foreword by Susanne
  C. Brenner.

\bibitem[H\"{o}rmander(1990)]{Hormander1990}
L.~H\"{o}rmander.
\newblock \emph{The analysis of linear partial differential operators. {I}},
  volume 256 of \emph{Grundlehren der mathematischen Wissenschaften
  [Fundamental Principles of Mathematical Sciences]}.
\newblock Springer-Verlag, Berlin, second edition, 1990.
\newblock Distribution theory and Fourier analysis.

\bibitem[Langtangen and Logg(2016)]{Fenics2016}
H.~P. Langtangen and A.~Logg.
\newblock \emph{Solving {PDE}s in {P}ython}, volume~3 of \emph{Simula
  SpringerBriefs on Computing}.
\newblock Springer, Cham, 2016.
\newblock The FEniCS tutorial I.

\bibitem[L\'eger and Miara(2008)]{Leger2008}
A.~L\'eger and B.~Miara.
\newblock Mathematical justification of the obstacle problem in the case of a
  shallow shell.
\newblock \emph{J. {E}lasticity}, 90:\penalty0 241--257, 2008.

\bibitem[L\'eger and Miara(2010)]{Leger2010}
A.~L\'eger and B.~Miara.
\newblock Erratum to: {M}athematical justification of the obstacle problem in
  the case of a shallow shell.
\newblock \emph{J. {E}lasticity}, 98:\penalty0 115--116, 2010.

\bibitem[L\'eger and Miara(2018)]{LegMia2018}
A.~L\'eger and B.~Miara.
\newblock A linearly elastic shell over an obstacle: The flexural case.
\newblock \emph{{J}. {E}lasticity}, 131:\penalty0 19--38, 2018.

\bibitem[Li et~al.(2015)Li, Huang, and Huang]{LiHuaAHuaQ2015}
K.~Li, A.~Huang, and Q.~Huang.
\newblock \emph{Finite element method and its applications}.
\newblock Beijing, China : Science Press, 2015.

\bibitem[Lions(1969)]{Lions1969}
J.-L. Lions.
\newblock \emph{Quelques m\'{e}thodes de r\'{e}solution des probl\`emes aux
  limites non lin\'{e}aires}.
\newblock Dunod; Gauthier-Villars, Paris, 1969.

\bibitem[Mezabia et~al.(2022)Mezabia, Chacha, and Bensayah]{MezChaBen2020}
M.~E. Mezabia, D.~A. Chacha, and A.~Bensayah.
\newblock Modelling of frictionless {S}ignorini problem for a linear elastic
  membrane shell.
\newblock \emph{Applicable Analysis}, 101\penalty0 (6):\penalty0 2295--2315,
  2022.

\bibitem[Ne\v{c}as(2012)]{Nec67}
J.~Ne\v{c}as.
\newblock \emph{Direct {M}ethods in the {T}heory of {E}lliptic {E}quations}.
\newblock Springer, Heidelberg, 2012.

\bibitem[Piersanti(2022{\natexlab{a}})]{Pie-2022-interior}
P.~Piersanti.
\newblock On the improved interior regularity of the solution of a second order
  elliptic boundary value problem modelling the displacement of a linearly
  elastic elliptic membrane shell subject to an obstacle.
\newblock \emph{Discrete Contin. Dyn. Syst.}, 42\penalty0 (2):\penalty0
  1011--1037, 2022{\natexlab{a}}.

\bibitem[Piersanti(2022{\natexlab{b}})]{Pie-2022-jde}
P.~Piersanti.
\newblock Asymptotic analysis of linearly elastic elliptic membrane shells
  subjected to an obstacle.
\newblock \emph{Journal of Differential Equations}, 320:\penalty0 114--142,
  2022{\natexlab{b}}.

\bibitem[Piersanti(2022{\natexlab{c}})]{Pie2020-1}
P.~Piersanti.
\newblock On the improved interior regularity of the solution of a fourth order
  elliptic problem modelling the displacement of a linearly elastic shallow
  shell subject to an obstacle.
\newblock \emph{{A}symptot. {A}nal.}, 127\penalty0 (1--2):\penalty0 35--55,
  2022{\natexlab{c}}.

\bibitem[Piersanti(2023)]{Pie2023}
P.~Piersanti.
\newblock Asymptotic analysis of linearly elastic flexural shells subjected to
  an obstacle in absence of friction.
\newblock \emph{J. Nonlinear Sci.}, 33\penalty0 (4):\penalty0 Paper No. 58, 39,
  2023.

\bibitem[Piersanti and Shen(2020)]{PS}
P.~Piersanti and X.~Shen.
\newblock Numerical methods for static shallow shells lying over an obstacle.
\newblock \emph{{N}umer. {A}lgorithms}, pages 623--652, 2020.

\bibitem[Piersanti and Temam(2023)]{PT2023}
P.~Piersanti and R.~Temam.
\newblock On the dynamics of grounded shallow ice sheets. modelling and
  analysis.
\newblock \emph{Adv. Nonlinear Anal.}, 12\penalty0 (1):\penalty0 40 pp., 2023.

\bibitem[Piersanti et~al.(2022{\natexlab{a}})Piersanti, White, Dragnea, and
  Temam]{PWDT2D}
P.~Piersanti, K.~White, B.~Dragnea, and R.~Temam.
\newblock Modelling virus contact mechanics under atomic force imaging
  conditions.
\newblock \emph{Appl. Anal.}, 101\penalty0 (11):\penalty0 3947--3957,
  2022{\natexlab{a}}.

\bibitem[Piersanti et~al.(2022{\natexlab{b}})Piersanti, White, Dragnea, and
  Temam]{PWDT3D}
P.~Piersanti, K.~White, B.~Dragnea, and R.~Temam.
\newblock A three-dimensional discrete model for approximating the deformation
  of a viral capsid subjected to lying over a flat surface.
\newblock \emph{Anal. Appl.}, 20\penalty0 (6):\penalty0 1159--1191,
  2022{\natexlab{b}}.

\bibitem[Piersanti et~al.(2021)Piersanti, Africa, Fedele, Vergara, Ded\`e,
  Corno, and Quarteroni]{Quarteroni2021-3}
R.~Piersanti, P.~C. Africa, M.~Fedele, C.~Vergara, L.~Ded\`e, A.~F. Corno, and
  A.~Quarteroni.
\newblock Modeling cardiac muscle fibers in ventricular and atrial
  electrophysiology simulations.
\newblock \emph{Comput. Methods Appl. Mech. Engrg.}, 373:\penalty0 113468, 33,
  2021.

\bibitem[Regazzoni et~al.(2021)Regazzoni, Ded\`{e}, and
  Quarteroni]{Quarteroni2021-2}
F.~Regazzoni, L.~Ded\`{e}, and A.~Quarteroni.
\newblock Active force generation in cardiac muscle cells: mathematical
  modeling and numerical simulation of the actin-myosin interaction.
\newblock \emph{Vietnam J. Math.}, 49\penalty0 (1):\penalty0 87--118, 2021.

\bibitem[Rodr\'{\i}guez-Ar\'{o}s(2018)]{Rodri2018}
A.~Rodr\'{\i}guez-Ar\'{o}s.
\newblock Mathematical justification of the obstacle problem for elastic
  elliptic membrane shells.
\newblock \emph{{A}pplicable {A}nal.}, 97:\penalty0 1261--1280, 2018.

\bibitem[Scholz(1984)]{Scholz1984}
R.~Scholz.
\newblock Numerical solution of the obstacle problem by the penalty method.
\newblock \emph{Computing}, 32\penalty0 (4):\penalty0 297--306, 1984.

\bibitem[Stampacchia(1966)]{Stampacchia1965}
G.~Stampacchia.
\newblock \emph{\`{E}quations elliptiques du second ordre \`{a} coefficients
  discontinus}, volume 1965 of \emph{S\'{e}minaire de Math\'{e}matiques
  Sup\'{e}rieures, No. 16 (\'{E}t\'{e}}.
\newblock Les Presses de l'Universit\'{e} de Montr\'{e}al, Montreal, Que.,
  1966.

\bibitem[Sun and Yuan(2006)]{SunYuan2006}
W.~Sun and Y.-X. Yuan.
\newblock \emph{Optimization theory and methods}, volume~1 of \emph{Springer
  Optimization and Its Applications}.
\newblock Springer, New York, 2006.
\newblock Nonlinear programming.

\bibitem[Zingaro et~al.(2021)Zingaro, Ded\`{e}, Menghini, and
  Quarteroni]{Quarteroni2021-1}
A.~Zingaro, L.~Ded\`{e}, F.~Menghini, and A.~Quarteroni.
\newblock Hemodynamics of the heart's left atrium based on a {V}ariational
  {M}ultiscale-{LES} numerical method.
\newblock \emph{Eur. J. Mech. B Fluids}, 89:\penalty0 380--400, 2021.

\end{thebibliography}

\end{document}